\documentclass[onefignum,onetabnum]{siamart220329}

\usepackage{amsmath, amsfonts}
\numberwithin{equation}{section}
\usepackage{bbm}
\usepackage[headings]{fullpage}
\usepackage{graphicx}
\usepackage{times}
\usepackage{tikz}
\usepackage{color}
\usepackage{mathrsfs}

\ifpdf
\hypersetup{
   colorlinks=true,%
            breaklinks=true,%
            linkcolor=blue,
            urlcolor=blue,%
            citecolor=blue,%
            pdftitle={Sensitivity of ODE Solutions and Quantities of Interest with Respect to Component Functions in the Dynamics}, %
            pdfauthor={J. R. Cangelosi and M. Heinkenschloss},%
            bookmarksopen=false,%
}
\fi

\newcommand{\aall}{\mbox{ a.a.} \;}
\DeclareMathOperator*{\esssup}{ess\,sup}
\newcommand {\half} {\mbox{$\frac{1}{2}$}}  
\newcommand {\real} {\mathbb{R}}
\newcommand{\bff}{\ensuremath{\mathbf{f}}} %
\newcommand{\bg}{\ensuremath{\mathbf{g}}} %
\newcommand{\bh}{\ensuremath{\mathbf{h}}} %
\newcommand{\bm}{\ensuremath{\mathbf{m}}} %
\newcommand{\bn}{\ensuremath{\mathbf{n}}} %
\newcommand{\bp}{\ensuremath{\mathbf{p}}} %
\newcommand{\br}{\ensuremath{\mathbf{r}}} %
\newcommand{\bu}{\ensuremath{\mathbf{u}}} %
\newcommand{\bv}{\ensuremath{\mathbf{v}}} %
\newcommand{\bx}{\ensuremath{\mathbf{x}}} %
\newcommand{\balpha}{\ensuremath{\mbox{\boldmath $\alpha$}}}
\newcommand {\bdelta} {\boldsymbol{\delta}}
\newcommand {\bepsilon} {\mbox{\boldmath $\epsilon$}}
\newcommand {\bgamma} {\mbox{\boldmath $\gamma$}}
\newcommand {\blambda} {\mbox{\boldmath $\lambda$}}
\newcommand {\bpsi} {\boldsymbol{\psi}}
\newcommand{\BA}{\ensuremath{\mathbf{A}} } %
\newcommand{\BB}{\ensuremath{\mathbf{B}} } %
\newcommand{\BE}{\ensuremath{\mathbf{E}} } %
\newcommand{\BF}{\ensuremath{\mathbf{F}} } %
\newcommand{\BI}{\ensuremath{\mathbf{I}} } %
\newcommand{\BL}{\ensuremath{\mathbf{L}} } %
\newcommand{\BQ}{\ensuremath{\mathbf{Q}} } %
\newcommand{\BV}{\ensuremath{\mathbf{V}} } %
\newcommand {\cB} {\mathcal{B}}
\newcommand {\cG} {\mathcal{G}}
\newcommand {\cL} {\mathcal{L}}
\newcommand {\cN} {\mathcal{N}}
\newcommand {\cX} {\mathcal{X}}
\newcommand {\cZ} {\mathcal{Z}}

\newcommand{\WW}{\big( W^{1, \infty}(I) \big)^{n_x}}
\newcommand{\LL}{\big(L^\infty(I) \big)^{n_x}}

\newtheorem{assumption}[theorem]{Assumption}
\newtheorem{remark}[theorem]{Remark}

\headers{Sensitivity of ODE Solutions}{J. R. Cangelosi and M. Heinkenschloss}

\title{Sensitivity of ODE Solutions and Quantities of Interest with Respect to Component Functions in the Dynamics
        \thanks{To appear in SIAM Journal on Numerical Analysis
       \funding{This research was supported in part by AFOSR Grant FA9550-22-1-0004 and NSF Grant DMS-2231482.}
       }}
         
\author{Jonathan R. Cangelosi
             \thanks{Department of Computational Applied Mathematics and Operations Research,
                      MS-134, Rice University, 6100 Main Street,
                     Houston, TX 77005-1892. 
                     E-mail: jrc20@rice.edu}
             \and
             Matthias Heinkenschloss
             \thanks{Department of Computational Applied Mathematics and Operations Research,
                 MS-134, Rice University, 6100 Main Street,
                 Houston, TX 77005-1892, and the Ken Kennedy Institute, Rice University.
                 E-mail: heinken@rice.edu}
        }

\ifpdf
\hypersetup{
   colorlinks=true,%
            breaklinks=true,%
            linkcolor=blue,
            urlcolor=blue,%
            citecolor=blue,%
            pdftitle={Sensitivity of ODE Solutions and Quantities of Interest with Respect to Component Functions in the Dynamics}, %
            pdfauthor={J. R. Cangelosi and M. Heinkenschloss},%
            bookmarksopen=false,%
}
\fi

\begin{document}

\maketitle

\begin{abstract}
This work analyzes the sensitivities of the solution of a system of ordinary differential equations (ODEs) and a corresponding quantity of interest (QoI) to perturbations in a state-dependent component function that appears in the governing ODEs. This extends existing ODE sensitivity results, which consider the sensitivity of the ODE solution with respect to state-independent parameters. It is shown that with Carath\'eodory-type assumptions on the ODEs, the Implicit Function Theorem can be applied to establish continuous Fr\'echet differentiability of the ODE solution with respect to the component function. These sensitivities are used to develop new estimates for the change in the ODE solution or QoI when the component function is perturbed. In applications, this new sensitivity-based bound on the ODE solution or QoI error is often much tighter than classical Gronwall-type error bounds. The sensitivity-based error bounds are applied to a trajectory simulation for a hypersonic vehicle.
\end{abstract}



\section{Introduction}   \label{sec:intro}
Many applications are modeled by systems of ordinary differential equations (ODEs)
in which some solution-dependent component functions are not exactly known.
In such cases, it is crucial to determine the sensitivity of the ODE solution or of a
quantity of interest (QoI) with respect to perturbations in the component functions. If bounds for the errors in the component functions are known, this sensitivity information can be used to estimate the error 
in the ODE solution or the QoI.
For example, the trajectory of an aircraft can be  modeled by a system
of ODEs including lift and drag coefficients, which are functions that themselves depend on the trajectory of the
aircraft. Often only values of lift and drag coefficients at some points are available, e.g., from experiments or
computationally expensive CFD simulations,  and approximate  lift and drag coefficients
are obtained from interpolation or regression for numerical solution of the ODEs.
See, e.g., \cite[Sec.~6.2]{JTBetts_2010a} or
\cite{JRCangelosi_MHeinkenschloss_JTNeedels_JJAlonso_2024a}.
In this case, one wants to estimate the error between the solution of the ODE system with the true component function
and the computable solution of the ODE system with the approximate component function.

In this paper, we first establish the Fr\'echet differentiability (in suitable function spaces) 
of the ODE solution with respect to these component functions.
This result is then used to provide a new sensitivity-based estimate for the error between ODE solutions computed with the 
approximate and true component functions
and a corresponding sensitivity-based estimate for the error in a QoI depending on the ODE solution.
These error estimates are crucial to determine whether the given approximate component function
is of sufficient quality. If it is not, then the error estimate could be used to determine in which regions of the solution 
space the approximate component function needs to be improved.
In applications, our new sensitivity-based error estimates can produce superior estimates compared to classical
ODE perturbation estimates, which depend exponentially on the logarithmic Lipschitz constant of the ODE system
and on the length of the time interval considered.

The problem under consideration is given as follows (the detailed function space setting will be specified in 
\cref{sec:problem_formulation}).
Given the interval $I := (t_0, t_f)$, initial data $x_0 \in \real^{n_x}$, and functions
\[
	\bg :  I \times \real^{n_x}  \rightarrow \real^{n_g}, \quad \mbox{ and } \quad
	\bff : I \times  \real^{n_x} \times \real^{n_g} \rightarrow \real^{n_x},
\]
we are interested in the dependence of the solution
$\bx : \overline{I} \rightarrow \real^{n_x}$ of the initial value problem (IVP)
\begin{equation} \label{eq:IVP}
\begin{aligned}
	\bx'(t) &= \bff\Big( t, \bx(t), \bg\big(t, \bx(t) \big) \Big), & \mbox{almost all (a.a.) } t \in I, \\
	\bx(t_0) &= x_0,
\end{aligned}
\end{equation}
on the component function $\bg$. The solution $\bx$ of \eqref{eq:IVP} is also referred to as the state.
The function $\bff$ represents the dynamics of the system, which depend on a component function $\bg$.
We often use $\bx(\cdot \, ; \bg)$ to denote the solution of \eqref{eq:IVP} to emphasize that it is computed
with the component function $\bg$.
We assume that instead of the true function $\bg_*$ one only has an approximation $\widehat{\bg}$ available.
Thus, instead of the desired $\bx(\cdot \, ; \bg_*)$ one can only compute $\bx(\cdot \, ; \widehat{\bg})$.

In \cref{sec:sensitivity} we will specify the function space setting for \eqref{eq:IVP}
and establish continuous Fr\'echet differentiability of $\bg \mapsto \bx(\cdot \, ; \bg)$.
Sensitivity analyses of the solution of an ODE
\begin{equation*}
	\bx'(t) = \bff\big( t, \bx(t), \bp \big), \quad \aall t \in I, \qquad   \bx(t_0) = x_0,
\end{equation*}
with respect to parameters $\bp \in \real^{n_p}$ are standard;
see, e.g., \cite[Sec.~9]{HAmann_1990}, \cite[Sec.~I.14]{EHairer_SONorsett_GWanner_1993a}.
However, in \eqref{eq:IVP} the third argument of $\bff$, in contrast to the parametric setting, features a \emph{state-dependent} model function $\bg$, introducing coupling between the model and the resulting trajectory. We will use the Implicit Function Theorem to establish continuous
Fr\'echet differentiability of the map $\bg \mapsto \bx(\cdot \, ; \bg)$. However, the setup is
different from that of proving continuous Fr\'echet differentiability of the ODE solution with respect to parameters $\bp \in \real^{n_p}$ 
due to the coupling between the model and the ODE solution, which is not present in the parametric setting.
At the heart of our analysis is the continuous Fr\'echet differentiability of a (somewhat nonstandard) superposition or Nemytskii operator.

In \cref{sec:ODE_error}, we will use the sensitivity results of \cref{sec:sensitivity} to establish a new approximate upper bound for the error
 $\| \bx(\cdot \, ; \widehat{\bg}) -  \bx(\cdot \, ; \bg_*) \|$ (in some suitable norm or seminorm) given a pointwise bound for $| \widehat{\bg} -  \bg_* |$ (understood componentwise).
Classical ODE perturbation results such as those in 
\cite[Sec.~I.10]{EHairer_SONorsett_GWanner_1993a}, \cite{GSoderlind_2006a} provide 
a bound for the error $\bx(t; \bg_*) -  \bx(t; \widehat{\bg})$, $t \in \overline{I} = [t_0, t_f]$, which we will review 
in \cref{sec:ODE_perturbation}. However, this bound can be very pessimistic, especially
when $t - t_0$ becomes larger. In our example shown in \cref{sec:numerics_flap}, this
bound becomes practically useless even for small $t - t_0$. This has motivated the sensitivity-based bound
we will develop in \cref{sec:ODE_sensitivity_sol}. The idea is to approximate 
\[
        \bx(\cdot \, ; \widehat{\bg}) - \bx(\cdot \, ;  \bg_*) \approx \bx_\bg(\widehat{\bg})  (\widehat{\bg} - \bg_*),
\] 
where $\bx_\bg(\widehat{\bg})$ denotes the Fr\'echet derivative of $\bg \mapsto \bx(\cdot \, ; \bg)$ at $\bg = \widehat{\bg}$,
then use a bound of the error in the component function along the
computed trajectory $\widehat{\bx} = \bx(\cdot \, ; \widehat{\bg})$ to obtain an upper bound for the error estimate $\| \bx_\bg(\widehat{\bg})  (\widehat{\bg} - \bg_*) \|$ (in some appropriate norm or seminorm), which is an approximate upper bound of $\| \bx(\cdot \, ; \widehat{\bg}) - \bx(\cdot \, ; \bg_*) \|$ when $\widehat{\bg} - \bg_*$ is relatively small.
Specifically, if  $$\big|  \widehat{\bg}\big( t, \widehat{\bx}(t) \big) - \bg_*\big( t, \widehat{\bx}(t) \big) \big| \leq  \bepsilon\big(t, \widehat{\bx}(t)\big), \qquad \aall t \in I,$$ where the absolute value and the inequality are applied componentwise  and $\bepsilon: I \times  \real^{n_x} \rightarrow \real^{n_g}$ is an error bound for the component function, we formulate and solve a linear quadratic optimal control
problem to obtain an approximate upper bound for $\| \bx_\bg(\widehat{\bg})  (\widehat{\bg} - \bg_*)\|$ using $\bepsilon\big(t, \widehat{\bx}(t)\big)$, $t \in I$.
In our examples shown in \cref{sec:numerics_flap}, this new bound provides excellent estimates for the error
$\| \bx(\cdot \, ; \widehat{\bg}) - \bx(\cdot \, ; \bg_*) \|$ provided $\bepsilon\big(t, \widehat{\bx}(t)\big)$ is a relatively tight upper bound 
for $\big|  \widehat{\bg}\big( t, \widehat{\bx}(t) \big) - \bg_*\big( t, \widehat{\bx}(t) \big) \big|$ and $\widehat{\bg}$ is relatively close to $\bg_*$.

Our new (approximate) bound for the ODE solution error $\| \bx(\cdot \, ; \widehat{\bg}) - \bx(\cdot \, ; \bg_*) \|$
comes at the cost of solving a linear quadratic optimal control problem, which has some theoretical shortcomings that will be discussed in \cref{sec:ODE_sensitivity_sol}.
However, in applications where the evaluation of the true $\bg_*$ is computationally expensive
but the evaluation of an approximate surrogate  $\widehat{\bg}$ is not, the extra expense of solving the  
linear quadratic optimal control problem is less expensive than working with the true $\bg_*$ and
yields good estimates for the ODE solution error in practice.
We are utilizing this in other work to adapt surrogate models for $\bg_*$ from evaluations of  $\bg_*$ at
points $(t, x) \in I \times \real^{n_x}$ along the current trajectory.

The linear quadratic optimal control problem can be avoided if only an estimate for the error in a
quantity of interest $\widetilde{q}(\bg) :=  q \big(\bx(\cdot \, ; \bg), \bg \big)$ is desired. Instead of sensitivities,
a so-called adjoint equation can be used to express the Fr\'echet derivative of $\bg \mapsto \widetilde{q}(\bg)$.
As we show in \cref{sec:ODE_sensitivity_qoi},  our approximate upper bound for the error $| \widetilde{q}(\widehat{\bg}) - \widetilde{q}(\bg_*) |$ can be obtained by solving a linear program with a simple analytical solution at the expense of one linear adjoint ODE solve. This approach avoids the theoretical issues associated with the linear quadratic optimal control problem and yields a much more easily computable error bound.

{\bf Notation.} 
We will use $\| \cdot \|$ to denote a vector norm on $\real^m$ (where $m$ depends on the context) or an 
induced matrix norm.
By $\mathcal{B}_R(0) \subset \real^m$ we denote the closed ball in $\real^m$ around zero with radius $R>0$.
When infinite-dimensional normed linear spaces are considered, the norm will always be specified explicitly 
using subscripts.

Given an interval $I = (t_0, t_f)$,  $\big( L^\infty(I) \big)^m$ denotes the Lebesgue space of essentially bounded functions on $I$ with
values in $\real^m$, and $\big( W^{1, \infty}(I) \big)^m$ denotes the Sobolev space of functions on $I$ with 
values in $\real^m$ that are weakly differentiable on $I$ and have essentially bounded derivative.

We typically use bold font for vector- or matrix-valued functions and regular font for scalars, vectors, matrices, and scalar-valued functions (except states, which will be boldface). For example, the function $\bx: I \rightarrow \real^{n_x}$ has values $\bx(t) \in \real^{n_x}$, and $x \in \real^{n_x}$ denotes a vector. This distinction will be useful when
studying compositions of functions. Also, when using subscripts for derivatives, regular subscripts will be used to denote partial derivatives with respect to a vector, while boldface subscripts will be used to denote Fr\'echet derivatives with respect to a function.

\section{Sensitivity Analysis} \label{sec:sensitivity}
In this section we first specify the function space setting for  \cref{eq:IVP}, and then we
establish sensitivity results for the map  $\bg \mapsto \bx(\cdot \, ; \bg)$ or for a quantity of interest
that depends on $\bg \mapsto \bx(\cdot \, ; \bg)$.

\subsection{Problem Setting}   \label{sec:problem_formulation}
We seek solutions of the IVP \cref{eq:IVP} in the sense of Carath\'eodory, i.e., the right-hand side $(t, x) \mapsto \bff\big( t, x, \bg(t, x) \big)$ is assumed to be measurable in $t$ and continuous in $x$.
The reason for this choice is that one often wants to consider an IVP that depends on a non-smooth input $\bu: I \rightarrow \real^{n_u}$, which may be written as
\begin{equation} \label{eq:IVP_control}
	\bx'(t) = \widetilde{\bff}\Big( t, \bx(t), \bu(t), \widetilde{\bg}\big(t, \bx(t), \bu(t) \big) \Big), \; \aall t \in I, \qquad
	\bx(t_0) = x_0.
\end{equation}
For example, the state equations in many optimal control problems are of the form \cref{eq:IVP_control} with
controls $\bu \in \big(L^\infty(I)\big)^{n_u}$; see, e.g., \cite{MGerdts_2012a}, \cite{EPolak_1997}.
To make our setting applicable with 
$(t, x) \mapsto   \bff\big( t, x, \bg(t, x ) \big) := \widetilde{\bff}\big( t, x, \bu(t), \widetilde{\bg}\big(t, x, \bu(t) \big) \big)$
for controls $\bu \in \big(L^\infty(I)\big)^{n_u}$,
we must allow functions $\bff$ and $\bg$ that are not continuous in $t$.

Existence and uniqueness of solutions to the IVP \cref{eq:IVP} can be proven, e.g., by adapting the results
in  \cite[Sec.~1]{AFFilippov_1988a} or in \cite[Sec.~5.6]{EPolak_1997}. We use \cite{AFFilippov_1988a}.

The following assumptions are used to ensure existence and uniqueness of solutions to the IVP \cref{eq:IVP}.
The assumptions can be weakened if one only needs existence of a solution locally around $t_0$; see \cite[Sec.~1]{AFFilippov_1988a}.
In the following integrability is understood in the Lebesgue sense.
\begin{assumption} \label{as:IVP_unique_solution}
	Let the following conditions hold for \cref{eq:IVP}:
	
   \begin{itemize}
   \item[(i)]   The function $\bff( t, x, g )$ is continuous in $x \in \real^{n_x}$ and $g  \in \real^{n_g}$ for almost all $t \in I$,
                   it is measurable in $t$ for each $x \in \real^{n_x}$ and $g  \in \real^{n_g}$, and there
                   exists a square integrable function $m_f$ such that $$\| \bff( t, x, g ) \| \le m_f(t) (1 + \| g \|), \qquad \aall t \in I \mbox{ and all } x \in \real^{n_x}, g \in \real^{n_g}. $$ 
   \item[(ii)]   There exists a square integrable function $l_f$ such that 
   \begin{align*}
   \| \bff( t, x_1, g_1 ) -  \bff( t, x_2, g_2 )  \| &\le l_f(t) ( \| x_1 - x_2  \| + \| g_1 - g_2  \| ), \\ &  \qquad \aall t \in I \mbox{ and all } x_1, x_2 \in \real^{n_x}, g_1, g_2 \in \real^{n_g}.
   \end{align*}
    \item[(iii)]   The function $\bg( t, x )$ is continuous in $x \in \real^{n_x}$ for almost all $t \in I$,
                   it is measurable in $t$ for each $x \in \real^{n_x}$, and there
                   exists a square integrable function $m_g$ such that $$\| \bg( t, x ) \| \le m_g(t), \qquad \aall t \in I \mbox{ and all } x \in \real^{n_x}.$$  
   \item[(iv)]   There exists a square integrable function $l_g$ such that $$\| \bg( t, x_1 ) -  \bg( t, x_2 )  \| \le l_g(t)  \| x_1 - x_2  \|, \qquad \aall t \in I \mbox{ and all } x_1, x_2 \in \real^{n_x}.$$
    \item[(v)]  The functions $m_f$, $m_g$ in (i) and (iii) satisfy $m_f, m_g \in L^\infty(I)$.
\end{itemize}
\end{assumption}

\begin{theorem} \label{th:IVP_unique_solution}
    If \cref{as:IVP_unique_solution}~(i), (iii) \ are satisfied, then 
     the IVP \cref{eq:IVP} has a solution on the entire interval $I$. 
     If \cref{as:IVP_unique_solution}~(i)-(iv) \  are satisfied, then 
     the IVP \cref{eq:IVP} has a unique solution on $I$. 
     If \cref{as:IVP_unique_solution}~(i)-(v) \  are satisfied, then 
     the IVP \cref{eq:IVP} has a unique solution $\bx \in \WW$.
\end{theorem}
\begin{proof} 
     If \cref{as:IVP_unique_solution} (i), (iii) \ are satisfied,
     the composition $\bff\big( t, x, \bg(t, x) \big)$  is continuous in $x \in \real^{n_x}$ for almost all $t \in I$,
     is measurable in $t$ for each $x \in \real^{n_x}$, and satisfies
     \begin{equation} \label{eq:essentially-bounded}
            \big\| \bff\big( t, x, \bg(t, x) \big) \big\| \le m_f(t) (1 + \|  \bg(t, x) \|) \le m_f(t) \big(1 + m_g(t) \big).
     \end{equation}
     Since $m_f$ and $m_g$ are square integrable, both $m_f$ and $m_f m_g$ are integrable, so existence of a solution follows from Theorem~1 in \cite[p.~4]{AFFilippov_1988a}.
  
     If \cref{as:IVP_unique_solution} (ii), (iv) \ are satisfied, the 
     composition satisfies 
     \begin{align*}\big\| \bff\big( t, x_1, \bg(t, x_1) \big) -  \bff\big( t, x_2, \bg(t, x_2) \big)  \big\| &\le l_f(t) \big(1+l_g(t)\big)  \| x_1 - x_2  \|, \\ & \qquad \aall t \in I \mbox{ and all } x_1, x_2 \in \real^{n_x}, g_1, g_2 \in \real^{n_g},
     \end{align*}
     and $l_f(t) \big(1+l_g(t)\big)$ is likewise integrable.
     Uniqueness of the solution follows from Theorem~2 in \cite[p.~5]{AFFilippov_1988a}.
     
     If $m_f, m_g \in L^\infty(I)$, then $m_f m_g \in L^\infty(I)$, and so it follows from  \cref{eq:essentially-bounded}
     that $\bx' \in \LL$. 
\end{proof}

\subsection{Fr\'echet Differentiability of the Dynamics} \label{sec:frechet-differentiability-rhs}
To establish sensitivity of the solution of the IVP \cref{eq:IVP} with respect to the function $\bg$,
we consider the IVP \cref{eq:IVP} as an operator equation in the functions $\bx$ and $\bg$.
The main ingredient of this operator equation is the right-hand side operator. To define this map, we
first need to specify the function space for $\bg$.

The set of component functions $\bg$ is given by the Banach space
\begin{subequations} \label{eq:g-space}
\begin{align}
	\big(\cG^k(I)\big)^{n_g} := \big\{ \bg: I \times \real^{n_x} \rightarrow \real^{n_g} \; : \; 
	                   & \bg( t, x ) \mbox{ is $k$-times continuously partially }   \nonumber \\
	                   & \mbox{differentiable  with respect to } x \in \real^{n_x}  \mbox{ for } \aall t \in I,   \nonumber \\
                           &   \mbox{is measurable in $t$ for each $x \in \real^{n_x}$, and }  
                             \| \bg \|_{(\cG^k(I))^{n_g}}  < \infty \big\},
\end{align}
where
\begin{equation} \label{eq:sensitivity:g-norm}
	\| \bg \|_{(\cG^k(I))^{n_g}} := \sum_{n=0}^k \; \esssup_{t \in I} \sup_{x \in \real^{n_x}} \left\| \frac{\partial^n}{\partial x^n} \bg(t, x) \right\|.
\end{equation}
\end{subequations}
We are primarily interested in the cases $k = 1$ and $k = 2$. 
In these cases, we use
$\bg_x(t, x) \in \real^{n_g \times n_x}$ to denote the partial Jacobian of $\bg$ with respect to $x$ at $t \in I$, $x \in \real^{n_x}$,
and $\bg_{xx}(t, x) \in \real^{n_g \times n_x \times n_x}$ to denote the partial  Hessian of $\bg$ with respect to $x$ at $t \in I$, $x \in \real^{n_x}$.

Because of \cref{eq:sensitivity:g-norm}, functions $\bg \in \big(\cG^1(I)\big)^{n_g}$ always satisfy \cref{as:IVP_unique_solution}(iii) with
\[
	m_g(t) \equiv \| \bg \|_{(\cG^0(I))^{n_g}}.
\]
Furthermore, this $m_g$ belongs to $L^\infty(I)$, so that part of \cref{as:IVP_unique_solution}(v) is always satisfied when 
taking $\bg \in \big(\cG^1(I)\big)^{n_g}$, and thus in such cases the square integrability of $m_f$ and $l_f$ in conditions (i), (ii) of \cref{as:IVP_unique_solution} may be weakened to integrability.
Additionally, due to the boundedness of $\bg_x$, functions $\bg \in \big(\cG^1(I)\big)^{n_g}$ always satisfy \cref{as:IVP_unique_solution}(iv) with
\[
	l_g(t) \equiv \| \bg \|_{(\cG^1(I))^{n_g}}.
\]
Note also that for $\ell > k$ the space  $\big(\cG^\ell(I)\big)^{n_g}$  is continuously embedded into $\big(\cG^k(I)\big)^{n_g}$.

The right-hand side operator is a function
\begin{subequations}   \label{eq:F-superposition}
\begin{align}
     \BF_k:   \LL \times \big(\cG^k(I)\big)^{n_g} \rightarrow  \LL
\end{align}
defined by
\begin{align}
	\BF_k(\bx, \bg)(t) := \bff\Big(t, \bx(t), \bg\big(t, \bx(t)\big)\Big).
\end{align}
\end{subequations}
The operator \cref{eq:F-superposition} is a superposition or Nemytskii operator; see, e.g., 
\cite{JAppell_PPZabrejko_1990a},  \cite[Sec.~4.3.2]{FTroeltzsch_2010a}.
However, in contrast to standard superposition or Nemytskii operators,
\cref{eq:F-superposition} depends on $\bx$ directly through the second argument of $\bff$ and
also through the composition $\bg(t, \bx)$. 
Under \cref{as:IVP_unique_solution}, the range of this operator is contained in $\LL$ due to \cref{eq:essentially-bounded}.
We use the subscript $k$ in \cref{eq:F-superposition} to emphasize the change in the domain of the operator.

We are interested in the differentiability properties of \cref{eq:F-superposition}.
The first result concerns the Fr\'echet differentiability of $\BF_1$ at a point $(\bx, \bg) \in \LL \times \big(\cG^1(I)\big)^{n_g}$, which requires
some additional smoothness assumptions on $\bff$. These assumptions are consistent with those made in \cite[Sec.~4.3.2]{FTroeltzsch_2010a} 
for the Fr\'echet differentiability
of (standard) Nemytskii operators in $L^\infty$ spaces.

\begin{assumption} \label{as:assumption1}
   Let the following conditions hold, in addition to those of \cref{as:IVP_unique_solution}:
   \begin{itemize}   
   \item[(i)]   The function $\bff : I \times \real^{n_x} \times \real^{n_g} \rightarrow \real^{n_x}$ 
                   is continuously partially differentiable with respect to $x \in \real^{n_x}$ and $g  \in \real^{n_g}$ for almost all $t \in I$ and
                   is measurable in $t$ for each $x \in \real^{n_x}$ and $g  \in \real^{n_g}$.
                   
   \item[(ii)]   There exists $K > 0$ such that $ \| \bff_x( t, 0, 0 )  \| \le K$ and 
                    $ \| \bff_g( t, 0, 0 )  \| \le K$  for almost all $t \in I$.
                   
   \item[(iii)]    For all $R > 0$ there exists $L(R) > 0$ such that
    \begin{align*}   
    	\| \bff_x(t, x_1, g_1) - \bff_x(t, x_2, g_2) \| &+ \| \bff_g(t, x_1, g_1) - \bff_g(t, x_2, g_2) \|
          \leq L(R) ( \| x_1 - x_2 \| + \| g_1 - g_2 \| ), \\
          &\qquad\qquad\qquad\qquad  \aall  t \in I  \mbox{ and all } x_1, x_2 \in \cB_R(0), \; g_1, g_2 \in \cB_R(0).
    \end{align*}
    \end{itemize}
\end{assumption}

The following theorem establishes Fr\'echet differentiability of $\BF_1$ at a point $(\bx, \bg) \in \LL \times \big(\cG^1(I)\big)^{n_g}$ 
provided the Jacobian of $\bg \in \big(\cG^1(I)\big)^{n_g}$ satisfies a local Lipschitz condition.

\begin{theorem} \label{thm:F-Frechet_point}
    If \cref{as:assumption1} holds and if 
    $\bg \in \big(\cG^1(I)\big)^{n_g}$ has the property that for all $R > 0$ there exists an $L(R)$ such that
    \begin{align}  \label{eq:Lipschitz_g_x_local}
    		\| \bg_x(t, x_1) - \bg_x(t, x_2) \| &\leq L(R) \| x_1 - x_2 \|,  
		& \aall t \in I \;   \mbox{ and all }  x_1, x_2 \in \cB_R(0), 
    \end{align}
    then $\BF_1$ defined in \cref{eq:F-superposition} is Fr\'echet differentiable at $(\bx, \bg) \in \LL \times \big(\cG^1(I)\big)^{n_g}$, and its derivative is given by
    \begin{align}  \label{eq:F-Frechet_deriv}
      	[\BF_1'(\bx, \bg) (\delta \bx, \delta \bg)](t) 
	&= \Big[ \bff_x\Big( t, \bx(t), \bg\big(t, \bx(t) \big) \Big) 
	              + \bff_g\Big( t, \bx(t), \bg\big(t, \bx(t) \big) \Big) \bg_x\big(t, \bx(t) \big) \Big] \delta \bx(t)    \nonumber \\ 
	 &\qquad + \bff_g\Big( t, \bx(t), \bg\big( t,  \bx(t) \big) \Big) \delta \bg\big(t,  \bx(t) \big).
    \end{align}
\end{theorem}
The proof of \cref{thm:F-Frechet_point} is given in \cref{sec:sensitivity_proof}.
    
The assumptions of \cref{thm:F-Frechet_point} are not strong enough to conclude \emph{continuous} Fr\'echet differentiability of the operator $\BF_1$ 
at the point $(\bx, \bg) \in \LL \times \big(\cG^1(I)\big)^{n_g}$, 
as any neighborhood of $\bg \in \big(\cG^1(I)\big)^{n_g}$ contains functions whose derivatives are not locally Lipschitz, even if $\bg$ satisfies \cref{eq:Lipschitz_g_x_local}. 
Therefore, we consider the operator $\BF_2$ instead, which is an operator from
$ \LL \times \big(\cG^2(I)\big)^{n_g}$ to $ \LL$. In so doing, we have restricted the component functions from $\big(\cG^1(I)\big)^{n_g}$ to the smaller space $\big(\cG^2(I)\big)^{n_g}$. As the following theorem shows, $\BF_2$ is in fact continuously Fr\'echet differentiable.

\begin{theorem} \label{thm:F-Frechet_global}
    If \cref{as:assumption1}  holds, the operator  $\BF_2 :  \LL \times \big(\cG^2(I)\big)^{n_g} \rightarrow  \LL$
    given by \cref{eq:F-superposition} is continuously Fr\'echet differentiable, and its derivative is given by
    \cref{eq:F-Frechet_deriv}. Moreover, the derivative is locally Lipschitz continuous.
\end{theorem}
The proof of \cref{thm:F-Frechet_global} is given in \cref{sec:sensitivity_proof}.

\subsection{Fr\'echet Differentiability of the ODE Solution} \label{sec:frechet-differentiability-solution}
Now, we revisit the IVP \cref{eq:IVP}. To establish the continuous Fr\'echet differentiability of the solution mapping
$\big(\cG^2(I)\big)^{n_g} \ni \bg \mapsto \bx(\cdot \, ; \bg) \in \WW$ we consider 
the operator
\begin{subequations} \label{eq:psi}
\begin{align}
	\bpsi :  \WW\times \big(\cG^2(I)\big)^{n_g}  & \rightarrow  \LL \times \real^{n_x}
\end{align}
defined by
\begin{align}
	\bpsi(\bx, \bg) 
	            = \begin{pmatrix} \BF_2(\bx, \bg )  -  \bx' \\  \bx(t_0) - \bx_0 \end{pmatrix}. 
\end{align}
\end{subequations}
By construction, the solution $\bx =  \bx(\cdot \, ; \bg)$ of \cref{eq:IVP} satisfies $\bpsi(\bx, \bg) = 0$.

As the following corollary of \cref{thm:F-Frechet_global} shows, continuous Fr\'echet differentiability of \cref{eq:F-superposition} implies continuous Fr\'echet differentiability of \cref{eq:psi}.

\begin{corollary} \label{cor:psi-Frechet_derivative}
If \cref{as:assumption1}  holds, then the map  \cref{eq:psi}
         is continuously Fr\'echet differentiable and the derivative is given by
	\begin{equation*}
	\begin{aligned}
		&\bpsi'(\bx, \bg)(\delta \bx, \delta \bg) 
		= \begin{pmatrix} \BA(\cdot) \delta \bx(\cdot) + \BB(\cdot)\delta \bg \big(\cdot, \bx(\cdot) \big) - \delta \bx'(\cdot) \\
		\delta \bx(t_0)
		 \end{pmatrix}
	\end{aligned}
	\end{equation*}
	where 
	\begin{equation} \label{eq:shorthand}
	\begin{aligned}
	  \BA(\cdot) &:= \bff_x\Big( \cdot,\bx(\cdot), \bg\big(\cdot, \bx(\cdot) \big) \Big)
	                          + \bff_g\Big( \cdot,\bx(\cdot), \bg\big(\cdot, \bx(\cdot) \big) \Big) \bg_x \big(\cdot, \bx(\cdot) \big), \\
	\BB(\cdot)  &:= \bff_g\Big( \cdot,\bx(\cdot), \bg\big(\cdot, \bx(\cdot) \big) \Big).
	\end{aligned}
	\end{equation}
\end{corollary}
\begin{proof}
	The continuous Fr\'echet differentiability of 
	$$\WW\times \big(\cG^2(I)\big)^{n_g}  \ni (\bx, \bg )
	 \mapsto  \BF_2(\bx, \bg ) \in  \LL$$  is a consequence of \cref{thm:F-Frechet_global} 
	 since $W^{1, \infty}(I)$ is continuously embedded into $L^\infty(I)$. 
	 Furthermore, the mappings 
	 \begin{align*}
	 \WW \ni \bx  &\mapsto \bx' \in   \LL,
	 &\WW \ni \bx  \mapsto \bx(t_0) \in \real^{n_x}
	 \end{align*} 
	 are bounded linear operators, the latter because $W^{1, \infty}(I)$ is continuously embedded into $C(\overline{I})$. 
	 Combining these results and the Fr\'echet derivative \cref{eq:F-Frechet_deriv} imply the desired result.
\end{proof}

From \cref{cor:psi-Frechet_derivative}, we have the partial Fr\'echet derivatives
\begin{equation} \label{eq:partial-frechet-derivatives}
\begin{aligned}
	\bpsi_\bx(\bx, \bg) \delta \bx 
	&= \begin{pmatrix} \Big(  \BA(\cdot)  \delta \bx(\cdot) - \delta \bx'(\cdot) \\  \delta \bx(t_0) \end{pmatrix}, &
	\bpsi_\bg(\bx, \bg) \delta \bg 
	&= \begin{pmatrix} \BB(\cdot)  \delta \bg\big( \cdot, \bx(\cdot) \big) \\ 0 \end{pmatrix}.
\end{aligned}
\end{equation}
The following bijectivity result for $\bpsi_\bx$ allows application of the Implicit Function Theorem.

\begin{lemma} \label{lemma:psi-Frechet_bijection}
If \cref{as:assumption1}  holds and $(\bx, \bg) \in  \WW \times \big(\cG^2(I)\big)^{n_g}$,
then the partial Fr\'echet derivative $\bpsi_\bx(\bx, \bg) : \big(W^{1, \infty}(I) \big)^{n_x} \rightarrow \big(L^\infty(I) \big)^{n_x} \times \real^{n_x}$ is bijective.
\end{lemma}
\begin{proof}
	From \cref{eq:partial-frechet-derivatives} it follows that for $(\br, r_0) \in \LL  \times \real^{n_x}$ the equation
	\[
	\bpsi_\bx(\bx, \bg) \delta \bx  = \begin{pmatrix} \br \\ r_0 \end{pmatrix}
	 \]
	is equivalent to the linear initial value problem
	 \begin{align*}
	 	\delta \bx'(t) &= \BA(t) \delta \bx(t) - \br(t), \qquad \aall t \in I, \\
	 	\delta \bx(t_0) &= r_0,
	 \end{align*}
	 with $\BA \in \big(L^\infty(I) \big)^{n_x \times n_x}$ given by \cref{eq:shorthand}.
	 This linear IVP has a unique
	 solution  $\delta \bx \in \WW$.
\end{proof}

\cref{cor:psi-Frechet_derivative} and \cref{lemma:psi-Frechet_bijection} now allow 
application of the Implicit Function Theorem. For completeness, we state  the Implicit Function Theorem
next, with notation adapted to our setting. See, e.g., \cite[Thm.~2.1.14]{MGerdts_2012a},
\cite[Thm.~4.E, p.~250]{EZeidler_1995b}, \cite[Thm.~3.4.10]{SGKrantz_HRParks_2013a}.

\begin{theorem}[Implicit Function Theorem] \label{thm:Sensitivity:intro:IFT}
	Let $\cX, \cG, \cZ$ be Banach spaces, let $D \subset \cX \times \cG$ be a neighborhood of the point $(\overline{\bx}, \overline{\bg}) \in \cX \times \cG$, and let 
	$\bpsi : D \rightarrow \cZ$ 
	be an operator satisfying $\bpsi(\overline{\bx}, \overline{\bg}) = 0_\cZ$. If 
	\begin{enumerate}
		\item[(i)] $\bpsi$ is continuously Fr\'echet differentiable and
		
		\item[(ii)] the partial Fr\'echet derivative $\bpsi_\bx(\overline{\bx}, \overline{\bg})$ is bijective,
	\end{enumerate}
	 then there exist neighborhoods $\cN(\overline{\bx}) \subset \cX$, $\cN(\overline{\bg}) \subset \cG$ and a unique mapping $\bx : \cN(\overline{\bg}) \rightarrow \cN(\overline{\bx})$ that is continuously Fr\'echet differentiable and satisfies
	\[
		\bx(\overline{\bg}) = \overline{\bx} \quad \mbox{and} \quad
		\bpsi\big(\bx(\bg), \bg \big) = 0_\cZ \qquad \mbox{for all } \ \big(\bx(\bg), \bg \big) \in \cN(\overline{\bx}) \times \cN(\overline{\bg}).
	\]
	Moreover, the Fr\'echet derivative is
	\begin{align}
		\bx_\bg(\overline{\bg}) = -\bpsi_\bx(\overline{\bx}, \overline{\bg})^{-1} \, \bpsi_\bg(\overline{\bx}, \overline{\bg}). \label{eq:sensitivity:sensitivity_eqn}
	\end{align}
\end{theorem}

Finally, we can apply \cref{thm:Sensitivity:intro:IFT} to obtain a sensitivity result for \cref{eq:IVP}.
\begin{theorem} \label{thm:sensitivity_result}
        Let $\bpsi$ be the map  \cref{eq:psi}.  If \cref{as:assumption1}  holds, then for any 
	$(\overline{\bx}, \overline{\bg}) \in \WW \times \big(\cG^2(I)\big)^{n_g}$ 
	satisfying $\bpsi(\overline{\bx}, \overline{\bg}) = 0$ there exist neighborhoods 
	\[
	       \cN(\overline{\bx}) \subset  \WW, \qquad \quad 
	       \cN(\overline{\bg}) \subset \big(\cG^2(I)\big)^{n_g}
	\]
	and a unique mapping $\bx : \cN(\overline{\bg}) \rightarrow \cN(\overline{\bx})$ that is continuously Fr\'echet differentiable and satisfies
	\[
		\bx(\overline{\bg}) = \overline{\bx} \quad \mbox{and} \quad
		\bpsi\big(\bx(\bg), \bg \big) = 0 \qquad \mbox{for all } \ \big(\bx(\bg), \bg\big) \in \cN(\overline{\bx}) \times \cN(\overline{\bg}).
	\]
	Moreover, the sensitivity $\delta \bx := \bx_\bg(\overline{\bg}) \delta \bg$ is the solution of the linear initial value problem
	\begin{equation}    \label{eq:x-sensitivity}
	\begin{aligned}
		\delta \bx'(t) &= \overline{\BA}(t) \, \delta \bx(t) +\overline{\BB}(t) \, \delta \bg\big(t,  \overline{\bx}(t) \big), & \aall t \in I, \\
		\delta \bx(t_0) &= 0,
	\end{aligned}
	\end{equation}
	where $\overline{\BA}$ and $\overline{\BB}$ are given by \cref{eq:shorthand} with $\bx$, $\bg$ replaced by $\overline{\bx}$,
	$\overline{\bg}$, respectively.
\end{theorem}
\begin{proof}
	The theorem is a consequence of the Implicit Function \cref{thm:Sensitivity:intro:IFT}, whose hypotheses (i) \ and (ii) \ follow from \cref{cor:psi-Frechet_derivative} and
	\cref{lemma:psi-Frechet_bijection} respectively. The IVP \cref{eq:x-sensitivity} follows from applying the partial Fr\'echet derivatives \cref{eq:partial-frechet-derivatives} to the sensitivity equation \cref{eq:sensitivity:sensitivity_eqn}.
\end{proof}

\subsection{Fr\'echet Differentiability of a Quantity of Interest} \label{sec:frechet-differentiability-adjoints}

The Fr\'echet derivative of a quantity of interest (QoI) as a function of the ODE solution $\bx \in \WW$ and the model function $\bg \in \big(\cG^2(I)\big)^{n_g}$ can be computed using adjoints.
As before, let  $I := (t_0, t_f)$. Given functions
\begin{align*}
	\varphi :  \real^{n_x} \rightarrow \real, \qquad \quad l : I \times \real^{n_x} \times \real^{n_g} \rightarrow \real
\end{align*}
consider the QoI
\begin{subequations} \label{eq:QoI}
\begin{equation}
     q:  \WW \times \big(\cG^2(I)\big)^{n_g} \rightarrow \real
\end{equation}
given by
\begin{equation} \label{eq:sensitivity:QoI-xg}
      q(\bx, \bg) := \varphi \big(   \bx(t_f) \big) + \int_{t_0}^{t_f} l \Big(t, \bx(t), \bg \big( t, \bx(t) \big) \Big) \, dt
\end{equation}
and
\begin{equation}
     \widetilde{q}:  \big(\cG^2(I)\big)^{n_g} \rightarrow \real \quad
     \mbox{ given by } \quad  \widetilde{q}(\bg) :=  q \big(\bx(\cdot \, ; \bg), \bg \big)
\end{equation}
\end{subequations}
where $\bx(\cdot \, ; \bg) \in \WW$ is the solution of \cref{eq:IVP} given $\bg \in \big(\cG^2(I)\big)^{n_g}$.

To ensure \cref{eq:QoI} is well-defined, we assume $l$ satisfies \cref{as:assumption1} with $\bff$ replaced by $l$. For the sensitivity analysis, we use the Nemytskii operator
\begin{subequations} \label{eq:l-superposition}
\begin{equation}
	\BL_2 :  \LL \times \big(\cG^2(I)\big)^{n_g} \rightarrow L^\infty(I)
\end{equation}
given by
	\begin{equation}
		\BL_2(\bx, \bg) := l \Big(\cdot, \bx(\cdot), \bg\big( \cdot, \bx(\cdot) \big)  \Big),
	\end{equation}
	\end{subequations}
	cf. $\BF_2$ in \cref{eq:F-superposition}.

\begin{remark}
        \cref{thm:F-Frechet_global} can be applied with $\BF_2$ replaced by $\BL_2$ in \cref{eq:l-superposition} to
        show that $\BL_2$ is continuously Fr\'echet differentiable when \cref{as:assumption1} is satisfied with $\bff$ replaced by $l$, 
        and that the derivative is given by
	\begin{equation} \label{eq:l-Frechet_deriv}
	\begin{aligned}
		[\BL_2'(\bx, \bg) (\delta \bx, \delta \bg)](t) &= \Big[ \nabla_x l\Big( t, \bx(t), \bg\big(t, \bx(t) \big) \Big) 
	              + \bg_x\big(t, \bx(t) \big)^T \nabla_g l\Big( t, \bx(t), \bg\big(t, \bx(t) \big) \Big) \Big]^T \delta \bx(t) \\ 
	 &\quad + \nabla_g l\Big( t, \bx(t), \bg\big( t,  \bx(t) \big) \Big)^T \delta \bg\big(t,  \bx(t) \big).
	\end{aligned}
	\end{equation}
\end{remark}

The following result establishes continuous Fr\'echet differentiability of \cref{eq:QoI} as a consequence of the differentiability of \cref{eq:l-superposition}.
\begin{theorem} \label{thm:sensitivity_result_qoi}
	If \cref{as:assumption1} holds with $\bff$ replaced by $l$ and $\varphi$ is continuously differentiable,
	then $q$ in $\cref{eq:QoI}$ is continuously Fr\'echet differentiable, and its derivative is given by
	\begin{align} \label{eq:q-Frechet_deriv}
		q'(\bx, \bg) (\delta \bx, \delta \bg) 
		= & \nabla_x \varphi \big( \bx(t_f) \big)^T \delta \bx(t_f)   \nonumber \\
		 & + \int_{t_0}^{t_f} \Big[ \nabla_x l\Big( t, \bx(t), \bg\big(t, \bx(t) \big) \Big) 
	                                             + \bg_x\big(t, \bx(t) \big)^T \nabla_g l\Big( t, \bx(t), \bg\big(t, \bx(t) \big) \Big) \Big]^T \delta \bx(t) \nonumber \\
	          &\quad +  \nabla_g l\Big( t, \bx(t), \bg\big( t,  \bx(t) \big) \Big)^T \delta \bg\big(t,  \bx(t) \big) \, dt.
	\end{align}
	If, in addition, \cref{as:assumption1} holds for the function $\bff$ in \cref{eq:IVP}, then $\widetilde{q}$ in \cref{eq:QoI} is continuously Fr\'echet differentiable, and its derivative is given by
	\[
	\widetilde{q}_\bg(\overline{\bg}) \delta \bg = q'(\overline{\bx}, \overline{\bg}) (\delta \bx, \delta \bg)
	\]
	where $\overline{\bx} = \bx(\cdot \, ; \overline{\bg})$ is the solution of \cref{eq:IVP} given $\bg = \overline{\bg}$
	and $\delta \bx = \bx_\bg(\overline{\bg}) \delta \bg$ is the solution of \cref{eq:x-sensitivity}.
\end{theorem}
\begin{proof}
    Because $L^\infty(I) \ni h \mapsto \int_{t_0}^{t_f} h(t) \, dt \in \real$ is a bounded linear operator, the continuous Fr\'echet differentiability of
    \[
    	\LL \times \big(\cG^2(I)\big)^{n_g} \ni (\bx, \bg) \mapsto \int_{t_0}^{t_f} l\Big(t, \bx(t), \bg\big(t, \bx(t)\big)\Big) \, dt \in \real
    \]
   follows from the continuous Fr\'echet differentiability of $\BL_2(\bx, \bg)$. 
   Thus, \cref{eq:sensitivity:QoI-xg} is continuously Fr\'echet differentiable on $\WW \times \big(\cG^2(I)\big)^{n_g}$, 
   the integral term because $\WW$ is continuously embedded in $\LL$ and the final time term because $\varphi$ is 
   continuously differentiable and $\WW \ni \bx \mapsto \bx(t_f) \in \real^{n_x}$ is a bounded linear mapping since 
   $W^{1, \infty}(I)$ is continuously embedded in $C(\overline{I})$. 
   The form of \cref{eq:q-Frechet_deriv} follows from \cref{eq:l-Frechet_deriv}. This completes the first part of the proof. 
    The second part then immediately follows from \cref{thm:sensitivity_result}.
\end{proof}

The following theorem uses adjoints to compute the Fr\'echet derivative of $\widetilde{q}$ in \cref{eq:QoI} without solving a sensitivity equation.
\begin{theorem} \label{thm:qoi-sensitivity-g-dependent}
	If the assumptions of \cref{thm:sensitivity_result_qoi} hold, 
    then 
    \begin{equation} \label{eq:sensitivity_result_qoi_adjoint}
    	\widetilde{q}_\bg(\overline{\bg}) \delta \bg = \int_{t_0}^{t_f} \Big[ \overline{\BB}(t)^T \overline{\blambda}(t) + \nabla_g l \Big( t, \overline{\bx}(t), \overline{\bg} \big( t, \overline{\bx}(t) \big) \Big) \Big]^T \delta \bg \big( t, \overline{\bx}(t) \big) \, dt,
    \end{equation}
    where $\overline{\blambda} \in \WW$ solves the adjoint equation
       	\begin{equation} \label{eq:x-adjoint}
	\begin{aligned}
	       - \overline{\blambda}'(t) &=\overline{\BA}(t)^T \overline{\blambda}(t) + \nabla_x l\Big(t, \overline{\bx}(t), \overline{\bg} \big( t, \overline{\bx}(t) \big) \Big) 
	       + \overline{\bg}_x\big(t, \overline{\bx}(t) \big)^T \nabla_g l\Big( t, \overline{\bx}(t), \overline{\bg}\big(t, \overline{\bx}(t) \big) \Big), & \aall t \in I, \\
		\overline{\blambda}(t_f) &= \nabla_x \varphi\big(   \overline{\bx}(t_f) \big).
	\end{aligned}
   \end{equation}
\end{theorem}
\begin{proof}
If  $\delta \bx \in \WW$ solves \cref{eq:x-sensitivity} and $\overline{\blambda} \in \WW$ solves \cref{eq:x-adjoint}, then
   \begin{align*}
         \nabla_x \varphi\big(   \overline{\bx}(t_f) \big)^T  \delta \bx(t_f)
         &= \overline{\blambda}(t_f)^T \delta \bx(t_f) - \overline{\blambda}(t_0)^T \delta \bx(t_0)
            =  \int_{t_0}^{t_f} \overline{\blambda}'(t)^T \delta \bx(t) + \overline{\blambda}(t)^T \delta \bx'(t) \, dt \\
         & = \int_{t_0}^{t_f} -\Big[ \nabla_x l\Big( t, \overline{\bx}(t), \overline{\bg}\big(t, \overline{\bx}(t) \big) \Big) 
	              + \overline{\bg}_x\big(t, \overline{\bx}(t) \big)^T \nabla_g l\Big( t, \overline{\bx}(t), \overline{\bg}\big(t, \overline{\bx}(t) \big) \Big) \Big]^T \delta \bx(t) \\ &\qquad + \overline{\blambda}(t)^T \overline{\BB}(t) \delta \bg\big( t,  \overline{\bx}(t) \big) \, dt.
   \end{align*}
   Using the previous identity in \cref{eq:q-Frechet_deriv} gives \cref{eq:sensitivity_result_qoi_adjoint}.
\end{proof}

\section{Error Estimates for ODE and QoI} \label{sec:ODE_error}
 Let $\bg_* \in \big(\cG^2(I)\big)^{n_g}$ be the true model and let $\bx_* = \bx(\cdot \, ; \bg_*)$ denote the corresponding solution 
 of  \eqref{eq:IVP} with $\bg = \bg_*$.
 Suppose that we can only access an approximation $\widehat{\bg}$ of $\bg_*$ and therefore
 we can only compute the corresponding solution $\widehat{\bx} = \bx(\cdot \, ; \widehat{\bg})$
 of  \eqref{eq:IVP} with $\bg = \widehat{\bg}$.
 In this section we discuss estimates of the size of the solution
 error $\bx(\cdot \, ; \widehat{\bg}) - \bx(\cdot \, ; \bg_*)$ or of the error $| \widetilde{q}(\widehat{\bg}) - \widetilde{q}(\bg_*) | = \big| q\big(\bx(\cdot \, ; \widehat{\bg}), \widehat{\bg} \big) - q\big(\bx(\cdot \, ; \bg_*), \bg_* \big) \big|$
 in a quantity of interest \eqref{eq:QoI}.

Assume that all possible state trajectories are known to be contained in a domain $\Omega \subset \real^{n_x}$.
This allows to incorporate \emph{a priori} knowledge of the system, but $\Omega = \real^{n_x}$ is possible.
Furthermore, we assume that we have a componentwise error bound 
\begin{equation}    \label{eq:g-error-vec}
	| \widehat{\bg}(t, x) - \bg_*(t, x) | \leq  \bepsilon(t, x), \quad \aall t \in I \; \mbox{and all } x \in \Omega,
\end{equation}
where
$\bepsilon: I \times \Omega \rightarrow \real^{n_g}$ is a function that can be evaluated inexpensively for
almost all $t \in I$ and all $x \in \Omega$. 

Classical ODE perturbation theory provides an estimate for the solution error 
$\bx(\cdot \, ; \widehat{\bg}) - \bx(\cdot \, ; \bg_*)$ given a norm error bound
\begin{equation}    \label{eq:g-error}
	\| \widehat{\bg}(t, x) - \bg_*(t, x) \| \leq  \epsilon(t, x), \quad \aall t \in I \mbox{ and all } x \in \Omega,
\end{equation}
which will be reviewed next. Unfortunately, in many cases, this classical solution error bound
is extremely pessimistic and useless in practice. This has motivated our new
estimates based on sensitivity analysis, which will be presented in
\cref{sec:ODE_sensitivity_sol} and \cref{sec:ODE_sensitivity_qoi}.
In principle, our proposed estimates can also be used when a  norm error bound \eqref{eq:g-error}
is available instead of \eqref{eq:g-error-vec}.


\subsection{ODE Perturbation Theory} \label{sec:ODE_perturbation}
Many texts study the impact of perturbations in the IVP on its solution;
see, e.g.,  \cite[Sec.~I.10]{EHairer_SONorsett_GWanner_1993a}, \cite{GSoderlind_2006a}.
We adapt perturbation results for ODEs to our context. 
Our presentation is motivated by \cite{DWirtz_DCSorensen_BHaasdonk_2014a}.
To directly use the results from these references, we consider \cref{eq:IVP} in the classical 
setting in this section and assume that $\bff$ and $\bg_*$, $\widehat{\bg}$ are at least continuous in
all arguments. 

Let $Q \in \real^{n_x \times n_x}$ be a symmetric positive definite matrix and
consider the weighted inner product $x_1^T Q x_2$ with associated norm
$\| x \|_Q = \sqrt{x^T Q x}$.
The logarithmic Lipschitz constant of the function $$x \mapsto \bff\big(t, x, \bg _*(t,x)\big)$$ with respect to the $Q$-norm
is
\[
       L_Q[t,\bff,\bg_*] 
       = \lim_{h \rightarrow 0+} \frac{1}{h}
                     \left(  \sup_{x, y \in \real^{n_x}, x \not= y} \frac{ \Big\| x - y + h \Big( \bff\big(t, x, \bg _*(t,x)\big) - \bff\big(t, y, \bg _*(t,y)\big) \Big) \Big\|_Q }{ \| x - y \|_Q } - 1 \right).
\]
If $x \mapsto \bff\big(t, x, \bg _*(t,x)\big)$ is Lipschitz continuous for all $t \in I$, then
\begin{equation*}
       L_Q[t,\bff,\bg_*] =   \sup_{x, y \in \real^{n_x}, x \not= y} \frac{ ( x - y )^T Q \Big( \bff\big(t, x, \bg _*(t,x)\big) - \bff\big(t, y, \bg _*(t,y)\big) \Big) }{ \| x - y \|_Q^2 };
\end{equation*}
see \cite[Lemma~2.2]{DWirtz_DCSorensen_BHaasdonk_2014a}.
Following \cite[Def.~2.5]{DWirtz_DCSorensen_BHaasdonk_2014a}, the local logarithmic Lipschitz constant 
(with respect to the $Q$-norm) of the function $x \mapsto \bff\big(t, x, \bg _*(t,x)\big)$ is defined as
\begin{equation}    \label{eq:logarithmic-norm-local}
       L_Q[t,\bff,\bg_*](y) 
       =   \sup_{x \in \real^{n_x}, x \not= y} \frac{ ( x - y )^T Q \Big( \bff\big(t, x, \bg _*(t,x)\big) - \bff\big(t, y, \bg _*(t,y)\big) \Big) }{ \| x - y \|_Q^2 }.
\end{equation}

The next lemma is a variant of Gronwall's lemma.
\begin{lemma}   \label{lem:comparison}
  Let $T > 0$ and let $e, \alpha, \beta: [0, T] \rightarrow \real$ be integrable functions, with $e$ also differentiable.
  If
  \[
        e'(t) \le \beta(t) e(t) + \alpha(t), \qquad  t \in (0,T),
  \]
  then
  \[
         e(t)  \le \int_0^t  \alpha(s) \, \exp\big( \int_s^t  \beta(\tau)  d\tau \big) ds  +   \exp\big( \int_0^t  \beta(\tau)  d\tau \big) ds  \; e(0), \qquad t \in [0, T].
  \]
\end{lemma}
For a proof see, e.g., \cite[Lemma~2.6]{DWirtz_DCSorensen_BHaasdonk_2014a}.

\cref{lem:comparison} may be used to obtain a bound for the error  $\| \widehat{\bx}(t) - \bx_*(t) \|_Q$ 
involving the local logarithmic Lipschitz norm \cref{eq:logarithmic-norm-local} evaluated 
along the nominal trajectory $\widehat{\bx}$.

\begin{theorem}   \label{thm:error-estimate-comparison-lemma}
     Let   $L_Q[t,\bff,\bg_*]\big(  \widehat{\bx}(t) \big)$ be the local logarithmic Lipschitz constant  \cref{eq:logarithmic-norm-local}
     of $x \mapsto \bff\big(t, x, \bg_*(t, x)\big)$ at $\widehat{\bx}(t)$. Define
     $R_{L^\infty} := \| \widehat{\bx} \|_{(L^\infty(I))^{n_x}}$,
     $R_{\cG} := \| \widehat{\bg} \|_{(\cG^0(I))^{n_g}} + \| \bg_* \|_{(\cG^0(I))^{n_g}}$,
     and assume there exists $L > 0$ such that
     \begin{equation} \label{eq:uniformly-lipschitz}
     	\| \bff(t, x, g_1) -  \bff(t, x, g_2) \|_Q \le L \| g_1 - g_2 \|_Q \quad \textrm{for all $t \in I$, $x \in \cB_{R_{L^\infty}}(0)$, $g_1, g_2 \in \cB_{R_{\cG}}(0)$.}
     \end{equation}
     If the error bound \cref{eq:g-error} holds in the $Q$-norm, and if $t \mapsto L_Q[t,\bff,\bg_*]\big(  \widehat{\bx}(t) \big)$
     and $t \mapsto \epsilon\big(t, \widehat{\bx}(t) \big)$ are integrable,  then the following Gronwall-type error bound holds:
     \begin{equation}   \label{eq:error-estimate-comparison-lemma}
          \big\| \widehat{\bx}(t) - \bx_*(t) \big\|_Q
          \le L   \int_0^t    \epsilon\big(s, \widehat{\bx}(s) \big)   \, 
                             \exp\big( \int_s^t  L_Q[\tau, \bff,\bg_*]\big(  \widehat{\bx}(\tau) \big)  d\tau \big) \, ds =: \BE(t).
     \end{equation}
\end{theorem}
\begin{proof}
    Since $\bx_*$ and $\widehat{\bx}$ are the solutions of \cref{eq:IVP} with $\bg = \bg_*$ and $\bg = \widehat{\bg}$,
    respectively,
    \begin{align*}
       & \Big( \bx_*(t) - \widehat{\bx}(t) \Big)^T Q  \Big(  \bx_*'(t) - \widehat{\bx}'(t) \Big)  \\
         &=  \Big( \bx_*(t) - \widehat{\bx}(t) \Big)^T Q
               \Bigg( \bff\Big( t, \bx_*(t), \bg_*\big( t, \bx_*(t) \big) \Big) -  \bff\Big( t, \widehat{\bx}(t), \bg_*\big( t, \widehat{\bx}(t) \big) \Big) \Bigg)  \\
            & \quad  + \Big( \bx_*(t) - \widehat{\bx}(t) \Big)^T Q
                   \Bigg( \bff\Big( t, \widehat{\bx}(t), \bg_*\big( t, \widehat{\bx}(t) \big) \Big) -  \bff\Big( t, \widehat{\bx}(t), \widehat{\bg}\big(t, \widehat{\bx}(t) \big) \Big) \Bigg) \\
       &\le  L_Q[t, \bff,\bg_*]\big(  \widehat{\bx}(t) \big)  \big\| \bx_*(t) - \widehat{\bx}(t) \big\|_Q^2 
                 + L   \Big\| \bg_*\big( t, \widehat{\bx}(t) \big)  -  \widehat{\bg}\big(\widehat{\bx}(t) \big) \Big\|_Q \big\| \bx_*(t) - \widehat{\bx}(t) \big\|_Q \\
      &\le  L_Q[t, \bff,\bg_*]\big(  \widehat{\bx}(t) \big)  \big\| \bx_*(t) - \widehat{\bx}(t) \big\|_Q^2 
                 + L    \epsilon\big( t, \widehat{\bx}(t) \big)  \big\| \bx_*(t) - \widehat{\bx}(t) \big\|_Q.
    \end{align*}
    Hence
    \begin{align*}           
         \frac{d}{dt}  \big\| \bx_*(t) - \widehat{\bx}(t) \big\|_Q
        &= \Big( \bx_*(t) - \widehat{\bx}(t) \Big)^T Q  \Big(  \bx_*'(t) - \widehat{\bx}'(t) \Big)   \; \big/ \; \big\| \bx_*(t) - \widehat{\bx}(t) \big\|_Q \\
        &\le L_Q[t,\bff,\bg_*]\big(  \widehat{\bx}(t) \big)  \big\| \bx_*(t) - \widehat{\bx}(t) \big\|_Q
                 + L    \epsilon\big( t, \widehat{\bx}(t) \big). 
    \end{align*}
    Using \cref{lem:comparison} gives the desired result. 
 \end{proof}

The local logarithmic Lipschitz constant \cref{eq:logarithmic-norm-local} is difficult to
compute. If $\bff$ is Lipschitz continuously differentiable in $x$ and $g$ and
$\bg$ is Lipschitz continuously differentiable, then we can use the Taylor expansion of 
$x \mapsto  \bff\big(t, x, \bg _*(t,x)\big)$ at $y$ to write
\begin{equation*}
 \frac{ ( x - y )^T Q \Big( \bff\big(t, x, \bg _*(t,x)\big) - \bff\big(t, y, \bg _*(t,y)\big) \Big) }{ \| x - y \|_Q^2 }
 = \frac{ ( x - y )^T Q \BA_*(t,y)  ( x - y )  }{ \| x - y \|_Q^2 }
                                    + O(\| x - y \|_Q ),
\end{equation*}
where
\[
            \BA_*(t,y) =  \bff_x\Big( t, y, \bg_*\big( t, y \big) \Big)  + \bff_g\Big( t, y, \bg_*\big( t, y \big) \Big) (\bg_*)_x\big( t, y \big).
\]
Omitting the $O(\| x - y \|_Q )$ term we can approximate
\begin{subequations}    \label{eq:logarithmic-norm-local-deriv}
\begin{equation}
       L_Q[t,\bff,\bg_*](y) 
       \approx 
        \widetilde{L}_Q[t,\bff,\bg_*](y) 
       =  \sup_{v \in \real^{n_x}, v \not= 0} \frac{ v^T Q \BA_*(t,y)  v  }{ \| v\|_Q^2 },
\end{equation}
which is the logarithmic norm of the matrix $\BA_*(t,y)$, and which can be computed via
\begin{align}    
       \widetilde{L}_Q[t,\bff,\bg_*](y) 
       =   \max \sigma  \Big( \half \big( C^T \BA_*(t,y)  C^{-T} + C^{-1} \BA_*(t,y)^T  C  \big) \Big),
\end{align}
\end{subequations}
where $Q = C C^T$ is the Cholesky decomposition of $Q$ and 
$\sigma(M)$ denotes the spectrum of the matrix $M$;
see, e.g., \cite[Corollary~2.3]{DWirtz_DCSorensen_BHaasdonk_2014a}.
The approximation \cref{eq:logarithmic-norm-local-deriv} still depends on the unknown $\bg_*$, 
and one can replace $\bg_*$ by $\widehat{\bg}$ to arrive at a computable quantity; however,
in the numerical example in \cref{sec:numerics_flap}, we have access to $\bg_*$ and use \cref{eq:logarithmic-norm-local-deriv}.
Unfortunately, as we will see in \cref{sec:numerics_flap}, the error bound \cref{eq:error-estimate-comparison-lemma} approximated using
\cref{eq:logarithmic-norm-local-deriv} can be extremely pessimistic. This motivates the need for our sensitivity-based error bounds, which will be introduced next.


\subsection{Sensitivity-Based Error Estimation for ODE Solution}  \label{sec:ODE_sensitivity_sol}
 We will use sensitivities to estimate the error 
 \begin{equation}   \label{eq:x+-err-size}
       \| \widehat{\bx} - \bx_* \|_\BQ^2 := \int_{t_0}^{t_f} \big( \bx(t;\widehat{\bg}) - \bx(t;\bg_*) \big)^T \BQ(t) \big( \bx(t;\widehat{\bg}) - \bx(t;\bg_*) \big) \, dt
\end{equation}
with a user-specified matrix-valued function $\BQ \in \big(L^\infty(I)\big)^{n_x \times n_x}$ such that for almost all $t \in I$ the matrix $\BQ(t)$ is symmetric positive semidefinite, with corresponding (semi)norm $\| x \|_{\BQ(t)} = \sqrt{x^T \BQ(t) x}$.

Recall from Theorem \ref{thm:sensitivity_result} that under Assumption \ref{as:assumption1}, the solution $\bx(\cdot \, ; \bg) \in \WW$ 
of \eqref{eq:IVP} is continuously Fr\'echet differentiable with respect to $\bg \in \big(\cG^2(I)\big)^{n_g}$.
We approximate
\begin{equation}   \label{eq:x-err-sensitivity}
	\bx(\widehat{\bg}) - \bx(\bg_*) 
	=  \bx_\bg(\widehat{\bg})  (\widehat{\bg} - \bg_*)
	        + \int_0^1 \big[\bx_\bg(\bg_*  + s (\widehat{\bg} - \bg_*)) - \bx_\bg(\widehat{\bg})\big] (\widehat{\bg} - \bg_*) \, ds
	\approx \bx_\bg(\widehat{\bg})  (\widehat{\bg} - \bg_*).
\end{equation}

If we knew  $\delta \bg_* := \widehat{\bg} - \bg_*$, then $ \bx_\bg(\widehat{\bg})  \delta \bg_*$ could be computed as
the solution of \eqref{eq:x-sensitivity} with the current  $\bg = \widehat{\bg}$ and  $\delta \bg = \delta \bg_*$.
However, the sensitivity equations \eqref{eq:x-sensitivity} require $\delta \bg_*\big(t, \widehat{\bx}(t)\big)$, $t \in I$, which is expensive,
but from the componentwise error bound \eqref{eq:g-error-vec} we
know that 
\begin{equation}    \label{eq:g-error-vec-delta}
	\big| \delta \bg_*( t, \widehat{\bx}(t) ) \big| \leq  \bepsilon\big(t, \widehat{\bx}(t)\big), \qquad \aall t \in I.
\end{equation}
We assume that $t \mapsto \bepsilon\big( t, \widehat{\bx}(t) \big) \in \big(L^\infty(I) \big)^{n_g}$.
Motivated by the error estimate \eqref{eq:x-err-sensitivity}, the 
sensitivity equation \eqref{eq:x-sensitivity} with the current  $\bg = \widehat{\bg}$ and  $\delta \bg = \delta \bg_*$,
and the model error bound \eqref{eq:g-error-vec-delta}, we consider the following optimization problem to obtain an approximate upper bound on the error measure \eqref{eq:x+-err-size}:
\begin{equation} \label{eq:Refinement:ODE:LQOCP0}
\begin{aligned}
	\sup_{\delta \bx, \delta \bg} \quad &\frac{1}{2} \int_{t_0}^{t_f} \| \delta \bx(t) \|_{\BQ(t)}^2 \, dt \\
	\mbox{s.t.} \quad & \delta \bx'(t) = \widehat{\BA}(t)  \delta \bx(t) +  \widehat{\BB}(t) \delta \bg\big( t, \widehat{\bx}(t) \big), & \qquad \aall t \in I, \\
	& \delta \bx(t_0) = 0, \\
	& - \bepsilon\big( t, \widehat{\bx}(t) \big) \leq \delta \bg\big( t, \widehat{\bx}(t) \big) \leq \bepsilon\big( t, \widehat{\bx}(t) \big), & \qquad \aall t \in I,
\end{aligned}
\end{equation}
where $\widehat{\BA}, \widehat{\BB}$ are given by \eqref{eq:shorthand} with $\bx, \bg$ replaced by $\widehat{\bx}, \widehat{\bg}$ respectively.

The idea behind \eqref{eq:Refinement:ODE:LQOCP0} is that we consider all possible model perturbations that obey the pointwise error bound along the nominal trajectory and use the sensitivity equation to determine the worst-case perturbation in the corresponding ODE solution.
The composition $\delta \bg \big(\cdot, \widehat{\bx}(\cdot) \big)$ can be replaced
by a function $\bdelta \in \big( L^\infty(I) \big)^{n_g}$, yielding the linear quadratic optimal control problem
\begin{subequations} \label{eq:Refinement:ODE:LQOCP}
\begin{align}
	\sup_{\delta \bx, \bdelta} \quad &\frac{1}{2} \int_{t_0}^{t_f} \| \delta \bx(t) \|_{\BQ(t)}^2 \, dt \\
	\mbox{s.t.} \quad & \delta \bx'(t) = \widehat{\BA}(t)  \delta \bx(t) +  \widehat{\BB}(t)   \bdelta(t), & \qquad \aall t \in I, \\
	& \delta \bx(t_0) = 0, \\
	& - \bepsilon\big( t, \widehat{\bx}(t) \big) \leq \bdelta(t) \leq \bepsilon\big( t, \widehat{\bx}(t) \big), & \qquad \aall t \in I. \label{eq:ODE_sensitivity_sol:box-constraints}
\end{align}
\end{subequations}

\begin{theorem}  \label{th:ODE:LQOCP0=LQOCP}
    The problems \eqref{eq:Refinement:ODE:LQOCP0} and  \eqref{eq:Refinement:ODE:LQOCP} are equivalent.
\end{theorem}
\begin{proof}
     Since the objective function depends only on $\delta \bx \in \WW$, it is sufficient to show that
     every feasible point of \eqref{eq:Refinement:ODE:LQOCP0} corresponds to a  feasible point of \eqref{eq:Refinement:ODE:LQOCP} and vice versa.
     If the pair $\delta \bx \in \WW$, $\delta \bg \in \big(\cG^2(I)\big)^{n_g}$ is feasible for  \eqref{eq:Refinement:ODE:LQOCP0},
     then define $\bdelta(t) := \delta \bg\big( t, \widehat{\bx}(t) \big)$. The pair 
     $\delta \bx \in \WW$, $ \bdelta \in  \big( L^\infty(I) \big)^{n_g}$ is feasible for  \eqref{eq:Refinement:ODE:LQOCP}.
     On the other hand, if the pair $\delta \bx \in \WW$, $ \bdelta \in  \big( L^\infty(I) \big)^{n_g}$ is feasible for  \eqref{eq:Refinement:ODE:LQOCP},
     then define  $\delta \bg( t, x ) :=  \bdelta(t)$. The pair  $\delta \bx \in \WW$, $\delta \bg \in \big(\cG^2(I)\big)^{n_g}$ is feasible for  
     \eqref{eq:Refinement:ODE:LQOCP0}.
\end{proof}

The problem \eqref{eq:Refinement:ODE:LQOCP} is a convex linear quadratic optimal control problem,
but we seek a {\em supremum} rather than an infimum; therefore, standard techniques for establishing existence of solutions cannot be applied here.
Moreover, if a solution exists, it is not unique; for instance, if $\bdelta$ solves \eqref{eq:Refinement:ODE:LQOCP}, then $- \bdelta$ does as well. 
After discretizing \eqref{eq:Refinement:ODE:LQOCP}, the resulting finite-dimensional linearly constrained quadratic program (LCQP)  has a solution.
However, the finite-dimensional LCQP is NP-hard, 
as it is a convex maximization problem and therefore has optimal solutions at the vertices of the feasible polyhedron, the number of which grows exponentially 
with the problem dimension. See, e.g., \cite{HPBenson_1995a},
\cite{SBurer_ANLetchford_2009a}, \cite{RHorst_PMPardalos_NVThoai_2000a}.
Despite these issues, for discretizations of \eqref{eq:Refinement:ODE:LQOCP},
we can often compute an approximate discretized state-control pair $(\delta \bx, \bdelta)$ whose objective value is close
to the maximum
using a tailored interior point method that exploits the symmetry of the discretized box constraints \eqref{eq:ODE_sensitivity_sol:box-constraints}.

Because we assume $\widehat{\bx}$ has already been computed and the error bound function $\bepsilon: I \times \Omega \rightarrow \real^{n_g}$ 
can be evaluated inexpensively for almost all $t \in I$ and all $x \in \Omega$,
the linear quadratic optimal control problem \eqref{eq:Refinement:ODE:LQOCP} can be set up inexpensively.

If an optimal solution to \eqref{eq:Refinement:ODE:LQOCP} exists, then the optimal objective function value
is an upper bound of the size of the error estimate 
$\bx_\bg(\widehat{\bg})  (\widehat{\bg} - \bg_*)$.
\begin{theorem}   \label{thm:state-error-bound}
     If $\bdelta \in \big(L^\infty(I) \big)^{n_g}$ and $\delta \bx \in \WW$ solve \eqref{eq:Refinement:ODE:LQOCP},
     then 
     \[
         \frac{1}{2} \int_{t_0}^{t_f}   \big\| [\bx_\bg(\widehat{\bg})  (\widehat{\bg} - \bg_*)](t)  \big\|_{\BQ(t)}^2 \, dt 
         \le  \frac{1}{2} \int_{t_0}^{t_f}  \big\| \delta \bx(t) \big\|_{\BQ(t)}^2 \, dt. 
     \]
\end{theorem}
\begin{proof}
    Because of the model error bound \eqref{eq:g-error-vec-delta}, 
    $\bdelta(t) := \delta \bg_*\big(t, \widehat{\bx}(t)\big) = \widehat{\bg}\big(t, \widehat{\bx}(t)\big) - \bg_*\big(t, \widehat{\bx}(t)\big)$
    satisfies \eqref{eq:ODE_sensitivity_sol:box-constraints}, and 
    $\bx_\bg(\widehat{\bg}) \delta \bg_* \in \WW$ 
    is a corresponding  feasible state  for  \eqref{eq:Refinement:ODE:LQOCP}.
    Thus, their objective function value is less than or equal to the optimal objective function value,
    which is the desired inequality.
\end{proof}
In cases where \eqref{eq:Refinement:ODE:LQOCP} has no solution, we may at best obtain an approximate upper bound by taking a sufficiently fine discretization of \eqref{eq:Refinement:ODE:LQOCP} and solving the resulting LCQP. 
In the numerical examples shown in Section~\ref{sec:numerics_flap} and in the supplement, we still obtain excellent estimates for the solution error despite this theoretical shortcoming.

\begin{remark}   \label{rem:Refinement:ODE:control-constraints-norm}
     If,  instead of the   componentwise error bound \eqref{eq:g-error-vec}, we have
     a norm error bound \eqref{eq:g-error}, then the control constraints \eqref{eq:ODE_sensitivity_sol:box-constraints}
     have to be replaced by
     \begin{equation} \label{eq:Refinement:ODE:control-constraints-norm}
            \|   \bdelta( t )  \|  \leq  \epsilon\big( t, \widehat{\bx}(t) \big), \quad \aall t \in I.
     \end{equation}
     Theorem~\ref{thm:state-error-bound} remains valid after this change of control constraints.
     However, depending on the choice of norm in \eqref{eq:Refinement:ODE:control-constraints-norm}, the
     resulting optimal control problem may be more difficult to solve than \eqref{eq:Refinement:ODE:LQOCP},
     which is why we have focused on componentwise error bounds \eqref{eq:g-error-vec}.
\end{remark}


\subsection{Sensitivity-Based Error Estimation for Quantity of Interest}   \label{sec:ODE_sensitivity_qoi}
In the previous two subsections the goal was to analyze the solution error $\bx(\cdot \, ; \widehat{\bg}) - \bx(\cdot \, ; \bg_*)$.
Often, however, we are interested in a quantity of interest \eqref{eq:QoI} and want to analyze
\begin{equation} \label{eq:QoI-error-measure}
| \widetilde{q}(\widehat{\bg}) - \widetilde{q}(\bg_*) | = \big| q\big(\bx(\cdot \, ; \widehat{\bg}), \widehat{\bg} \big) - q\big(\bx(\cdot \, ; \bg_*), \bg_*\big) \big|.
\end{equation}
We proceed as in the previous section.

Under the assumptions of \cref{thm:sensitivity_result_qoi} the quantity of interest \eqref{eq:QoI} is continuously
Fr\'echet differentiable with respect to $\bg \in \big(\cG^2(I)\big)^{n_g}$. We approximate
\begin{equation*}
	| \widetilde{q}(\widehat{\bg}) - \widetilde{q}(\bg_*) | 
	\approx  | \widetilde{q}_\bg(\widehat{\bg})   (\widehat{\bg} - \bg_*) |.
\end{equation*}
If we knew  $\delta \bg_* = \widehat{\bg} - \bg_*$, then $\widetilde{q}_\bg(\widehat{\bg})   \delta \bg_*$ could be computed 
using the adjoint equation approach based on information at the 
already computed $\widehat{\bx}$. 
Specifically, if the matrices $\widehat{\BA}(t)$, $\widehat{\BB}(t)$ are given by \eqref{eq:shorthand} with
$\bx$, $\bg$ replaced by $\widehat{\bx}$, $\widehat{\bg}$, 
it follows from \cref{eq:sensitivity_result_qoi_adjoint} that
\begin{subequations} \label{eq:sensitivity_result_qoi_adjoint-c}
 \begin{equation} 
         |    \widetilde{q}_\bg(\widehat{\bg}) \delta \bg_* | 
            = \Big|   \int_{t_0}^{t_f}  \Big[  \widehat{\BB}(t)^T \widehat{\blambda}(t) + \nabla_g l \Big( t, \widehat{\bx}(t), \widehat{\bg} \big( t, \widehat{\bx}(t) \big) \Big)  \Big]^T \delta \bg_*\big( t, \widehat{\bx}(t) \big) \, dt \, \Big|
 \end{equation}
where  $\widehat{\blambda} \in \WW$ solves the adjoint equation
 \begin{equation}    \label{eq:x-adjoint-c}
	\begin{aligned}
	       - \widehat{\blambda}'(t) &=  \widehat{\BA}(t)^T \widehat{\blambda}(t) + \nabla_x l\big(t, \widehat{\bx}(t)\big) + \widehat{\bg}_x \big( t, \widehat{\bx}(t) \big)^T \nabla_g l\Big(t, \widehat{\bx}(t), \widehat{\bg} \big( t, \widehat{\bx}(t) \big) \Big), & \aall t \in I, \\
		\widehat{\blambda}(t_f) &= \nabla_x \varphi\big( \widehat{\bx}(t_f) \big).
	\end{aligned}
\end{equation}
\end{subequations}

Similar to the approach in the previous section, we use the adjoint-based sensitivity result \eqref{eq:sensitivity_result_qoi_adjoint-c} and the model error bound \eqref{eq:g-error-vec-delta} to motivate the following problem to compute an approximate upper bound for the QoI error \eqref{eq:QoI-error-measure}:
\begin{subequations}    \label{eq:Refinement:ODE:QoI-LP0}
\begin{align}
	\max_{\delta \bg} \quad & \Big|   \int_{t_0}^{t_f}  \Big[  \widehat{\BB}(t)^T \widehat{\blambda}(t) + \nabla_g l \Big( t, \widehat{\bx}(t), \widehat{\bg} \big( t, \widehat{\bx}(t) \big) \Big)  \Big]^T \delta \bg\big( t, \widehat{\bx}(t) \big) \, dt \, \Big|   \\
	\mbox{s.t.} \quad &  - \bepsilon\big( t, \widehat{\bx}(t) \big) \leq \delta \bg \big(t, \widehat{\bx}(t) \big) \leq \bepsilon\big( t, \widehat{\bx}(t) \big), & \qquad \aall t \in I.
\end{align}
\end{subequations}
Next, we replace $\delta \bg\big( \cdot, \widehat{\bx}(\cdot) \big) \in \big(L^\infty(I) \big)^{n_g}$ by $\bdelta \in \big(L^\infty(I) \big)^{n_g}$ to get
\begin{subequations} \label{eq:Refinement:ODE:QoI-LP}
\begin{align}
	\max_{\bdelta } \quad & \Big| \int_{t_0}^{t_f}  \Big[  \widehat{\BB}(t)^T \widehat{\blambda}(t) + \nabla_g l \Big( t, \widehat{\bx}(t), \widehat{\bg} \big( t, \widehat{\bx}(t) \big) \Big)  \Big]^T \bdelta(t) \, dt \, \Big| \label{eq:ODE_sensitivity_qoi:objective}  \\
	\mbox{s.t.} \quad &  - \bepsilon\big( t, \widehat{\bx}(t) \big) \leq \bdelta(t) \leq \bepsilon\big( t, \widehat{\bx}(t) \big), & \qquad \aall t \in I. \label{eq:ODE_sensitivity_qoi:box-constraints}
\end{align}
\end{subequations}
The problem \eqref{eq:Refinement:ODE:QoI-LP} is a simple linear program in $\bdelta \in \big(L^\infty(I) \big)^{n_g}$ that has an easily computable analytical solution.

\begin{theorem}   \label{th:Refinement:ODE:QoI-LP-sol}
    The problems \eqref{eq:Refinement:ODE:QoI-LP0} and  \eqref{eq:Refinement:ODE:QoI-LP} are equivalent.

     The linear program \eqref{eq:Refinement:ODE:QoI-LP} is solved by the functions $\pm \widehat{\bdelta}$, where $\widehat{\bdelta}$ is defined componentwise by
     \begin{equation}   \label{eq:Refinement:ODE:QoI-LP-sol}
           \widehat{\bdelta}_i(t)   
           \begin{cases} 
                       = \bepsilon_i\big( t, \widehat{\bx}(t) \big)  & \mbox{if }  \Big[\widehat{\BB}(t)^T \widehat{\blambda}(t) + \nabla_g l \Big( t, \widehat{\bx}(t), \widehat{\bg} \big( t, \widehat{\bx}(t) \big) \Big) \Big]_i > 0, \\[1ex]
                        = -\bepsilon_i\big( t, \widehat{\bx}(t) \big)  & \mbox{if }  \Big[\widehat{\BB}(t)^T \widehat{\blambda}(t) + \nabla_g l \Big( t, \widehat{\bx}(t), \widehat{\bg} \big( t, \widehat{\bx}(t) \big) \Big) \Big]_i < 0, \\[1ex]
                         \in \big[  - \bepsilon_i\big( t, \widehat{\bx}(t) \big) , \, \bepsilon_i\big( t, \widehat{\bx}(t) \big) \big] &  \mbox{if }  \Big[\widehat{\BB}(t)^T \widehat{\blambda}(t) + \nabla_g l \Big( t, \widehat{\bx}(t), \widehat{\bg} \big( t, \widehat{\bx}(t) \big) \Big) \Big]_i  = 0,
           \end{cases}
     \end{equation}
     for all $i = 1, \dots, n_g$.
The optimal objective function value is
     \begin{equation}   \label{eq:Refinement:ODE:QoI-LP-sol-obj}
\int_{t_0}^{t_f}  \Big|  \widehat{\BB}(t)^T \widehat{\blambda}(t) + \nabla_g l \Big( t, \widehat{\bx}(t), \widehat{\bg} \big( t, \widehat{\bx}(t) \big) \Big)  \Big|^T   \bepsilon\big( t, \widehat{\bx}(t) \big) \,  dt,
     \end{equation}
     where the absolute value is applied componentwise.
\end{theorem}
\begin{proof}
    The equivalence of  \eqref{eq:Refinement:ODE:QoI-LP0} and  \eqref{eq:Refinement:ODE:QoI-LP} can be shown as in the proof
    of \cref{th:ODE:LQOCP0=LQOCP}.
    
    The function $\bdelta$ satisfies \eqref{eq:ODE_sensitivity_qoi:box-constraints} if and only if $-\bdelta$ satisfies
     \eqref{eq:ODE_sensitivity_qoi:box-constraints},
    and $\pm \bdelta$  have the same objective function values.
    Thus,  $\bdelta$  solves \eqref{eq:Refinement:ODE:QoI-LP} if and only
    if $-\bdelta$ solves \eqref{eq:Refinement:ODE:QoI-LP}, and we can solve \eqref{eq:Refinement:ODE:QoI-LP} without the absolute value in the objective
    function \eqref{eq:ODE_sensitivity_qoi:objective}. By inspection, the function  $\bdelta$ that maximizes 
this value subject to  \eqref{eq:ODE_sensitivity_qoi:box-constraints}
    is given by \eqref{eq:Refinement:ODE:QoI-LP-sol} with objective value \eqref{eq:Refinement:ODE:QoI-LP-sol-obj}.
\end{proof}

Analogously to \cref{thm:state-error-bound} we have the following bound. 
Unlike \cref{thm:state-error-bound}, this bound is always easily computable 
because \eqref{eq:Refinement:ODE:QoI-LP} always has a solution \eqref{eq:Refinement:ODE:QoI-LP-sol}.

\begin{theorem}   \label{thm:QoI-error-bound}
	If the assumptions of \cref{thm:sensitivity_result_qoi} hold, then the following bound holds:
     \[
         |    \widetilde{q}_\bg(\widehat{\bg}) ( \widehat{\bg} - \bg_* )  | 
         \le  \int_{t_0}^{t_f}  \Big|  \widehat{\BB}(t)^T \widehat{\blambda}(t) + \nabla_g l \Big( t, \widehat{\bx}(t), \widehat{\bg} \big( t, \widehat{\bx}(t) \big) \Big)  \Big|^T  \bepsilon\big( t, \widehat{\bx}(t) \big) \, dt.
     \]
\end{theorem}
\begin{proof}
    The result follows from \cref{th:Refinement:ODE:QoI-LP-sol}
    because
    $\bdelta(t) := \delta \bg_*\big(t, \widehat{\bx}(t)\big) = \widehat{\bg}\big(t, \widehat{\bx}(t)\big) - \bg_*\big(t, \widehat{\bx}(t)\big)$
    is feasible for \eqref{eq:Refinement:ODE:QoI-LP} due to \eqref{eq:g-error-vec-delta}, so the corresponding objective value is bounded by the optimal objective value \eqref{eq:Refinement:ODE:QoI-LP-sol-obj}.
\end{proof}

\begin{remark}   \label{rem:Refinement:ODE:control-constraints-norm-QoI}
     If,  instead of the   componentwise error bound \eqref{eq:g-error-vec}, we have
     a norm error bound \eqref{eq:g-error}, then the control constraints \eqref{eq:ODE_sensitivity_qoi:box-constraints}
     are replaced by
     \begin{equation*}
            \|   \bdelta( t )  \|  \leq  \epsilon\big( t, \widehat{\bx}(t) \big), \quad \aall t \in I.
     \end{equation*}
     The resulting optimization problem will not have an analytical solution in general; however, one
     could still use numerical optimization methods.
\end{remark}


\section{Hypersonic Vehicle Trajectory Simulation} \label{sec:numerics_flap}
We present numerical results for a system of ODEs that employ a model function $\bg$. 
A second example is shown in the supplement.
We assume there is a true function $\bg_*$, but we may only solve the ODE using an approximation $\widehat{\bg}$. We compute the error in the ODE solution and compare with the sensitivity-based estimate given by \cref{thm:sensitivity_result} as well as the Gronwall and sensitivity-based error bounds derived in \cref{thm:error-estimate-comparison-lemma} and~\cref{thm:state-error-bound} respectively. For both problems, the $Q$-norm in \cref{thm:error-estimate-comparison-lemma} is the 2-norm, and the matrix-valued function $\BQ(t)$ in \cref{thm:state-error-bound} is simply the identity matrix, yielding an $L^2$-norm; this choice ensures an equitable comparison. We also compute the error in a QoI and compare with the sensitivity-based estimate given by \cref{thm:sensitivity_result_qoi} as well as the sensitivity-based error bound given by \cref{thm:QoI-error-bound}.

We consider a dynamical system for a notional hypersonic vehicle in longitudinal flight. 
Time $t$ is measured in seconds, and the states are horizontal position $\bx_1$ [km], altitude $\bx_2$ [km], speed $\bv$  [km/s], and flight path angle $\bgamma$  [$^\circ$],
i.e., in this example,
\[
	\bx(t) =  \big( \bx_1(t), \bx_2(t), \bv(t), \bgamma(t) \big).
\]
The angle of attack $\balpha$ [$^\circ$] is a given input, which we set to
\[
	\balpha(t) = \Bigg(10 - \left(\frac{t}{2000} \right) 4 \Bigg)^\circ, \qquad t \in [0, 2000].
\]
In a trajectory optimization problem, it would be the control. 

The hypersonic vehicle used in this example has mass $m = 1200 \textrm{ kg}$ 
and reference area $A_w = 10 \textrm{ m$^2$}$. Lift and drag are given by
 \begin{align*}
			L(x_2, v, \alpha) = \overline{q}(x_2, v) C_L(\alpha) A_w, \qquad
			D(x_2, v, \alpha) = \overline{q}(x_2, v) C_D(\alpha) A_w, 
\end{align*}
where
\begin{equation*}
	\overline{q}(x_2, v) = \frac{1}{2} \rho(x_2) v^2 
\end{equation*}
is the dynamic pressure, which depends on atmospheric density $\rho(x_2) = 1.225 \, \mathrm{exp}(-0.14 x_2)$ [kg/m$^3$].
The lift and drag coefficients $C_L, C_D$ are assumed to be functions of angle of attack that are expensive to compute. They will play the role of the model function in this example, i.e.,
\[
      \bg \big(t, \bx(t) \big) =  \begin{pmatrix}  C_L \big(\balpha(t) \big)  \\  C_D \big(\balpha(t) \big) \end{pmatrix}.
\]
In this example, the ``true'' lift and drag coefficients are taken from \cite{STauqeerulIslamRizvi_LHe_DXu_2015a} and are given by
\[
	C_L^*(\alpha) = -0.04 + 0.8 \alpha, \qquad C_D^*(\alpha) = 0.012 - 0.01\alpha + 0.6 \alpha^2
\]
where $\alpha$ is in radians.

The dynamics of the hypersonic vehicle also depend on gravitational acceleration, which is computed as $g(x_2) = \mu / (R_E + x_2)^2$ [m/s$^2$],
where $\mu = 3.986 \times 10^{14} \textrm{ m$^3$/s$^2$}$ is the standard gravitational parameter and $R_E \approx 6.371 \times 10^6 \textrm{ m}$ is the radius of Earth.

The dynamic equations are given by 
\begin{equation*}
\begin{aligned}
    \bx_1'(t) &= \bv(t) \cos \bgamma(t), \\
    \bx_2'(t) &= \bv(t) \sin \bgamma(t), \\
    \bv'(t) &= -\frac{1}{m} \Big( D\big(\bx_2(t), \bv(t), \balpha(t)\big) + m g\big(\bx_2(t)\big) \sin \bgamma(t) \Big), \\
    \bgamma'(t) &= \frac{1}{m\bv(t)} \Big( L\big(\bx_2(t), \bv(t), \balpha(t)\big) - mg\big(\bx_2(t)\big) \cos \bgamma(t) + \frac{m\bv(t)^2 \cos \bgamma(t)}{R_E + \bx_2(t)} \Big)
\end{aligned}
\end{equation*}
with initial conditions
\[
 	\bx_1(0) = 0, \quad \bx_2(0) = 80, \quad \bv(0) = 5, \quad \bgamma(0) = -5^\circ.
\]

We assume the lift and drag coefficients are estimated by a perturbed model
\[
	C_L^\epsilon(\alpha) = -0.04 + (0.8 + \epsilon) \alpha, \qquad C_D^\epsilon(\alpha) = 0.012 - 0.01 \alpha + (0.6 - \epsilon) \alpha^2.
\]
For $\epsilon = 0.01$ the perturbed trajectory $\widehat{\bx}$ and the true trajectory $\bx_*$ are shown in \cref{fig:hypersonic-trajectories}.

\begin{figure}[!htb]
\centering
\includegraphics[width=0.6\textwidth]{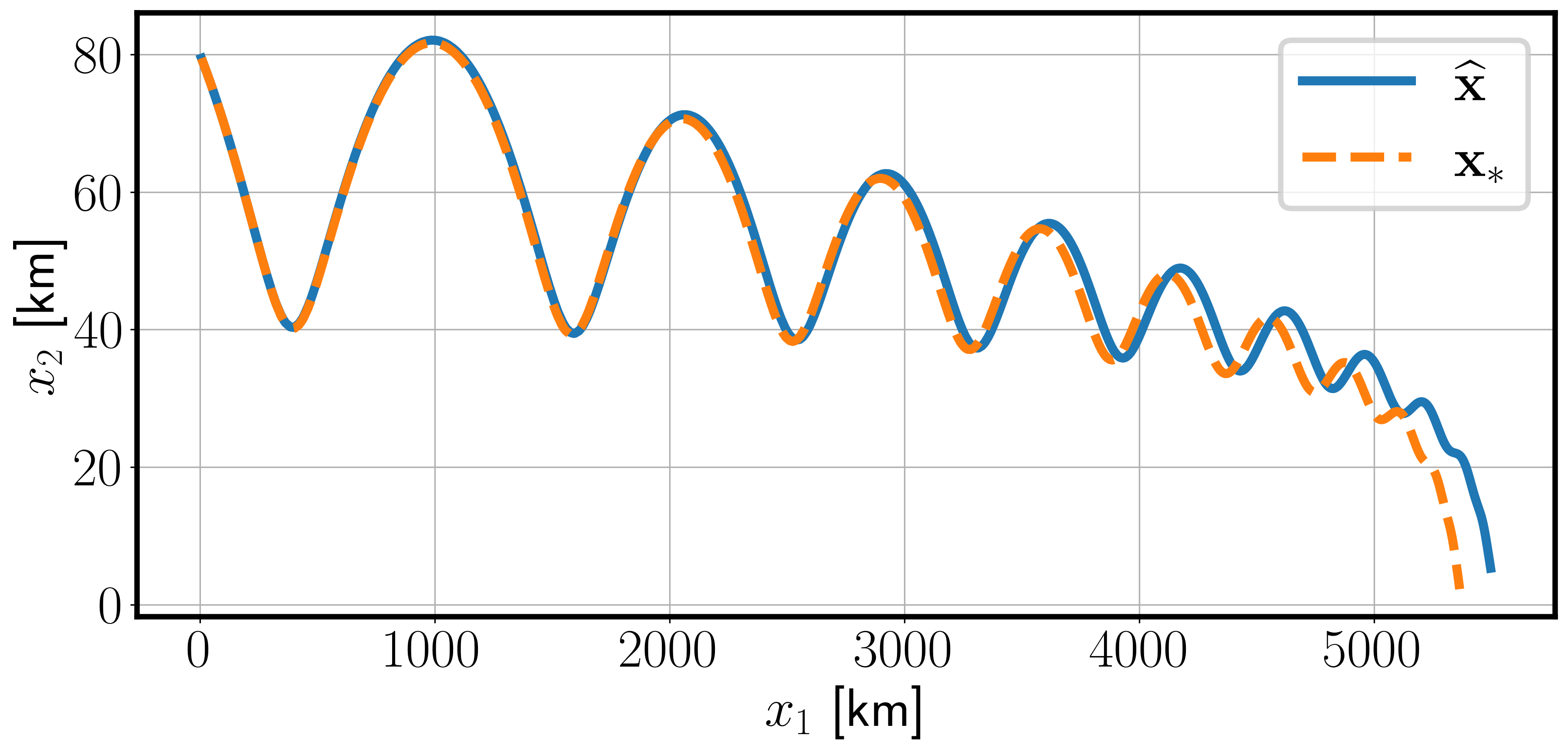}
\vspace*{-2ex}
\caption{Perturbed ($\epsilon = 0.01$) and true trajectories for hypersonic ODE.} \label{fig:hypersonic-trajectories}
\end{figure}

\begin{figure}[!htb]
\includegraphics[width=0.49\textwidth]{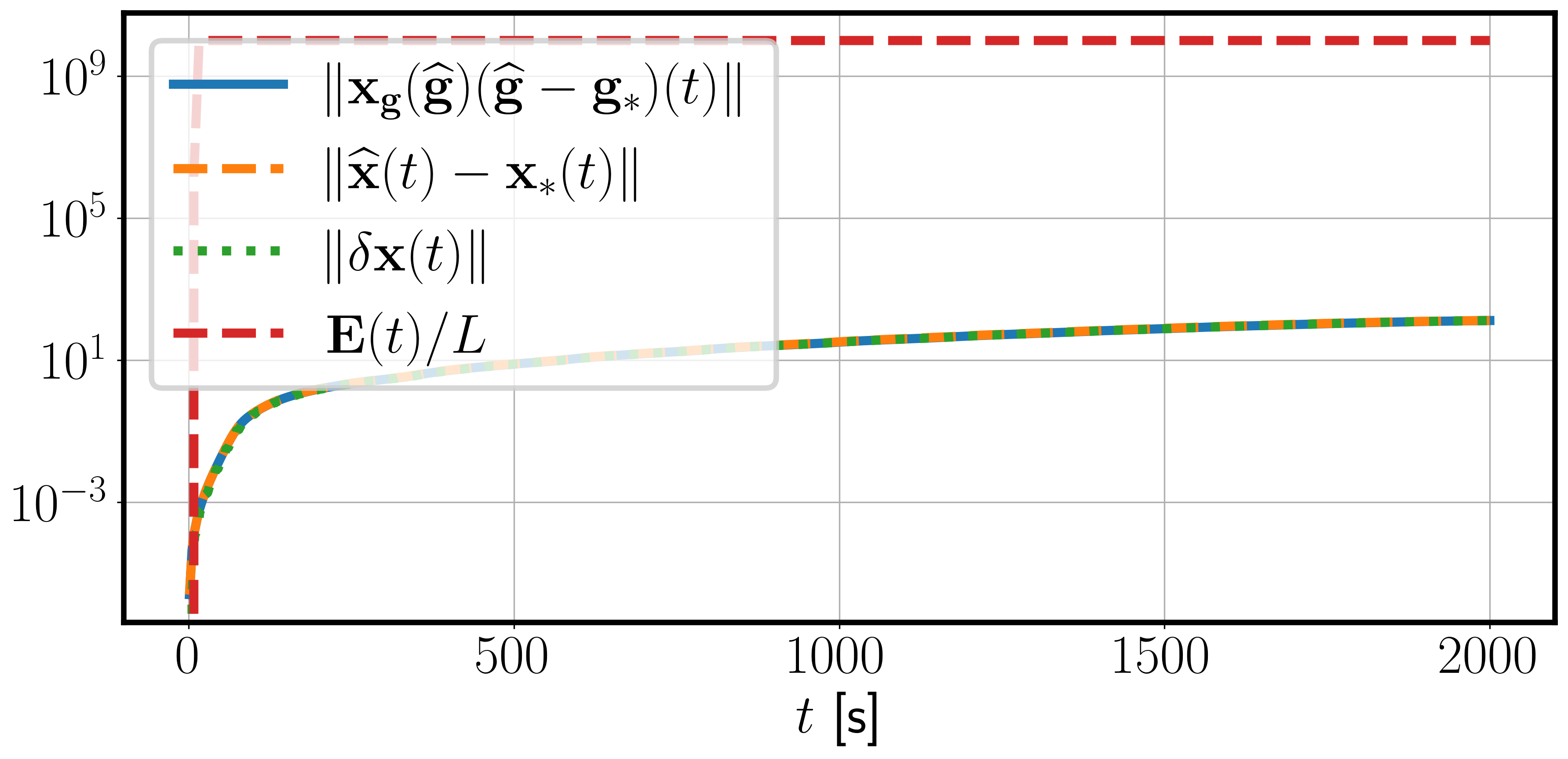} \hfill
\includegraphics[width=0.49\textwidth]{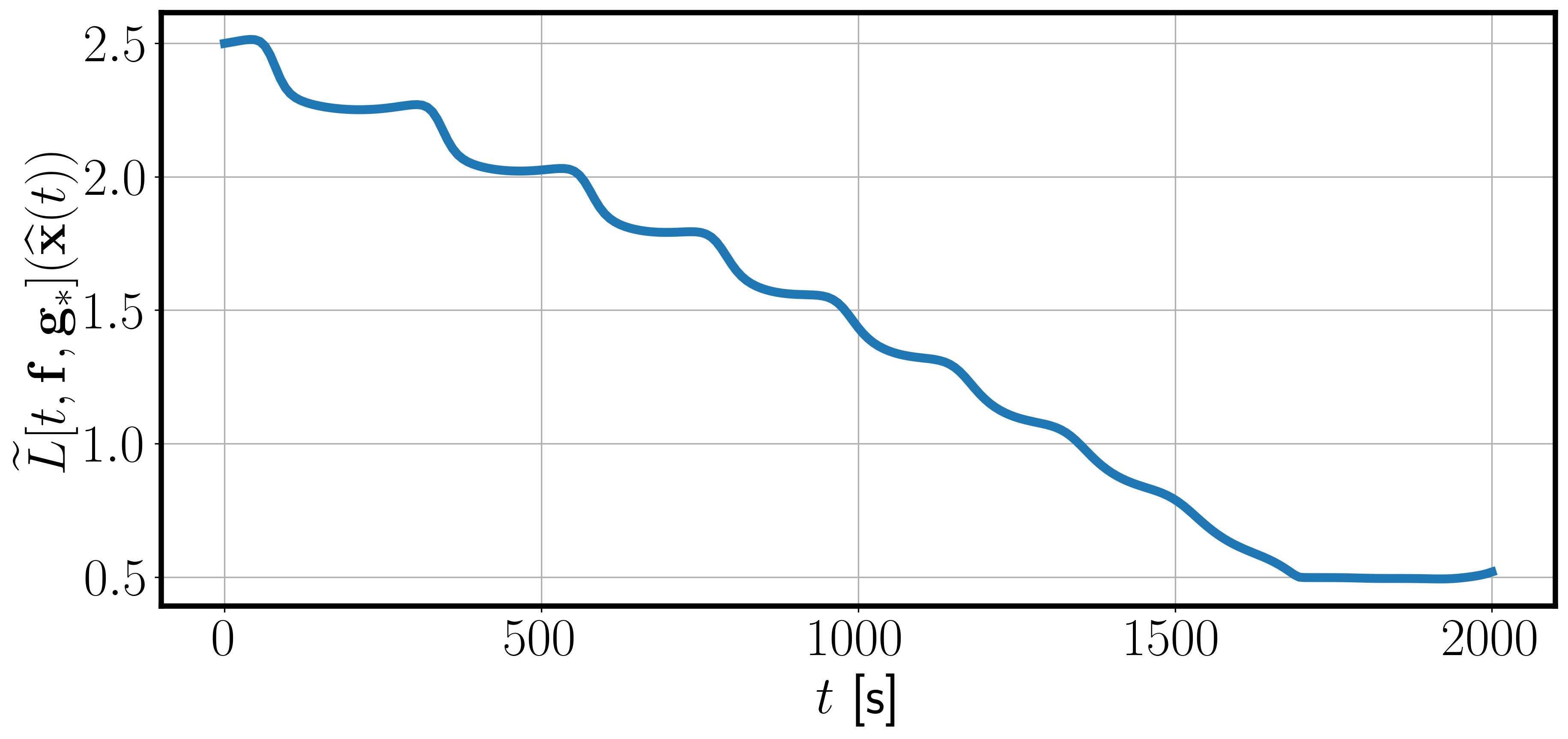}

\caption{Left: Sensitivity-based trajectory error estimate (blue) and trajectory error (orange) in good agreement, and trajectory error bound (green) yields a much tighter bound on the trajectory error than the Gronwall-type error bound (red, capped at $10^{10}$).
              Right: Approximate logarithmic Lipschitz constant along perturbed trajectory for hypersonic ODE.}
\label{fig:error-estimates-hypersonic}
\end{figure}

We compute the Gronwall and sensitivity-based error bounds for the ODE solution. A Lipschitz constant $L$ satisfying \cref{eq:uniformly-lipschitz} is more difficult to identify, so we consider $\BE(t) / L$ instead for simplicity. 
As the left plot in \cref{fig:error-estimates-hypersonic} shows, $\BE(t) / L$ grows extremely fast over time and reaches the $10^{10}$ cap within the first few seconds. It is clear from these results that no matter the value of $L$, \cref{thm:error-estimate-comparison-lemma} yields an error bound that is far too conservative to be useful for this problem.
The reason for the pessimistic Gronwall-type error bound is that the approximate logarithmic Lipschitz constant evaluated along the trajectory, i.e.,
$\widetilde{L}[t, \bff, \bg_*]\big( \widehat{\bx}(t) \big)$ where $\widetilde{L}$ is as defined in \cref{eq:logarithmic-norm-local-deriv} with respect to the 2-norm,
is positive-valued over a long time interval; see the right plot in \cref{fig:error-estimates-hypersonic}.

The left plot in \cref{fig:hypersonic-epsilon-study} shows the effect of the perturbation parameter $\epsilon$ on the $L^2$-error of the trajectory and the sensitivity-based estimate of the trajectory error, as well as the sensitivity-based upper bound.
\begin{figure}[!htb]
\includegraphics[width=0.48\textwidth]{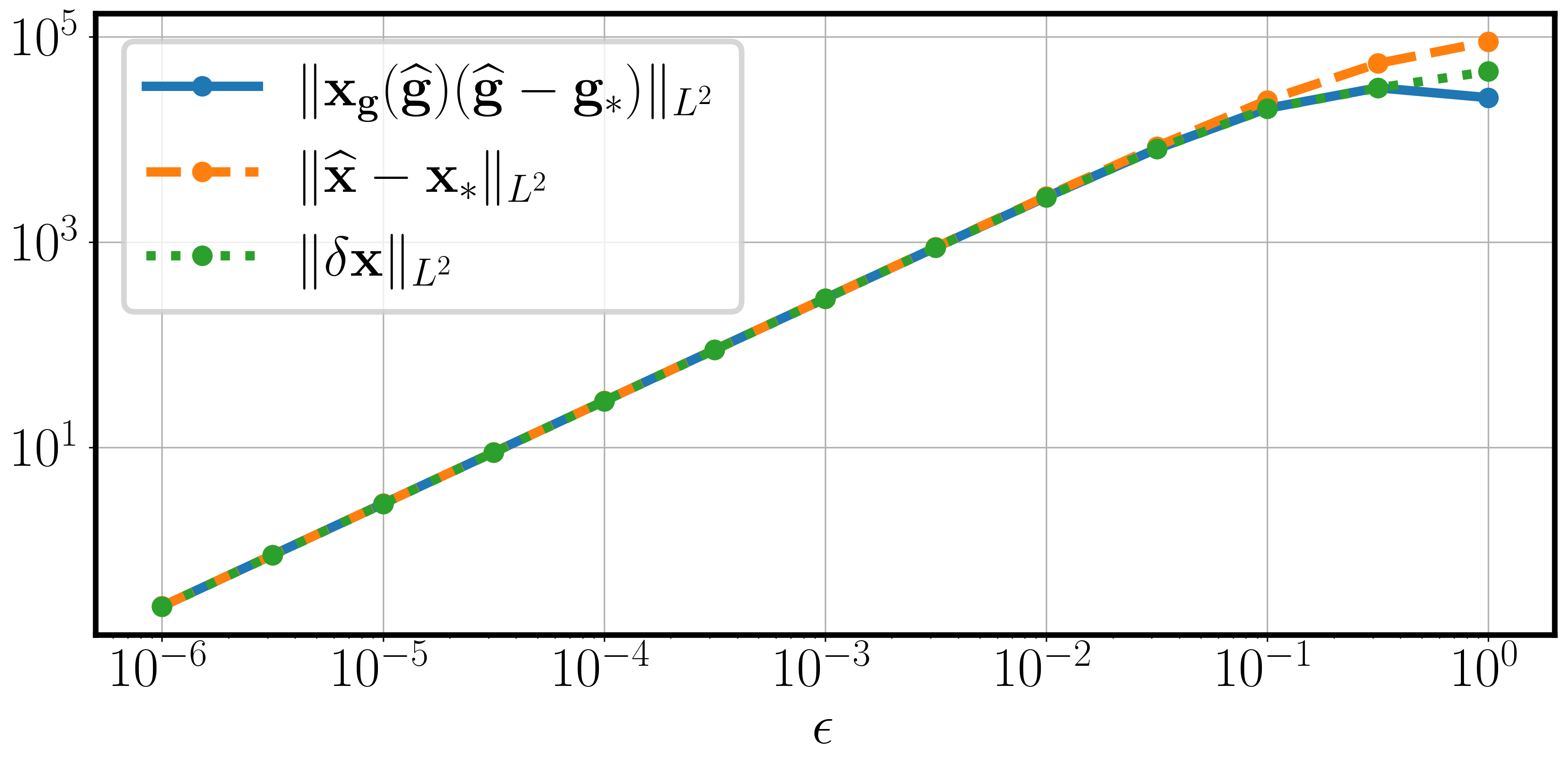} \hfill
\includegraphics[width=0.50\textwidth]{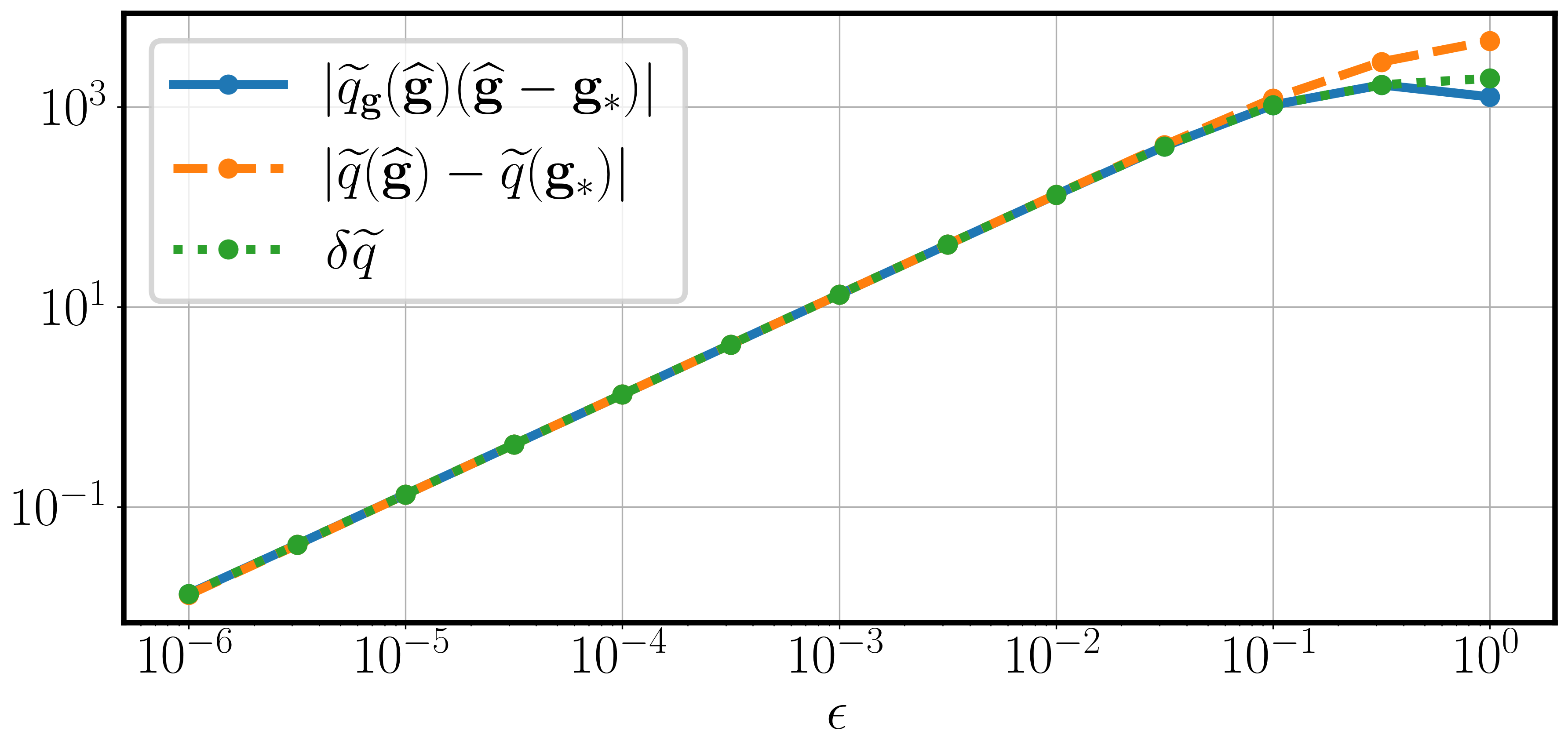}

\caption{Left: Strong agreement between the $L^2$ trajectory error estimates (blue), trajectory errors (orange), and error bounds (green)
              for hypersonic ODE for a range of perturbations $\epsilon$.
              Right: Strong agreement between the QoI error estimates (blue), QoI errors (orange), and error bounds (green) for hypersonic ODE
               for a range of perturbations $\epsilon$.}
       \label{fig:hypersonic-epsilon-study}
\end{figure}
\cref{thm:state-error-bound} yields a tight upper bound on the sensitivity-based estimate of the trajectory error, although for larger
perturbations $\epsilon \in [10^{-1}, 10^0]$ the trajectory error is visibly underestimated by the sensitivity-based error bound. This is expected, as the sensitivity-based error estimate is more reliable when the model perturbation is small.

Next, we consider the sensitivity of a quantity of interest. We consider the vehicle downrange:
\[
	\widetilde{q}(\bg) = \bx_1(t_f).
\]
The true QoI error $| \widetilde{q}(\widehat{\bg}) - \widetilde{q}(\bg_*) |$, the sensitivity-based estimate $| \widetilde{q}_\bg(\widehat{\bg}) (\widehat{\bg} - \bg_*) |$, and the upper bound 
$\delta \widetilde{q} := \int_{t_0}^{t_f}  \big|  \widehat{\BB}(t)^T \widehat{\blambda}(t) + \nabla_g l \big( t, \widehat{\bx}(t), \widehat{\bg} \big( t, \widehat{\bx}(t) \big) \big)  \big|^T  \bepsilon\big( t, \widehat{\bx}(t) \big) \, dt$
 of \cref{thm:QoI-error-bound} were computed for several values of the perturbation parameter $\epsilon$ and are given in the right plot in \cref{fig:hypersonic-epsilon-study}. 
 The sensitivity-based estimate is reliable and the approximate upper bound is tight for small $\epsilon$.
 
 The supplement presents numerical results for two additional examples.

\section{Conclusions and Future Work} \label{sec:conclusion}

We have applied the Implicit Function Theorem in an appropriate function space setting to obtain rigorous sensitivity results for the solution of an ODE with respect to a 
state-dependent component function, and we used these findings to develop sensitivity-based error estimates and bounds for the ODE solution and for a quantity of interest that depends on the ODE solution.
The sensitivity-based error bound for the ODE solution was shown to significantly outperform error bounds from classical ODE perturbation theory. 

The sensitivity-based upper bounds for the ODE solution error and QoI error given (resp.) by \cref{thm:state-error-bound} and 
\cref{thm:QoI-error-bound} give a computationally inexpensive way to assess the quality of the computed solution to \eqref{eq:IVP} 
when an approximation $\widehat{\bg}$ is used in place of the true model $\bg_*$ and an inexpensive pointwise error bound is available for $| \widehat{\bg} - \bg_* |$. 
This can be leveraged to develop a \emph{sensitivity-driven} method for adaptively constructing surrogate models from high-fidelity data, which can be used to simulate a dynamical system using surrogates while still ensuring a high-quality solution; this will be explored in a forthcoming paper. Future work will also focus on obtaining sensitivity analysis results for solutions of optimal control problems with surrogates entering in the dynamics. The function spaces used for sensitivity analysis in this paper were chosen with these future extensions in mind.

\begin{appendix}

\section[Fr\'echet Differentiability of Nemytskii Operators]{Fr\'echet Differentiability of Nemytskii Operators} \label{sec:sensitivity_proof}

\subsection*{Proof of~\cref{thm:F-Frechet_point}}
    Let $(\bx, \bg) \in (L^\infty(I))^{n_x} \times \big(\cG^1(I)\big)^{n_g}$ where $\bg$ satisfies \cref{eq:Lipschitz_g_x_local}, and define
    \begin{align*}
    	R_{L^\infty} &:= \| \bx \|_{(L^\infty(I))^{n_x}}, & 
	R_\cG := \| \bg \|_{(\cG^1(I))^{n_g}}.
    \end{align*}
    We will need the local Lipschitz properties related to \cref{as:assumption1}~(iii) and  \cref{eq:Lipschitz_g_x_local}
    for $x_1, x_2 \in \cB_{2 R_{L^\infty}}(0)$ and $g_1, g_2 \in \cB_{2 R_\cG}(0)$ along with some related properties.
    We will collect these first.
    
    The local Lipschitz property, \cref{as:assumption1}~(iii), implies that there exists $L^1_f > 0$ such that
    \begin{equation} \label{eq:Lipschitz_f'}
    \begin{aligned}    
    		\| \bff_x(t, x_1, g_1) - \bff_x(t, x_2, g_2) \| &+ \| \bff_g(t, x_1, g_1) - \bff_g(t, x_2, g_2) \| 
		\leq L^1_f ( \| x_1 - x_2 \| + \| g_1 - g_2 \| ), \\
		&\qquad \qquad \aall t \in I \mbox{ and all } x_1, x_2 \in \cB_{2 R_{L^\infty}}(0), g_1, g_2 \in \cB_{2 R_\cG}(0).
    \end{aligned}
    \end{equation}
    It follows from \cref{as:assumption1}~(ii), (iii) that
    \begin{equation}
    \begin{aligned}
        \|  \bff_x( t, x, g )  \|    
        & \le  \|  \bff_x( t, 0, 0 )  \|   +  ( \|  \bff_x( t, x, g )  \|   -   \|  \bff_x( t, 0, 0 )  \|  )
        \le K  + L^1_f \,  (  2 R_{L^\infty} +  2 R_\cG ), \\
        \|  \bff_g( t, x, g )  \|    
        & \le K  + L^1_f \, (  2 R_{L^\infty} +  2 R_\cG ) =: R_f, \\
        &\qquad \qquad \qquad \qquad \qquad \qquad \aall t \in I \mbox{ and all } x \in \cB_{2 R_{L^\infty}}(0), g \in \cB_{2 R_\cG}(0).
    \label{eq:bound_f'}
    \end{aligned}
    \end{equation}
    
    By \cref{eq:Lipschitz_g_x_local}, there exists $L_g^1 > 0$ such that
    \begin{equation}  \label{eq:Lipschitz_g_x}
    	\| \bg_x(t, x_1) - \bg_x(t, x_2) \| \leq L_g^1 \| x_1 - x_2 \|, 
	\quad \aall t \in I \;   \mbox{ and all }  x_1, x_2 \in \cB_{2 R_{L^\infty}}(0).
    \end{equation}
   Moreover, since $\bg$ has a bounded derivative, it is (globally) Lipschitz continuous with respect to $x$:
    \begin{equation*}
    	\| \bg(t, x_1) - \bg(t, x_2) \| \leq L_g \| x_1 - x_2 \|,  \quad  \aall t \in I \mbox{ and all }  x_1, x_2 \in \real^{n_x}.
    \end{equation*}
       
    Finally, define 
    \begin{equation*}
    	\bh: I \times \real^{n_x} \rightarrow  I \times  \real^{n_x} \times \real^{n_g}, \qquad
	(t,x) \mapsto \bh(t, x) := [t, x, \bg(t,x)].
    \end{equation*}
   Because $\bg$ is Lipschitz in $x$, $\bh$ is also Lipschitz in $x$:
    \begin{align}  \label{eq:Lipschitz_h}
    	\| \bh(t, x_1) - \bh(t, x_2) \| \leq L_h \| x_1 - x_2 \|,  \quad \aall t \in I \;   \mbox{ and all }  x_1, x_2 \in \real^{n_x}.
    \end{align}
    
    \sloppy Now, let  $\delta \bx \in \LL$ and $\delta \bg \in \big(\cG^1(I)\big)^{n_g}$ be functions satisfying 
    $\| \delta \bx \|_{(L^\infty(I))^{n_x}}\leq R_{L^\infty}$ and $\| \delta \bg \|_{(\cG^1(I))^{n_g}} \leq R_\cG$. 
    For all $s \in [0, 1]$ we have
    \begin{align*}
    	\| \bx + s\delta \bx \|_{(L^\infty(I))^{n_x}}&\leq \| \bx \|_{(L^\infty(I))^{n_x}}+ s \| \delta \bx \|_{(L^\infty(I))^{n_x}}\leq 2 R_{L^\infty}, \\
        \| \bg + s \delta \bg \|_{(\cG^1(I))^{n_g}} &\leq \| \bg \|_{(\cG^1(I))^{n_g}} + s \| \delta \bg \|_{(\cG^1(I))^{n_g}}  \leq 2 R_\cG,
    \end{align*}
    so that the Lipschitz and boundedness properties  \cref{eq:Lipschitz_f'}, \cref{eq:bound_f'}, \cref{eq:Lipschitz_g_x} 
    hold with $$x_1 = \bx(t), \quad x_2 = \bx(t) + s\delta \bx(t), \quad g_1 = \bg\big(t, \bx(t)\big), \quad g_2 = \bg\big(t, \bx(t) + s\delta \bx(t)\big).$$

    By definition of \cref{eq:F-superposition} and \cref{eq:F-Frechet_deriv} we have for  almost all $t \in I$,
    \begin{align}  \label{eq:r1-2}
       & \BF_1(\bx+\delta \bx, \bg + \delta \bg)(t)  - \BF_1(\bx, \bg)(t)  - [\BF_1'(\bx, \bg) (\delta \bx, \delta \bg)](t)   \nonumber \\
       &=\bff\Big( t, \bx(t) + \delta \bx(t), \bg\big(t, \bx(t) + \delta \bx(t) \big) \Big) 
                         -  \bff\Big( t, \bx(t), \bg\big(t, \bx(t) \big) \Big)        \nonumber \\
       &\quad \underbrace{ \qquad  -  \Big[ \bff_x\Big( t, \bx(t), \bg\big(t, \bx(t) \big) \Big) + \bff_g\Big( t, \bx(t), \bg\big(t, \bx(t) \big) \Big) \bg_x\big(t, \bx(t) \big) \Big] \delta \bx(t)  \quad }_{ := r_1[t]}      \\
       &\quad  + \bff\Big( t, \bx(t) + \delta \bx(t), (\bg + \delta \bg)\big(t, \bx(t) + \delta \bx(t) \big) \Big) 
                      -    \bff\Big( t, \bx(t) + \delta \bx(t), \bg\big(t, \bx(t) + \delta \bx(t) \big) \Big)             \nonumber \\
        & \quad \underbrace{ \qquad    - \bff_g\Big( t, \bx(t), \bg\big( t, \bx(t) \big) \Big) \delta \bg \big(t,  \bx(t) \big). \hspace*{40ex} \qquad }_{:= r_2[t]}   \nonumber 
    \end{align}
    The first remainder term in \cref{eq:r1-2} is
    \begin{align}
        r_1[t]
        &= \int_0^1 \bff_x\Big( t, \bx(t) + s \delta \bx(t), \bg\big(t, \bx(t) + s \delta \bx(t) \big) \Big)  \nonumber  \\
         &\hspace*{5ex} + \bff_g\Big( t, \bx(t) + s \delta \bx(t), \bg\big(t, \bx(t) + s \delta \bx(t) \big) \Big)  \bg_x\big(t, \bx(t) + s \delta \bx(t) \big)  \nonumber  \\
         &\hspace*{5ex}  -  \bff_x\Big( t, \bx(t), \bg\big(t, \bx(t) \big) \Big) 
                          -  \bff_g\Big( t, \bx(t), \bg\big(t, \bx(t) \big) \Big) \bg_x\big(t, \bx(t) \big)  \; ds \; \delta \bx(t)   \nonumber  \\
        &= \int_0^1  \bff_x\Big( t, \bx(t) + s \delta \bx(t), \bg\big(t, \bx(t) + s \delta \bx(t) \big) \Big) 
                                   -  \bff_x\Big( t, \bx(t), \bg\big(t, \bx(t) \big) \Big)  \nonumber  \\
        &\hspace*{5ex}      + \Bigg( \bff_g\Big( t, \bx(t) + s \delta \bx(t), \bg\big(t, \bx(t) + s \delta \bx(t) \big) \Big) 
                     -  \bff_g\Big( t, \bx(t), \bg\big(t, \bx(t) \big) \Big) \Bigg) \, \bg_x\big(t, \bx(t) \big)   \nonumber  \\
        &\hspace*{5ex}      + \bff_g\Big( t, \bx(t) + s \delta \bx(t), \bg\big(t, \bx(t) + s \delta \bx(t) \big) \Big)  
                    \Big( \bg_x\big(t, \bx(t) + s \delta \bx(t) \big) - \bg_x\big(t, \bx(t) \big) \Big) \, ds \; \delta \bx(t). \label{eq:r_1}
    \end{align}
    It follows from \cref{eq:Lipschitz_f'}, \cref{eq:bound_f'}, \cref{eq:Lipschitz_g_x}, and \cref{eq:Lipschitz_h} that  the remainder term in \cref{eq:r_1} satisfies
    \begin{align}  \label{eq:r_1_bound}
       \big\|  r_1[t]  \big\| 
        &\le \frac{  L_f^1 L_h + L_f^1 L_h R_\cG + R_f L_g^1  }{2} \; \| \delta \bx(t) \|^2      \nonumber \\
        &\leq \frac{  L_f^1 L_h + L_f^1 L_h R_\cG + R_f L_g^1  }{2} \; \| \delta \bx \|_{(L^\infty(I))^{n_x}}^2, \quad \aall t \in I.
    \end{align}
    
   The second remainder term in \cref{eq:r1-2} is
    \begin{align}  \label{eq:r_2}
       r_2[t] 
        &= \int_0^1 \bff_g\Big( t, \bx(t) + \delta \bx(t), \bg\big( t,  \bx(t) + \delta \bx(t) \big) + s \delta \bg\big( t,   \bx(t) + \delta \bx(t) \big) \Big)  \; ds  
               \; \delta \bg \big(  t,  \bx(t) + \delta \bx(t) \big)  \nonumber \\
         & \qquad -  \bff_g\Big( t, \bx(t), \bg\big( t,   \bx(t) \big) \Big) \delta \bg \big(  t,  \bx(t) \big)  \nonumber \\
         &= \int_0^1  \bff_g\Big( t, \bx(t) + \delta \bx(t), \bg\big(  t,  \bx(t) + \delta \bx(t) \big) + s \delta \bg\big(  t,   \bx(t) + \delta \bx(t) \big) \Big)   \nonumber \\
           &\underbrace{ \qquad      - \bff_g\Big( t, \bx(t) + \delta \bx(t), \bg\big(  t,  \bx(t) + \delta \bx(t) \big) \Big)     \; ds   \; \delta \bg \big(  t,  \bx(t) + \delta \bx(t) \big) \quad }_{ := r_{2, 1}[t]  } \nonumber \\
         &   + \underbrace{\Big[  \bff_g\Big( t, \bx(t) + \delta \bx(t), \bg\big( t,   \bx(t) + \delta \bx(t) \big) \Big)
                      - \bff_g \Big(t, \bx(t), \bg\big(  t,  \bx(t) \big) \Big) \Big] \; \delta \bg \big( t,  \bx(t) + \delta \bx(t) \big)}_{:=r_{2, 2}[t]}    \nonumber \\
          &  + \underbrace{\bff_g \Big(t, \bx(t), \bg\big(  t,  \bx(t) \big) \Big) \Big[ \delta \bg
                                          \big( t,   \bx(t) + \delta \bx(t) \big) -  \delta \bg \big(  t,   \bx(t) \big) \Big]}_{:=r_{2, 3}[t]}.
     \end{align}
     Using \cref{eq:Lipschitz_f'}, the remainder term $r_{2,1}[t]$ in \cref{eq:r_2} is bounded by
     \begin{equation} \label{eq:r_21_bound}
     	\big\|   r_{2,1}[t]   \big\| \leq \frac{L_f^1}{2} \big\| \delta \bg\big( t,  \bx(t) + \delta \bx(t)\big) \big\|^2 
	                                      \leq \frac{L_f^1}{2} \| \delta \bg \|_{(\cG^1(I))^{n_g}}^2, \quad \aall t \in I.
	\end{equation}
     Similarly, using \cref{eq:Lipschitz_f'} and \cref{eq:Lipschitz_h}, the $r_{2, 2}[t]$ term in \cref{eq:r_2} is bounded by
     \begin{equation}  \label{eq:r_22_bound}
     	\big\| r_{2, 2}[t] \big\| \leq L_f^1 L_h \, \| \delta \bx(t) \| \, \big\| \delta \bg\big(  t,  \bx(t) + \delta \bx(t) \big) \big\| 
	        \leq L_f^1 L_h \, \| \delta \bx \|_{(L^\infty(I))^{n_x}}\, \| \delta \bg \|_{(\cG^1(I))^{n_g}}, \quad \aall t \in I.
     \end{equation}
     To estimate  $r_{2, 3}[t]$ in \cref{eq:r_2} we first use \cref{eq:Lipschitz_g_x} to bound
     \begin{align*}
  & \big\|  \delta \bg\big( t,  \bx(t) + \delta \bx(t) \big)  - \delta \bg\big( t,   \bx(t) \big) \big\|      \nonumber \\
  & \le \Big\|  \int_0^1  \delta \bg_x\big( t,   \bx(t) + s \delta \bx(t)\big) - \delta \bg_x\big( t,   \bx(t)\big) \, ds \, \Big\|  \; \big\|  \delta \bx(t) \big\| 
	     + \big\|  \delta \bg_x \big(  t,   \bx(t) \big) \big\|  \, \big\|  \delta \bx(t) \big\|      \nonumber \\
	  & \le \frac{L_g^1}{2}   \| \delta \bx(t) \|^2  + \big\| \delta \bg_x \big(  t,  \bx(t) \big) \big\| \; \| \delta \bx(t) \|  
	     \le \frac{L_g^1}{2}   \| \delta \bx \|_{(L^\infty(I))^{n_x}}^2  + \| \delta \bg \|_{(\cG^1(I))^{n_g}}  \| \delta \bx \|_{(L^\infty(I))^{n_x}}.    
     \end{align*}
     Using this bound and \cref{eq:bound_f'} implies
     \begin{equation}  \label{eq:r_23_bound}
     	\big\| r_{2, 3}[t] \big\| 
	\leq R_f \Big( \frac{L_g^1}{2}   \| \delta \bx \|_{(L^\infty(I))^{n_x}}^2  + \| \delta \bg \|_{(\cG^1(I))^{n_g}}  \| \delta \bx \|_{(L^\infty(I))^{n_x}}\Big), \quad \aall t \in I.
    \end{equation}
    The bounds \cref{eq:r_21_bound}, \cref{eq:r_22_bound}, and \cref{eq:r_23_bound} imply the existence of a $c>0$ such that
     \begin{equation} \label{eq:r_2_bound}
     	\big\| r_2[t] \big\| \leq  c (  \| \delta \bg \|_{(\cG^1(I))^{n_g}} +  \| \delta \bx \|_{(L^\infty(I))^{n_x}})^2, \quad \aall t \in I.
     \end{equation}
     
     Finally, \cref{eq:r1-2} and the bounds \cref{eq:r_1_bound} and \cref{eq:r_2_bound} yield
    \[
           \lim_{ \| \delta \bx \|_{(L^\infty(I))^{n_x}}+ \| \delta \bg \|_{(\cG^1(I))^{n_g}} \rightarrow 0 }
            \frac{\| \BF_1(\bx + \delta \bx, \bg + \delta \bg) - \BF_1(\bx, \bg) - \BF_1'(\bx, \bg) (\delta \bx, \delta \bg) \|_{(L^\infty(I))^{n_x}}}
                    {\| \delta \bx \|_{(L^\infty(I))^{n_x}}+ \| \delta \bg \|_{(\cG^1(I))^{n_g}}} 
       = 0.
     \]
     
     To conclude the proof, we need to show that $(\delta \bx, \delta \bg) \mapsto \BF_1'(\bx, \bg) (\delta \bx, \delta \bg)$
     is a bounded linear operator from $\big(L^\infty(I)\big)^{n_x} \times \big(\cG^1(I)\big)^{n_g}$ to $ \LL$.
     This immediately follows from \cref{eq:bound_f'} and the essential boundedness of $\bg_x$ assured by definition of $\big(\cG^1(I)\big)^{n_g}$ in \cref{eq:g-space}.
\hfill $\Box$

\subsection*{Proof of \cref{thm:F-Frechet_global}}
     The Fr\'echet differentiability of $\BF_2$ at any point $(\bx, \bg) \in \big(L^\infty(I)\big)^{n_x} \times \big(\cG^2(I)\big)^{n_g} $ immediately
     follows from \cref{thm:F-Frechet_point} since
    $ \| \bg \|_{(\cG^1(I))^{n_g}} \le  \| \bg \|_{(\cG^2(I))^{n_g} }$ for all $\bg \in \big(\cG^2(I)\big)^{n_g} $ and
     the boundedness of $\bg_{xx}$ ensures that $\bg_x$ is locally (and in fact globally) Lipschitz, 
     so \cref{eq:Lipschitz_g_x_local} is satisfied for all $\bg \in \big(\cG^2(I)\big)^{n_g}$.
     
     \sloppy
     Next, we show that the (global) Fr\'echet derivative 
     \[
          \BF_2' : \big(L^\infty(I)\big)^{n_x} \times \big(\cG^2(I)\big)^{n_g}  \rightarrow \cL\Big( \big(L^\infty(I)\big)^{n_x} \times \big(\cG^2(I)\big)^{n_g}, \big(L^\infty(I)\big)^{n_x} \Big)
     \]
      is a continuous map. 
     Let $(\overline{\bx}, \overline{\bg}) \in \big(L^\infty(I)\big)^{n_x} \times \big(\cG^2(I)\big)^{n_g} $ be given. 
     Given $\delta > 0$, let $(\bx, \bg) \in \big(L^\infty(I)\big)^{n_x} \times \big(\cG^2(I)\big)^{n_g} $ satisfy 
     $\| \bx - \overline{\bx} \|_{(L^\infty(I))^{n_x}}+ \| \bg - \overline{\bg} \|_{(\cG^2(I))^{n_g}} < \delta$.
     To apply the bounds  \cref{eq:Lipschitz_f'}, \cref{eq:bound_f'}, \cref{eq:Lipschitz_g_x}, \cref{eq:Lipschitz_h}, assume that
    \[
    	 \| \bx \|_{(L^\infty(I))^{n_x}}, \; \| \overline{\bx} \|_{(L^\infty(I))^{n_x}}\le  R_{L^\infty}, \qquad
	  \| \bg \|_{(\cG^1(I))^{n_g}}, \; \| \overline{\bg} \|_{(\cG^1(I))^{n_g}}  \le  R_\cG.
    \]

     We have
     \begin{equation}
      \begin{aligned}
     	&\| \BF_2'(\bx, \bg) - \BF_2'(\overline{\bx}, \overline{\bg}) \|_{\cL\big( (L^\infty(I))^{n_x} \times (\cG^2(I))^{n_g}, (L^\infty(I))^{n_x} \big)} \\
     	&\quad = \sup_{\| \delta \bx \|_{(L^\infty(I))^{n_x}}+ \| \delta \bg \|_{(\cG^2(I))^{n_g} } = 1}  \esssup_{t \in I} \Big\| \Big[ \bff_x \Big(t, \bx(t), \bg \big(t, \bx(t) \big) \Big) - \bff_x \Big(t, \overline{\bx}(t), \overline{\bg} \big(t, \overline{\bx}(t) \big) \Big) \Big] \delta \bx(t) \\ 
     	&\qquad+ \Big[ \bff_g \Big(t, \bx(t), \bg \big(t, \bx(t) \big) \Big) \bg_x \big(t, \bx(t) \big) - \bff_g \Big(t, \overline{\bx}(t), \overline{\bg} \big(t, \overline{\bx}(t) \big) \Big) \overline{\bg}_x \big(t, \overline{\bx}(t) \big) \Big] \delta \bx(t) \\
     	&\qquad+ \bff_g \Big(t, \bx(t), \bg \big(t, \bx(t) \big) \Big) \delta \bg \big(t, \bx(t) \big) - \bff_g \Big(t, \overline{\bx}(t), \overline{\bg}(t, \overline{\bx}(t) \big) \Big) \delta \bg \big(t, \overline{\bx}(t) \big)
     	\Big\| \\
     	&\quad \leq \sup_{\| \delta \bg \|_{(\cG^2(I))^{n_g} } = 1}  \esssup_{t \in I} \underbrace{\Big\| \bff_x \Big(t, \bx(t), \bg \big(t, \bx(t) \big) \Big) - \bff_x \Big(t, \overline{\bx}(t), \overline{\bg} \big(t, \overline{\bx}(t) \big) \Big) \Big\|}_{:= S_1[t]} \\ 
     	&\qquad+ \underbrace{\Big\| \bff_g \Big(t, \bx(t), \bg \big(t, \bx(t) \big) \Big) \bg_x \big(t, \bx(t) \big) - \bff_g \Big(t, \overline{\bx}(t), \overline{\bg} \big(t, \overline{\bx}(t) \big) \Big) \overline{\bg}_x \big(t, \overline{\bx}(t) \big) \Big\|}_{:= S_2[t]} \\
     	&\qquad+ \underbrace{\Big\| \bff_g \Big(t, \bx(t), \bg \big(t, \bx(t) \big) \Big) \delta \bg \big(\bx(t) \big) - \bff_g \Big(t, \overline{\bx}(t), \overline{\bg}(t, \overline{\bx}(t) \big) \Big) \delta \bg \big(t, \overline{\bx}(t) \big) \Big\|}_{:= S_3[t]}. \label{eq:S1-3}
     \end{aligned}
     \end{equation}
     The term $S_1[t]$ in \cref{eq:S1-3} is bounded for almost all $t \in I$ using \cref{eq:Lipschitz_f'} and \cref{eq:Lipschitz_h} by
     \begin{align}   \label{eq:S_1[t]_bound}
     	S_1[t] 
     	&\leq \Big\| \bff_x \Big(t, \bx(t), \bg \big(t, \bx(t) \big) \Big) - \bff_x \Big(t, \overline{\bx}(t), \bg \big(t, \overline{\bx}(t) \big) \Big) \Big\| \nonumber \\
     	& \qquad  + \Big\| \bff_x \Big(t, \overline{\bx}(t), \bg \big(t, \overline{\bx}(t) \big) \Big) - \bff_x \Big(t, \overline{\bx}(t), \overline{\bg} \big(t, \overline{\bx}(t) \big) \Big) \Big\|  \nonumber \\
     	&\leq L_f^1  L_h \, \| \bx(t) - \overline{\bx}(t) \| + L_f^1 \, \big\| \bg\big(t, \overline{\bx}(t)\big) - \overline{\bg}\big(t, \overline{\bx}(t) \big) \big\| 
	< (L_f^1  L_h + L_f^1) \, \delta. 
     \end{align}
     Next, we use \cref{eq:Lipschitz_f'}, \cref{eq:bound_f'}, \cref{eq:Lipschitz_g_x}, \cref{eq:Lipschitz_h}, and 
     $\| \bg_x(t, x) \| \le  \| \bg \|_{(\cG^1(I))^{n_g}} \le R_\cG$ to 
      bound $S_2[t]$ in \cref{eq:S1-3} for almost all $t \in I$:
     \begin{equation}
     \begin{aligned}
     	S_2[t] &= \Big\| \bff_g \Big(t, \bx(t), \bg \big(t,  \bx(t) \big) \Big) \bg_x \big(\bx(t) \big) - \bff_g \Big(t, \overline{\bx}(t), \overline{\bg} \big(t,  \overline{\bx}(t) \big) \Big) \overline{\bg}_x \big(t,  \overline{\bx}(t) \big) \Big\| \\
     	&\leq \Big\| \bff_g \Big(t, \bx(t), \bg \big(t,  \bx(t) \big) \Big) \bg_x \big(t, \bx(t) \big) - \bff_g \Big(t, \overline{\bx}(t), \bg \big(t, \overline{\bx}(t) \big) \Big) \bg_x \big(t, \bx(t) \big) \Big\| \\
     	&\quad + \Big\| \bff_g \Big(t, \overline{\bx}(t), \bg \big(t,  \overline{\bx}(t) \big) \Big) \bg_x \big(t, \bx(t) \big) - \bff_g \Big(t, \overline{\bx}(t), \overline{\bg} \big(t,  \overline{\bx}(t) \big) \Big) \bg_x \big(t, \bx(t) \big) \Big\| \\
     	&\quad + \Big\| \bff_g \Big(t, \overline{\bx}(t), \overline{\bg} \big(t,  \overline{\bx}(t) \big) \Big) \bg_x \big(t,  \bx(t) \big) - \bff_g \Big(t, \overline{\bx}(t), \overline{\bg} \big(t,  \overline{\bx}(t) \big) \Big) \bg_x \big(t,  \overline{\bx}(t) \big) \Big\| \\
     	&\quad + \Big\| \bff_g \Big(t, \overline{\bx}(t), \overline{\bg} \big(t,  \overline{\bx}(t) \big) \Big) \bg_x \big(t,  \overline{\bx}(t) \big) - \bff_g \Big(t, \overline{\bx}(t), \overline{\bg} \big(t,  \overline{\bx}(t) \big) \Big) \overline{\bg}_x \big(t,  \overline{\bx}(t) \big) \Big\| \\
     	&\leq L_f^1  L_h  R_\cG \, \| \bx(t) - \overline{\bx}(t) \| 
	         + L_f^1  R_\cG \, \big\| \bg\big(t,  \overline{\bx}(t) \big) - \overline{\bg} \big( t, \overline{\bx}(t) \big) \big\| \\ 
	&\quad + R_f  L_g^1 \, \| \bx(t) - \overline{\bx}(t) \| + R_f \, \big\| \bg_x\big(t, \overline{\bx}(t) \big) - \overline{\bg}_x\big( t, \overline{\bx}(t) \big) \big\| \\
     	&< (L_f^1  L_h  R_\cG + L_f^1  R_\cG + R_f  L_g^1 + R_f) \delta. \label{eq:S_2[t]_bound}
     \end{aligned}
     \end{equation} 
      To obtain a bound on the term $S_3[t]$ in \cref{eq:S1-3}, we first note that if $\| \delta \bg \|_{(\cG^2(I))^{n_g} } = 1$, then $\delta \bg$ is Lipschitz continuous in $x$ with Lipschitz constant $1$ due to the boundedness properties of $\delta \bg_x$:
     \begin{equation} \label{eq:Lipschitz_delta_g}
     	\| \delta \bg(t, x_1) - \delta \bg(t, x_2) \| \leq \| x_1 - x_2 \|, \quad \aall t \in I  \mbox{ and all }  x_1, x_2 \in \real^{n_x}.
  \end{equation}
    Using \cref{eq:Lipschitz_f'}, \cref{eq:bound_f'}, \cref{eq:Lipschitz_h}, and \cref{eq:Lipschitz_delta_g} we obtain the bound
     \begin{equation}
     \begin{aligned}
     	S_3[t] &= \Big\| \bff_g \Big(t, \bx(t), \bg \big(t, \bx(t) \big) \Big) \delta \bg \big(t, \bx(t) \big) - \bff_g \Big(t, \overline{\bx}(t), \overline{\bg} \big(t, \overline{\bx}(t) \big) \Big) \delta \bg \big(t, \overline{\bx}(t) \big) \Big\| \\
     	&\leq \Big\| \bff_g \Big(t, \bx(t), \bg \big(t, \bx(t) \big) \Big) \delta \bg \big(t, \bx(t) \big) - \bff_g \Big(t, \bx(t), \bg \big(t, \bx(t) \big) \Big) \delta \bg \big(t, \overline{\bx}(t) \big) \Big\| \\
     	&\quad + \Big\| \bff_g \Big(t, \bx(t), \bg \big(t, \bx(t) \big) \Big) \delta \bg \big(t, \overline{\bx}(t) \big) - \bff_g \Big(t, \overline{\bx}(t), \bg \big(t, \overline{\bx}(t) \big) \Big) \delta \bg \big(t, \overline{\bx}(t) \big) \Big\| \\
     	&\quad + \Big\| \bff_g \Big(t, \overline{\bx}(t), \bg \big(t, \overline{\bx}(t) \big) \Big) \delta \bg \big(t, \overline{\bx}(t) \big) - \bff_g \Big(t, \overline{\bx}(t), \overline{\bg} \big(t, \overline{\bx}(t) \big) \Big) \delta \bg \big(t, \overline{\bx}(t) \big)  \Big\| \\
     	&\leq R_f \, \| \bx(t) - \overline{\bx}(t) \| + L_f^1  L_h \, \| \bx(t) - \overline{\bx}(t) \| + L_f^1 \, \big\| \bg\big(t,  \overline{\bx}(t) \big) - \overline{\bg}\big(t,  \overline{\bx}(t) \big) \big\| \\
     	&< (R_f + L_f^1  L_h + L_f^1) \, \delta \label{eq:S_3[t]_bound}
     \end{aligned}
     \end{equation}
     for any $\delta \bg$ satisfying $\| \delta \bg \|_{(\cG^2(I))^{n_g} } = 1$ and almost all $t \in I$.
     
     Inserting \cref{eq:S_1[t]_bound}, \cref{eq:S_2[t]_bound}, and \cref{eq:S_3[t]_bound} into \cref{eq:S1-3} yields
     \begin{equation*}
     	\| \BF_2'(\bx, \bg) - \BF_2'(\overline{\bx}, \overline{\bg}) \|_{\cL\big( (L^\infty(I))^{n_x} \times (\cG^2(I))^{n_g},\LL\big)}
     	< C \delta
     \end{equation*}
     where
     $C = (L_f^1  L_h + L_f^1) + (L_f^1  L_h  R_\cG + L_f^1  R_\cG + R_f  L_g^1 + R_f) + (R_f + L_f^1  L_h + L_f^1)$.
          Therefore, for any $\epsilon > 0$, taking $\delta = \epsilon/C$ ensures that 
     $\| \bx - \overline{\bx} \|_{(L^\infty(I))^{n_x}}+ \| \bg - \overline{\bg} \|_{(\cG^2(I))^{n_g} } < \delta$ implies 
    \[
          \| \BF_2'(\bx, \bg) - \BF_2'(\overline{\bx}, \overline{\bg}) \|_{\cL\big( (L^\infty(I))^{n_x} \times (\cG^2(I))^{n_g}, (L^\infty(I))^{n_x}\big)}  < C \delta = \epsilon.
      \]
     Since $(\overline{\bx}, \overline{\bg}) \in \big(L^\infty(I)\big)^{n_x} \times \big(\cG^2(I)\big)^{n_g}$ was arbitrary, $\BF_2$ is continuously Fr\'echet differentiable
     on $\big(L^\infty(I)\big)^{n_x} \times \big(\cG^2(I)\big)^{n_g}$. Furthermore, since $\delta$ depends linearly on $\epsilon$, we have shown that the Fr\'echet derivative is locally Lipschitz continuous (but not globally since $C$ depends on $R_{L^\infty}$ and $R_\cG$, which depend on $\bx, \overline{\bx}$ and $\bg, \overline{\bg}$).
\hfill $\Box$
 
\end{appendix}


\newpage

\setcounter{section}{0}
\renewcommand\thesection{SM\arabic{section}}

\setcounter{figure}{0}

\setcounter{page}{1}
\renewcommand{\thepage}{SM.\arabic{page}}

\begin{center}
{\bf SUPPLEMENTARY MATERIALS: SENSITIVITY OF ODE SOLUTIONS AND QUANTITIES OF
INTEREST WITH RESPECT TO COMPONENT FUNCTIONS IN THE DYNAMICS } 

\medskip
JONATHAN R.\ CANGELOSI AND MATTHIAS HEINKENSCHLOSS 

\medskip
\end{center}

In this supplement we present numerical results for two other ODE systems:
Zermelo's problem in \cref{sec:numerics_zermelo} and the 
Hodgkin-Huxley model describing the membrane excitability of the squid giant axon 
in \cref{sec:numerics_hh}.

\section{Zermelo's Problem} \label{sec:numerics_zermelo}
We present numerical results for Zermelo's problem. 
The Zermelo problem models the trajectory of a boat moving downstream through a river with a current whose strength depends on the boat's position. 
We consider a particular instance of the Zermelo problem where the strength of the current depends on a function $\bg$ of the boat's horizontal position:
\begin{equation} \label{eq:numerics_zermelo:zermelo_ODE}
\begin{aligned}
	\bx_1'(t) &= \cos \bu(t) + \bg\big(\bx_1(t)\big) \bx_2(t), & t \in (0, 1), \\
	\bx_2'(t) &= \sin  \bu(t), & t \in (0, 1), \\
	\bx_1(0) &= \bx_2(0) = 0,
\end{aligned}
\end{equation}
where time $t$ is assumed dimensionless, the state $\bx(t) = \big(\bx_1(t), \bx_2(t)\big)$ is the boat's position (also dimensionless), and $\bu(t)$ is the boat's heading angle (in radians), which is a given input.
For this example we use $\bu(t) = (1 - 2t) \pi / 3$.
We suppose that the ``true'' function $\bg_*$ is given by
\[
	\bg_*(x_1) = 2 + 10x_1 - (x_1 - 2)^3.
\]

The goal is to solve \eqref{eq:numerics_zermelo:zermelo_ODE} with $\bg = \bg_*$, but to illustrate our error estimate we instead
solve \eqref{eq:numerics_zermelo:zermelo_ODE} using an approximation $\widehat{\bg}$ of $\bg_*$ given by
\[
	\widehat{\bg}(x_1) = 2 + 10x_1 - (1 - \epsilon) (x_1 - 2)^3
\]
where $\epsilon > 0$ is small. Accordingly, $\widehat{\bg}$ may be regarded as a perturbation of $\bg_*$ by 
\begin{equation} \label{eq:delta-g-zermelo}
	\delta \bg(x_1) = \widehat{\bg}(x_1) - \bg_*(x_1) = \epsilon (x_1 - 2)^3.
\end{equation}
The perturbed trajectory $\widehat{\bx}$ and the true trajectory $\bx_*$ are shown in \cref{fig:zermelo-trajectories} for $\epsilon = 0.1$.
\begin{figure}[!htb]
\centering
\includegraphics[width=0.5\textwidth]{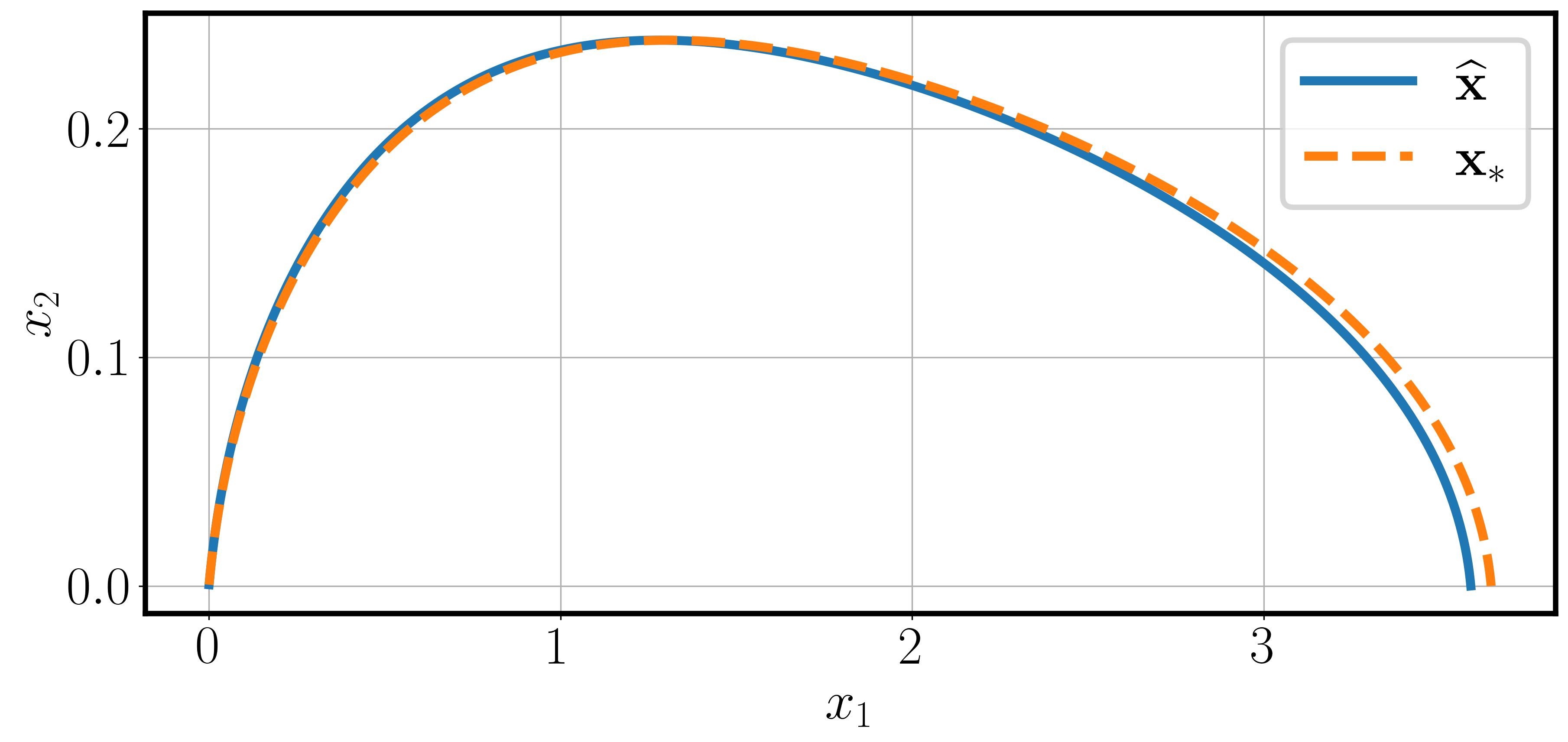}
\vspace*{-2ex}
\renewcommand{\thefigure}{SM.1}
\caption{Solutions to the true and perturbed Zermelo ODE for $\epsilon = 0.1$.}
\label{fig:zermelo-trajectories}
\end{figure}

Next, we compute the Gronwall error bound and the sensitivity-based approximate error bound.
To obtain a Lipschitz constant $L$ that satisfies \eqref{eq:uniformly-lipschitz} in the assumptions of \cref{thm:error-estimate-comparison-lemma}, we first observe from \cref{fig:zermelo-trajectories} that $R_{L^\infty} := \| \widehat{\bx} \|_{(L^\infty(I))^{n_x}} \leq 4$, which implies
\[
	\| \bff(t, x, g_1) - \bff(t, x, g_2) \|_2
	= \left\| \begin{bmatrix} (g_1 - g_2) x_2 \\ 0 \end{bmatrix} \right\|_2
	 \leq 4 \, | g_1 - g_2 | \quad \mbox{ for all } t \in I, \; x \in \cB_{R_{L^\infty}}(0), \; g_1, g_2 \in \real.
\]
Thus, $L = 4$ satisfies \eqref{eq:uniformly-lipschitz} in this example.

For the sensitivity-based bound we use
\begin{equation} \label{eq:model-error-bound-zermelo}
	\bepsilon\big(t, \widehat{\bx}(t)\big) := \big| \delta \bg \big( \widehat{\bx}_1(t) \big) \big|
\end{equation}
 in \eqref{eq:Refinement:ODE:LQOCP},
i.e., we set the model error bound equal to the absolute model error along the nominal (perturbed) trajectory.
The results for the two error bounds are given in the left plot in \cref{fig:error-estimates-zermelo}. 
\begin{figure}[!htb]
\includegraphics[width=0.49\textwidth]{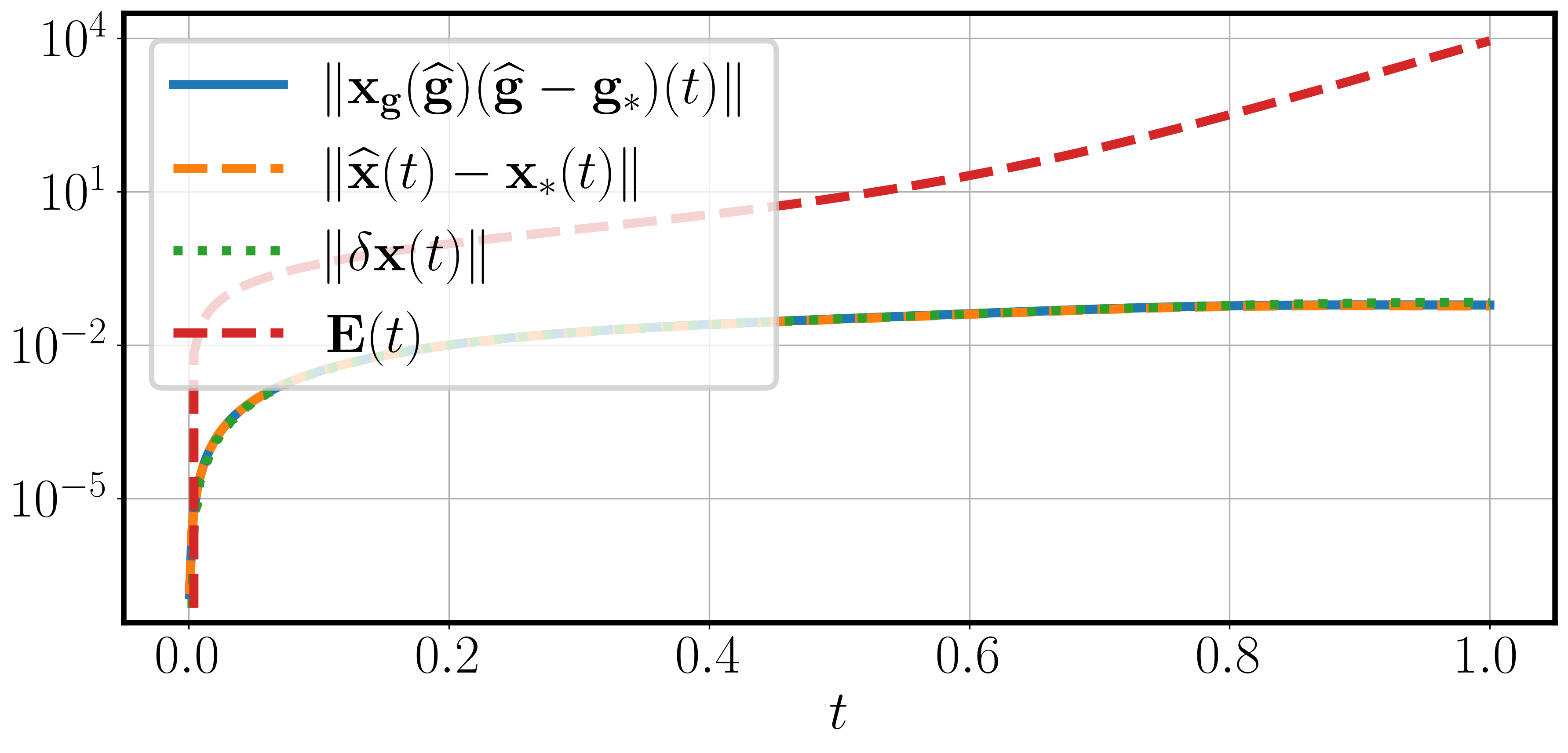}
\includegraphics[width=0.49\textwidth]{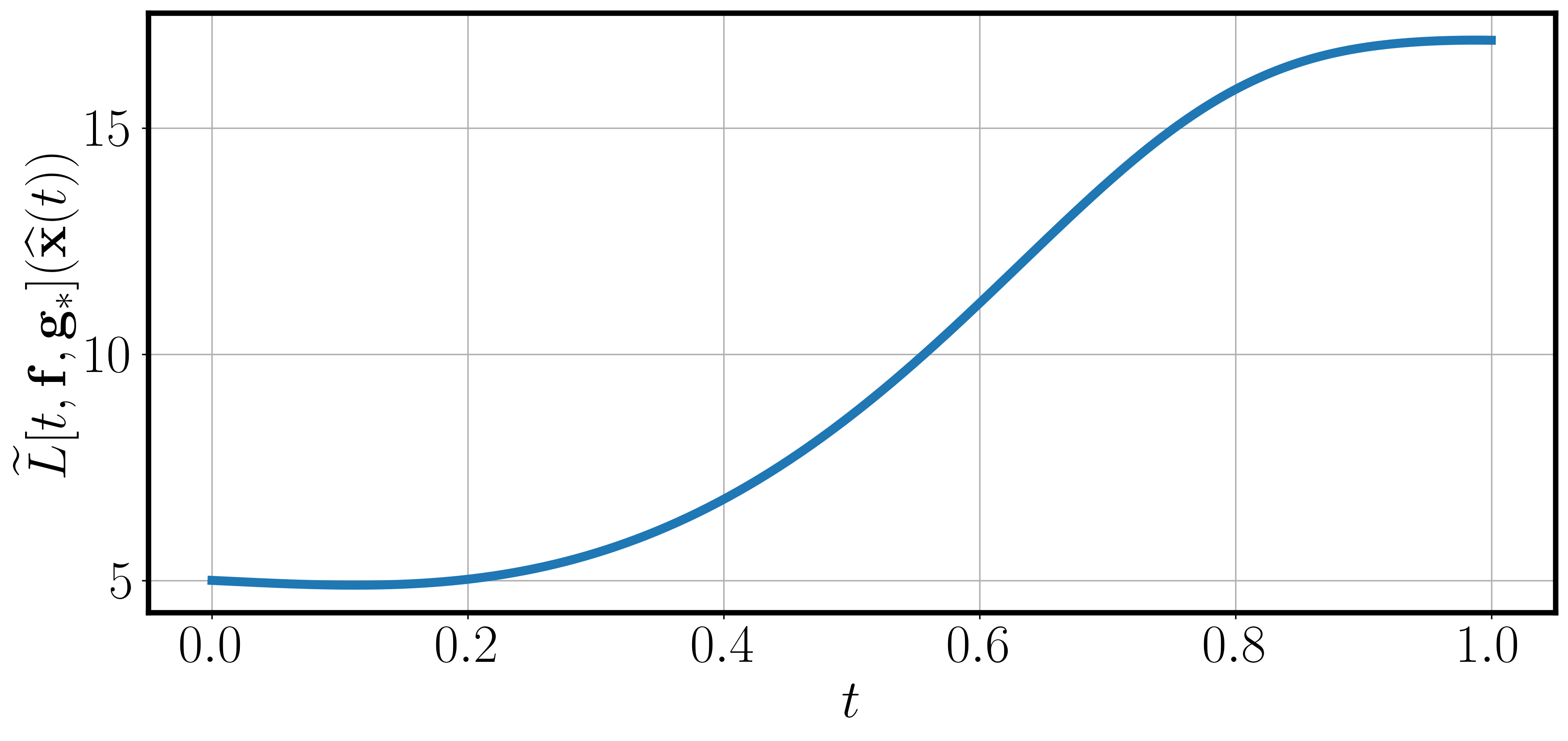}

\renewcommand{\thefigure}{SM.2}
\caption{Left: Sensitivity-based trajectory error estimate (blue) and trajectory error (orange) in good agreement, and trajectory error bound (green) yields a much tighter bound on the trajectory error than the Gronwall-type error bound (red).
              Right: Approximate logarithmic Lipschitz constant along perturbed trajectory for Zermelo ODE.}
              \label{fig:error-estimates-zermelo}
\end{figure}
The trajectory error (in the 2-norm) is also displayed for comparison. 
The sensitivity-based trajectory error bound $\| \delta \bx(t) \|$ from \cref{thm:state-error-bound}
yields a much tighter bound on the trajectory error $\| \widehat{\bx}(t) -  \bx_*(t) \|$  than the Gronwall-type error bound \eqref{eq:error-estimate-comparison-lemma}.
The reason for the pessimistic Gronwall-type error bound is that the approximate logarithmic Lipschitz constant evaluated along the trajectory, i.e.,
$\widetilde{L}[t, \bff, \bg_*]\big( \widehat{\bx}(t) \big)$ where $\widetilde{L}$ is as defined in \eqref{eq:logarithmic-norm-local-deriv} with respect to the 2-norm, 
is positive; see the right plot in \cref{fig:error-estimates-zermelo}.

Note that \cref{thm:state-error-bound} gives an upper bound on $\| \widehat{\bx} - \bx_* \|_{(L^2(I))^{n_x}}$, 
not a pointwise upper bound on $\| \widehat{\bx}(t) - \bx_*(t) \|$ for $t \in I$, so some care is needed in interpreting the results of \cref{fig:error-estimates-zermelo}. Still, it is useful to compare $\| \delta \bx(t) \|$ with $\| \widehat{\bx}(t) - \bx_*(t) \|$ pointwise to see how the worst-case perturbation of the ODE solution 
based on the model error bound \eqref{eq:g-error-vec} compares to the observed perturbation in the ODE solution. In this case, $\| \delta \bx(t) \|$ turns out to be a tight upper bound of $\| \widehat{\bx}(t) - \bx_*(t) \|$.

The left plot in \cref{fig:zermelo-epsilon-study} shows the effect of the perturbation parameter $\epsilon$ in \eqref{eq:delta-g-zermelo} on the $L^2$-error of the 
trajectory and the sensitivity-based estimate of the trajectory error, as well as the upper bound of \cref{thm:state-error-bound}.

These results show that the sensitivity-based estimate of the trajectory error is close to the actual trajectory error and the sensitivity-based upper bound is tight
for a wide range of perturbation parameters $\epsilon$.
Note that the bound is tight in this example because we set the bounds $\bepsilon$ in \eqref{eq:Refinement:ODE:LQOCP} equal to the absolute model error, as seen in \eqref{eq:model-error-bound-zermelo}.
Relaxing $\bepsilon$ would result in a looser bound.

\begin{figure}[!htb]
\includegraphics[width=0.49\textwidth]{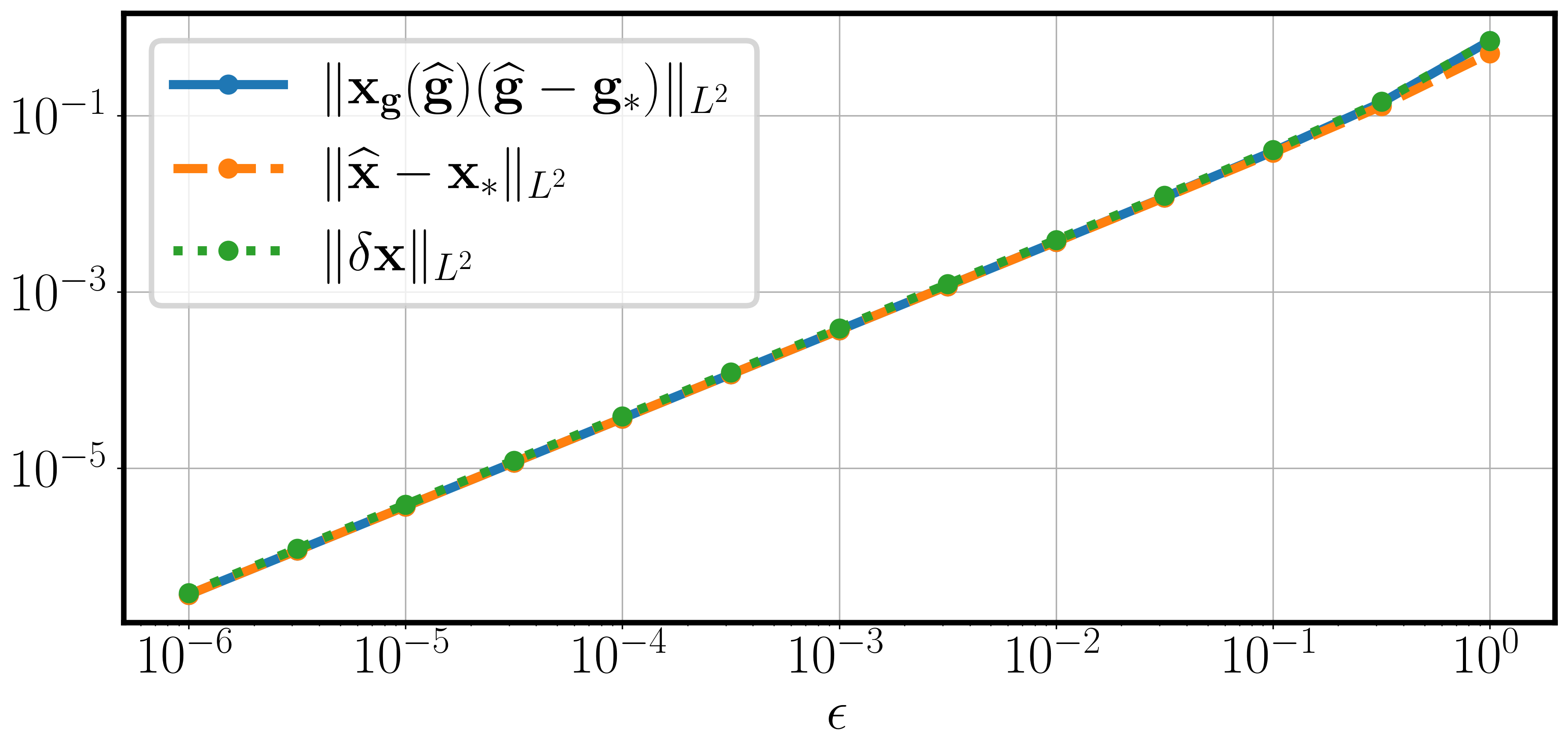} \hfill
\includegraphics[width=0.49\textwidth]{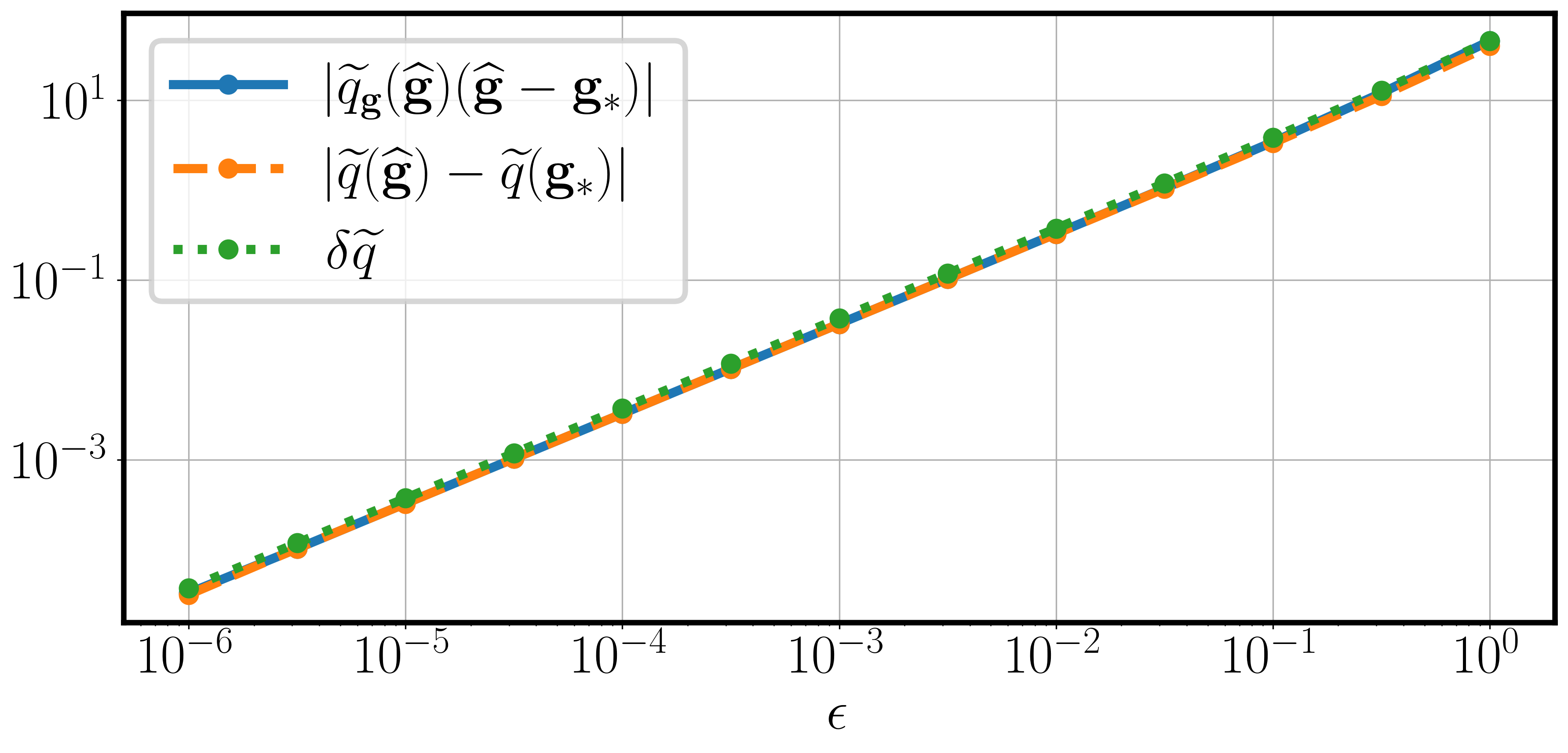}

\vspace*{-2ex}
\renewcommand{\thefigure}{SM.3}
\caption{Left: Strong agreement between the $L^2$ trajectory error estimates (blue), trajectory errors (orange), and sensitivity-based error bounds (green) for Zermelo ODE.
 Right: Strong agreement between the QoI error estimates (blue), QoI errors (orange), and error bounds (green) for Zermelo ODE.} 
 \label{fig:zermelo-epsilon-study}
\end{figure}

Recall that the LCQP obtained from discretization of \eqref{eq:Refinement:ODE:LQOCP} is a convex maximization problem, 
which is NP-hard.
However, \cref{fig:zermelo-epsilon-study} shows that when using our tailored interior point method, 
the computed upper bound is consistently very close to the sensitivity-based estimate, so the error bound based on
\eqref{eq:Refinement:ODE:LQOCP}  is still useful in practice.

Next, we consider the sensitivity of a quantity of interest.  We consider the total distance traveled,
\[
	\widetilde{q}(\bg) = \int_0^1 \sqrt{\bx_1'(t)^2 + \bx_2'(t)^2} \, dt,
\]
where $\bx_1', \bx_2'$ are as in \eqref{eq:numerics_zermelo:zermelo_ODE}.

The true QoI error $| \widetilde{q}(\widehat{\bg}) - \widetilde{q}(\bg_*) |$, the sensitivity-based estimate $| \widetilde{q}_\bg(\widehat{\bg}) (\widehat{\bg} - \bg_*) |$, 
and the sensitivity-based upper bound 
$\delta \widetilde{q} := \int_{t_0}^{t_f}  \big|  \widehat{\BB}(t)^T \widehat{\blambda}(t) + \nabla_g l \big( t, \widehat{\bx}(t), \widehat{\bg} \big( t, \widehat{\bx}(t) \big) \big)  \big|^T  \bepsilon\big( t, \widehat{\bx}(t) \big) \, dt$ of \cref{thm:QoI-error-bound} were computed for 
several values of the perturbation parameter $\epsilon$ in \eqref{eq:delta-g-zermelo} and are shown in the right plot in
\cref{fig:zermelo-epsilon-study}. All three quantities are in strong agreement for a wide range of perturbation parameters $\epsilon$.

\section{Hodgkin-Huxley Model} \label{sec:numerics_hh}
In this section, we present the numerical results for the Hodgkin-Huxley model describing the membrane excitability of the squid giant axon 
 \cite{ALHodgkin_AFHuxley_1952a}.
The states are the membrane voltage $V \textrm{ [mV]}$,
a displacement of the membrane potential $\textrm{[mV]}$  from its resting value of $V_{\rm rest}  = -65 \textrm{ [mV]}$,
and gating variables $n, m, h$, i.e.,
\[
	\bx(t) = \big( \BV(t), \bn(t), \bm(t), \bh(t) \big).
\]
The current equations for the one-compartment Hodgkin-Huxley model are
\begin{subequations} \label{eq:HH_model}
    \begin{align}
            C_M \BV'(t) &= -g_\textrm{ \!K} \bn(t)^4 \big(\BV(t) - E_\textrm{ \!K}\big) - g_{\textrm{ \!Na}} \bm(t)^3 \bh(t) \big(\BV(t) - E_{\textrm{ \!Na}} \big) \nonumber \\
            &\quad - g_\textrm{ \!L} \big(\BV(t) - E_\textrm{ \!L} \big) + \BI(t), & t \in (t_0, t_f), \\
        \bn'(t) &= \alpha_n\big(\BV(t)\big)\big(1-\bn(t)\big)-\beta_n\big(\BV(t)\big) \bn(t), & t \in (t_0, t_f), \label{eq:n_equation} \\
        \bm'(t) &= \alpha_m\big(\BV(t)\big)\big(1-\bm(t)\big)-\beta_m\big(\BV(t)\big) \bm(t), & t \in (t_0, t_f), \label{eq:m_equation} \\
        \bh'(t) &= \alpha_h\big(\BV(t)\big)\big(1-\bh(t)\big)-\beta_h\big(\BV(t)\big) \bh(t), & t \in (t_0, t_f), \label{eq:h_equation} \\
        \BV(t_0) &= V_{\rm rest}  = -65, \quad
        \bn(t_0) = 0.317, \quad
        \bm(t_0) = 0.052, \quad
        \bh(t_0) = 0.596,
    \end{align}
\end{subequations}
where the total membrane current $\BI(t)$ $[\mu$A/cm$^2$] is a given input, 
the model parameters are specified in \cref{tab:Table_parameters_HHmodel_def}, and the
functions $\alpha_y, \beta_y$, $y \in \{ n, m, h\}$ are given by
\begin{subequations}    \label{eq:HH_model-alpha-beta}
    \begin{align}
        \alpha_n(V)&=\frac{0.01\big( 10 - (V - V_{\rm rest}) \big)}{\exp\big( \big( 10 - (V - V_{\rm rest}) \big)/10\big) - 1}, &
         \beta_n(V)&=0.125\exp\big( -(V - V_{\rm rest})/80\big),\\
        \alpha_m(V)&=\frac{0.1\big(25-(V - V_{\rm rest})\big)}{\exp\big(\big(25-(V - V_{\rm rest})\big)/10 \big) - 1}, &
        \beta_m(V)&=4\exp\big( -(V - V_{\rm rest}) /18\big),\\
        \alpha_h(V)&=0.07\exp\big(- (V - V_{\rm rest})/20\big), &
        \beta_h(V)&= \frac{1}{\exp\big(\big(30 - (V - V_{\rm rest}) \big)/10\big) + 1 }.
     \end{align}
\end{subequations}
The initial data $\bn(t_0)$, $\bm(t_0)$, $\bh(t_0)$ in (\ref{eq:HH_model}e) are computed as steady-state solutions of equations
(\ref{eq:HH_model}b-d) with $V = V_{\rm rest} = -65  \textrm{ [mV]}$, i.e.,
\[
     y(t_0) = \frac{ \alpha_y(V_{\rm rest}) }{  \alpha_y(V_{\rm rest}) + \beta_y(V_{\rm rest}) }, \quad y \in \{ n, m, h \}.
\]
(The values of $\bn(t_0)$, $\bm(t_0)$, $\bh(t_0)$ are shown to three decimal places in  (\ref{eq:HH_model}e).)
We consider the time interval $(t_0, t_f) = (0, 50)$ [ms] and use a total membrane current $\BI(t) = 10$ $[\mu$A/cm$^2$].
\begin{table}[!htb]
        \centering
        \begin{tabular}{c|c|c}
                      & Parameter & Value \\   \hline
            Resting potential  &  $V_{\rm rest}   \textrm{ [mV]}$ & -65 \\ \hline
            Membrane Capacitance  & $C_M \textrm{ [$\mu$F/cm$^2$]}$ & 1 \\   \hline
            Sodium reversal potential & $E_\textrm{ \!Na} \textrm{ [mV]}$ &  50 \\    \hline
            Potassium reversal potential  & $E_\textrm{ \!K} \textrm{ [mV]}$ & -77 \\ \hline
            Leak reversal potential & $E_\textrm{ \!L} \textrm{ [mV]}$ & -54.387 \\  \hline
            Maximal sodium conductance  & $g_\textrm{ \!Na} \textrm{ [mS/cm$^2$]}$ & 120 \\ \hline
            Maximal potassium conductance & $g_\textrm{ \!K} \textrm{ [mS/cm$^2$]}$ & 36 \\  \hline
            Maximal leak conductance & $g_\textrm{ \!L} \textrm{ [mS/cm$^2$]}$ & 0.3 \\  
        \end{tabular}
        \renewcommand{\thetable}{SM.1}
        \caption{Parameter values used in model \eqref{eq:HH_model}.}
    \label{tab:Table_parameters_HHmodel_def}
\end{table}

We view the six voltage-dependent component functions \eqref{eq:HH_model-alpha-beta} with the parameter values specified
in \cref{tab:Table_parameters_HHmodel_def} as our true component function, i.e.,
\[
	\bg_*\big(t, \bx(t) \big) 
	:= \big[
	\alpha_n\big(\BV(t)\big),\; 
	\alpha_m\big(\BV(t)\big),\; 
	\alpha_h\big(\BV(t)\big),\; 
	\beta_n\big(\BV(t)\big),\; 
	\beta_m\big(\BV(t)\big),\; 
	\beta_h\big(\BV(t)\big) \big].
\]

We set
\[
	\delta \bg(t, x) 
	:= \big[
	0.5 \alpha_n(V),\; 
	-0.2 \alpha_m(V),\; 
	-0.3\alpha_h(V),\; 
	0.5\beta_n(V),\; 
	-0.3\beta_m(V),\; 
	-0.2\beta_h(V) \big]
\]
and consider 
\[
     \widehat{\bg}(t, x) :=  \bg_*(t, x) + \epsilon \delta \bg(t, x). 
\]
The impact of such perturbations were studied in \cite{HOri_EMarder_SMarom_2018a} by evaluating the
solution of \eqref{eq:HH_model} for randomly generated perturbations.

\cref{fig:neuron-trajectories} shows the solutions of \eqref{eq:HH_model} using $\bg_*$ and using $\widehat{\bg}$ with $\epsilon = 0.1$. 
The voltage spikes at slightly different frequencies, and the gating variables also exhibit periodicity.
\begin{figure}[!htb]
\centering
\includegraphics[width=0.45\textwidth]{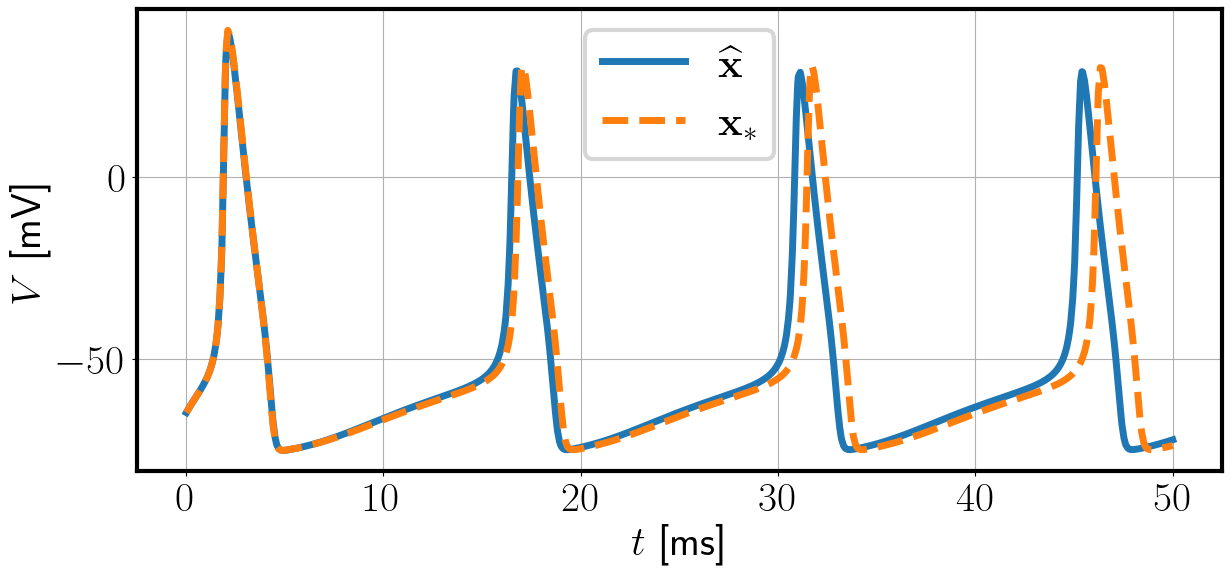}
\includegraphics[width=0.45\textwidth]{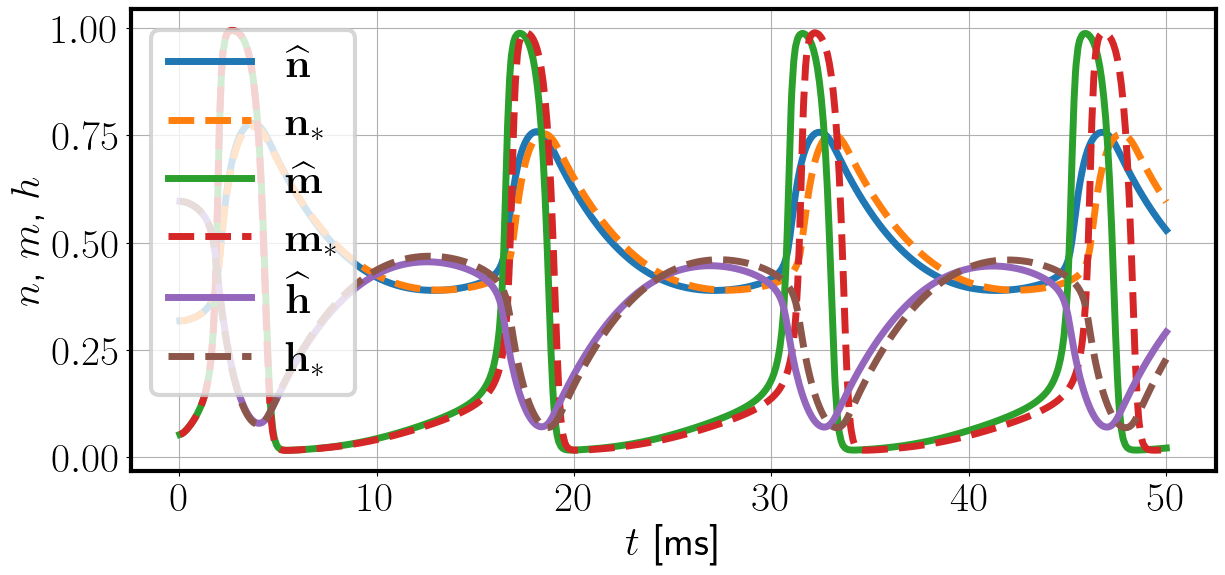}

\vspace*{-2ex}
\renewcommand{\thefigure}{SM.4}
\caption{Solutions to the true and perturbed Hodgkin-Huxley model for $\epsilon = 0.1$.} \label{fig:neuron-trajectories}
\end{figure}

The pointwise error bound is set to the absolute value of the perturbation (understood componentwise):
\[
	\bepsilon\big(t, \widehat{\bx}(t) \big) := \epsilon \big| \delta \bg\big(t, \widehat{\bx}(t) \big) \big|.
\]

The norm of the trajectory error and the worst-case trajectory error (in the $L^2$ sense, i.e, $\BQ(t)$ in \cref{thm:state-error-bound} equal to the identity matrix) obtained from the solution $(\delta \bx, \bdelta)$ of \eqref{eq:Refinement:ODE:LQOCP} 
as functions of $t$ are shown for $\epsilon = 0.1$ in \cref{fig:error-estimates-neuron}, along with the logarithmic Lipschitz constant 
$ \widetilde{L}[t,\bff,\bg_*](\widehat{\bx}(t))$ \eqref{eq:logarithmic-norm-local-deriv} 
computed along the trajectory $\widehat{\bx}$. 
\begin{figure}[!htb]
\includegraphics[width=0.49\textwidth]{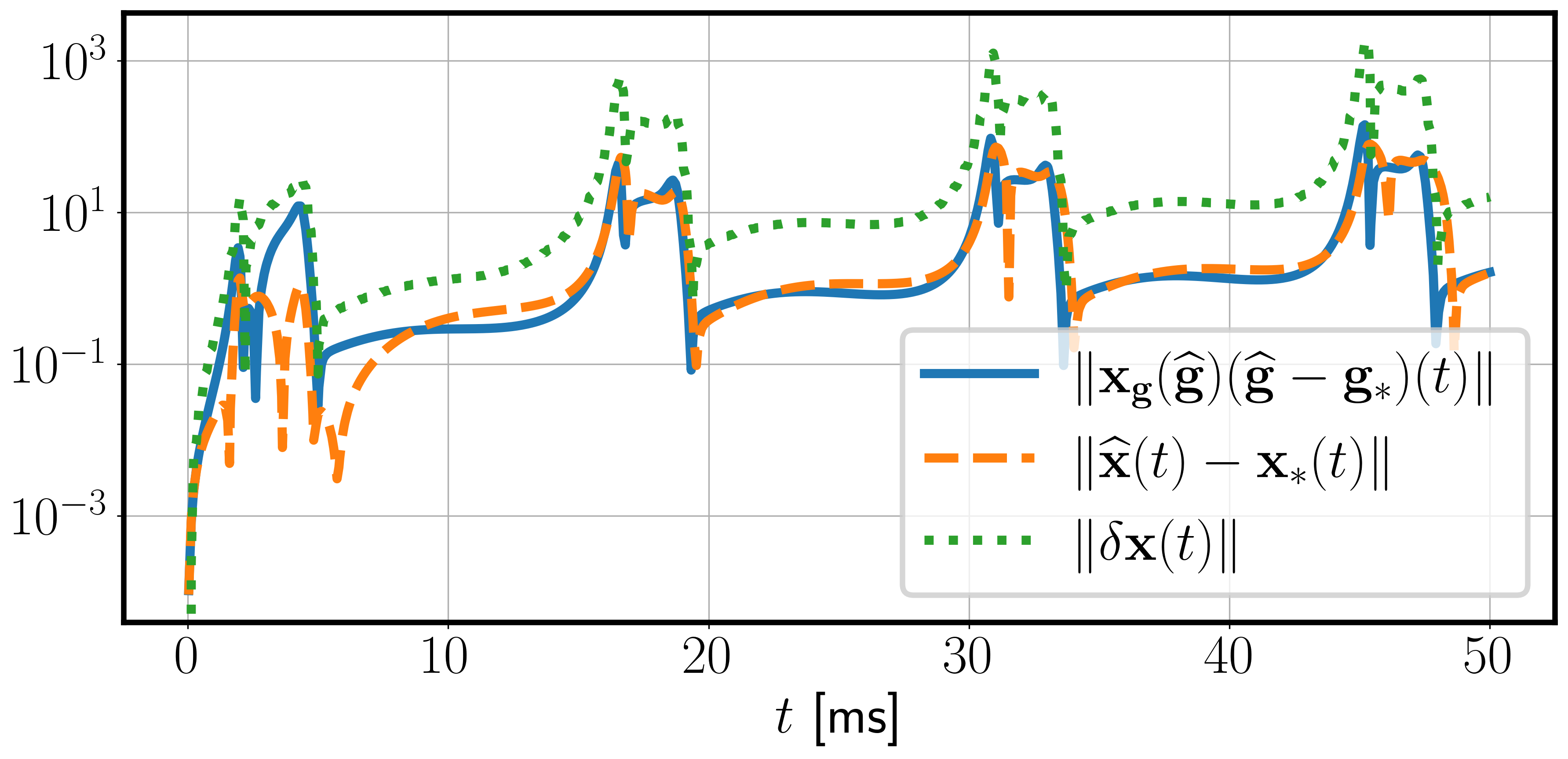}
\includegraphics[width=0.49\textwidth]{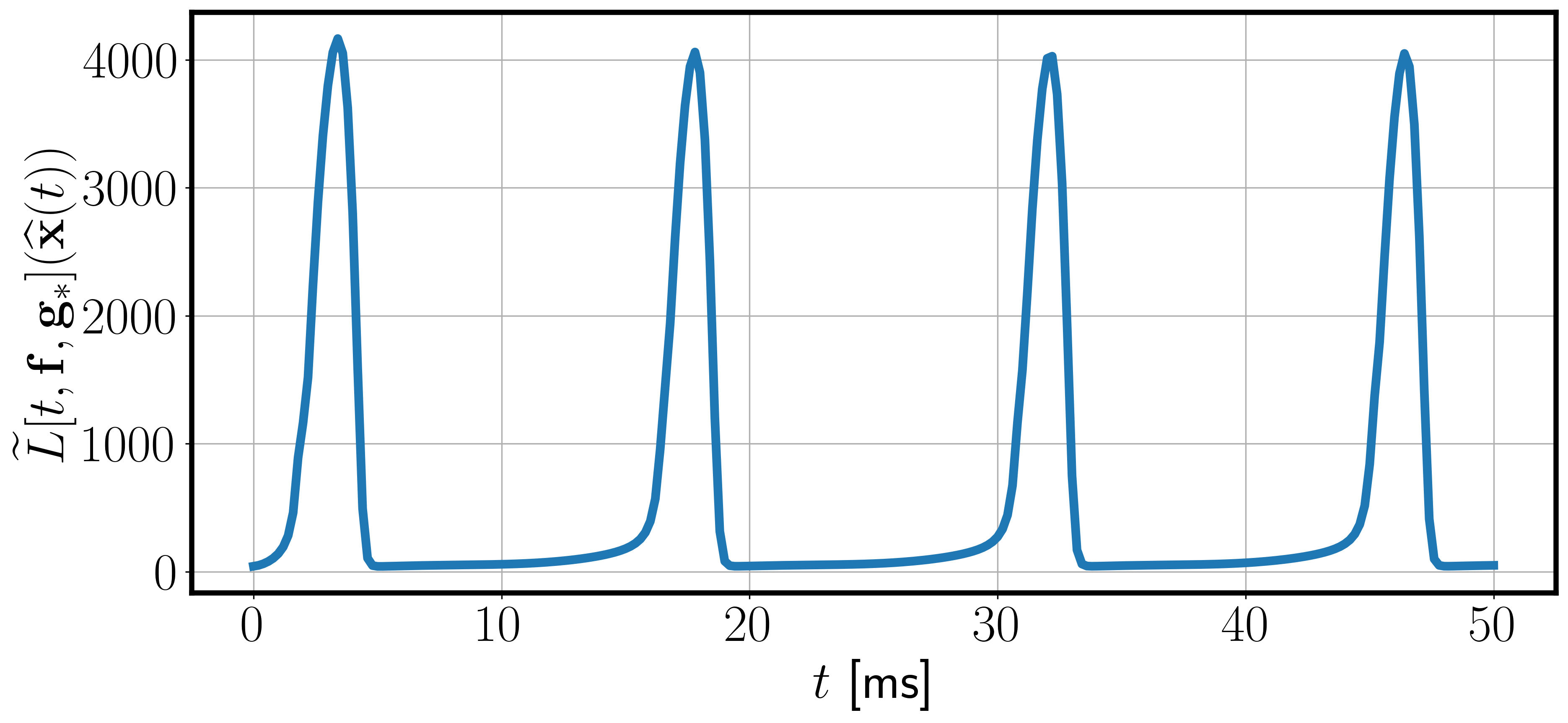}

\renewcommand{\thefigure}{SM.5}
\caption{Left: Sensitivity-based trajectory error estimate (blue) and trajectory error (orange) in good agreement, and trajectory error bound (green) yields a much tighter bound on the trajectory error than the Gronwall-type error bound (not shown due to size).
              Right: Approximate logarithmic Lipschitz constant along perturbed trajectory for Hodgkin-Huxley ODE.} 
              \label{fig:error-estimates-neuron}
\end{figure}
The logarithmic Lipschitz constant is large in time intervals that contain the voltage spikes, and as
a result, the error bound \eqref{eq:error-estimate-comparison-lemma} is even more pessimistic than in the
hypersonic vehicle trajectory simulation in \cref{sec:numerics_flap}. 
Therefore, the classical bound \eqref{eq:error-estimate-comparison-lemma} is not plotted in \cref{fig:error-estimates-neuron}.
In contrast, \cref{fig:error-estimates-neuron} shows that while our sensitivity-based error bound is less tight than it was in the previous examples, 
it still reflects the true error reasonably well, and
the sensitivity $\| \bx_\bg(\widehat{\bg})( \widehat{\bg} - \bg_*) \|$  is fairly close to the true error $\| \widehat{\bx} - \bx_* \|$.

It is observed that our sensitivity-based bound based on \eqref{eq:Refinement:ODE:LQOCP} is more conservative in this example than in the others.
This is explained by examining the optimal solution of \eqref{eq:Refinement:ODE:LQOCP}. For instance, the $\beta_m$-component of the  true model error $\widehat{\bg}\big(t, \widehat{\bx}(t)\big) - \bg_*\big(t, \widehat{\bx}(t)\big)$
(dotted green line) is equal to $- \bepsilon\big(t, \widehat{\bx}(t)\big)$ (dashed orange line), whereas the $\beta_m$-component of the solution $\bdelta$ of \eqref{eq:Refinement:ODE:LQOCP} 
(solid blue line) occasionally changes sign to maximize the objective in  \eqref{eq:Refinement:ODE:LQOCP}.
The $V$-component of the solution $\delta \bx$  of \eqref{eq:Refinement:ODE:LQOCP} is shown in \cref{fig:neuron-states}.

\begin{figure}[!htb]
\begin{center}
\includegraphics[width=0.60\textwidth]{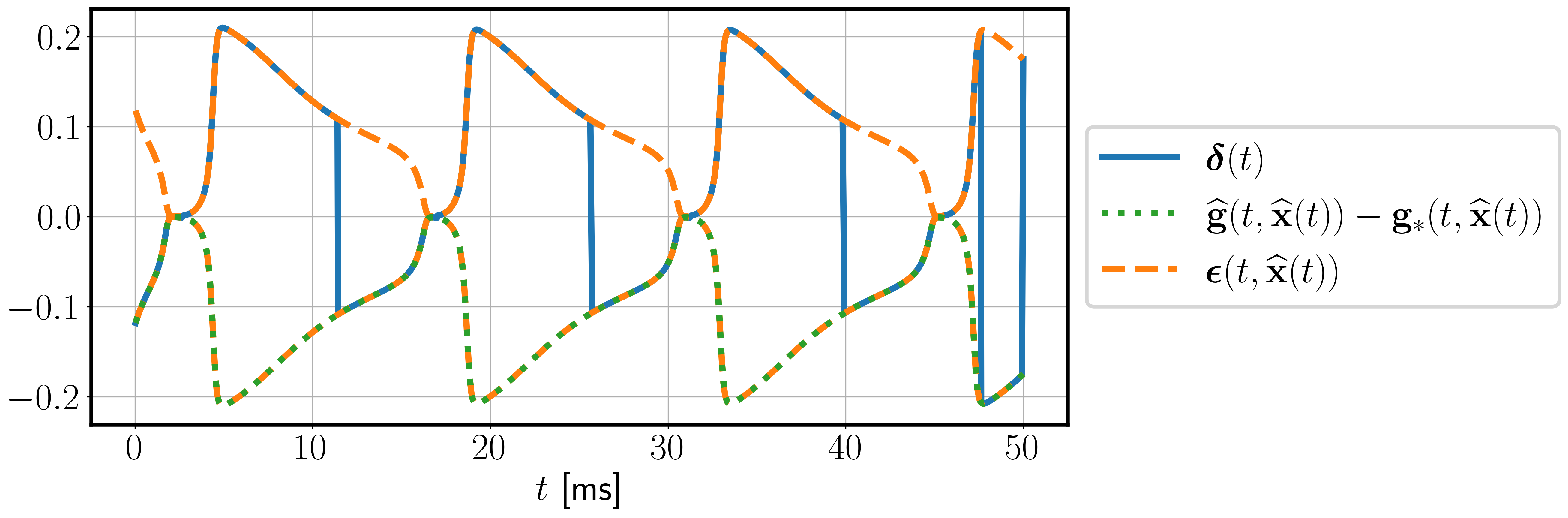}
\end{center}

\vspace*{-2ex}
\renewcommand{\thefigure}{SM.6}
\caption{Optimal solution of \eqref{eq:Refinement:ODE:LQOCP} for $\epsilon = 0.1$ exhibits several sign changes in $\bdelta$ near voltage spikes. Shown: component of $\bdelta$ corresponding to $\beta_m$ (blue), error of $\beta_m$ model along perturbed trajectory (green), and thresholds for box constraints \eqref{eq:ODE_sensitivity_sol:box-constraints} in $\beta_m$-component (orange).}
 \label{fig:delta-function}
\end{figure}

\begin{figure}[!htb]
\begin{center}
\includegraphics[width=0.49\textwidth]{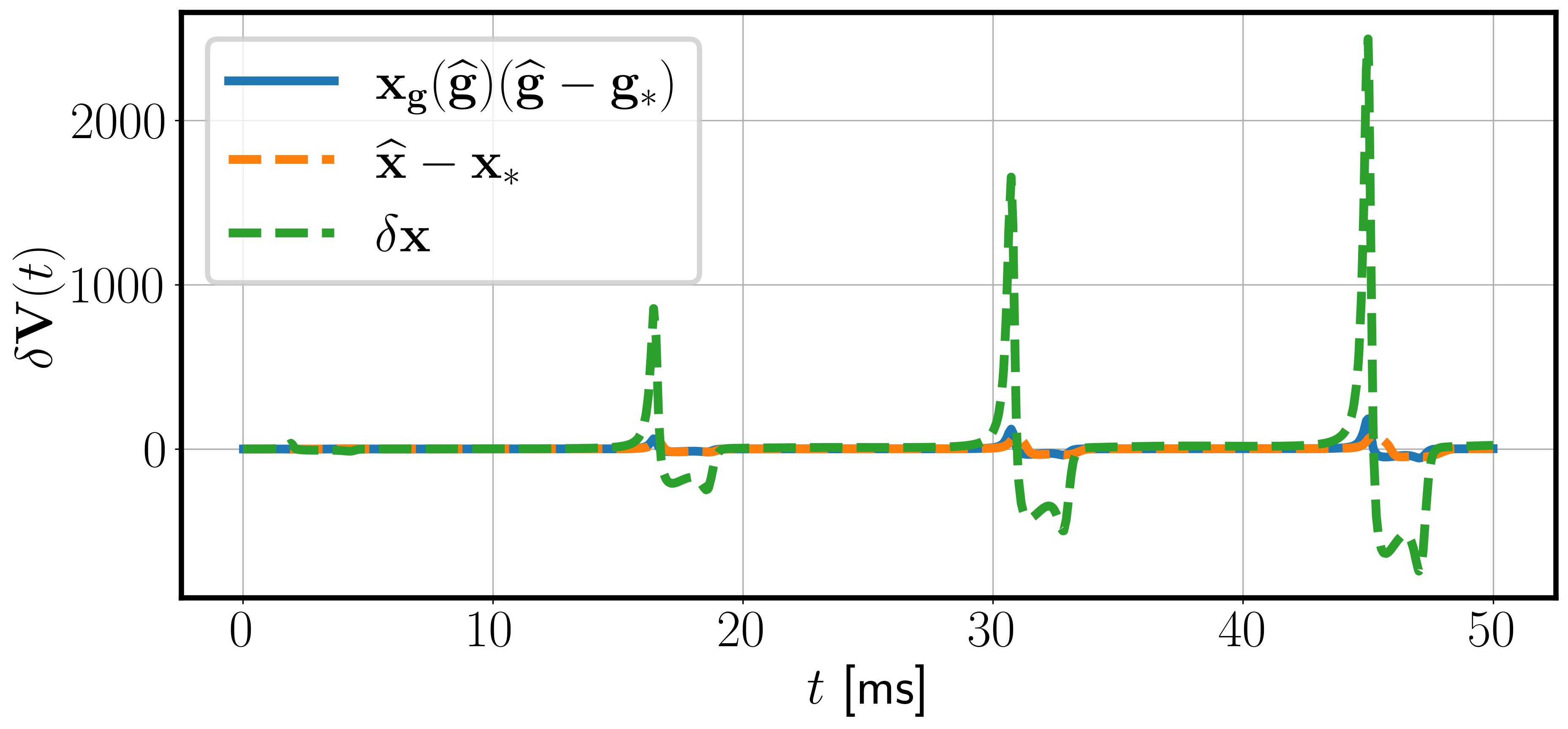} 
\end{center}

\vspace*{-2ex}
\renewcommand{\thefigure}{SM.7}
\caption{Optimal solution of \eqref{eq:Refinement:ODE:LQOCP} for $\epsilon = 0.1$ yields voltage spikes that get progressively larger, making the $L^2$ error bound (green) more conservative than the sensitivity-based estimate (blue) of the actual error (orange).}
 \label{fig:neuron-states}
\end{figure}

The left plot in \cref{fig:neuron-epsilon-study} shows the true solution error  $\| \widehat{\bx} - \bx_* \|_{L^2}$, the
sensitivity  $\| \bx_\bg(\widehat{\bg})( \widehat{\bg} - \bg_*) \|_{L^2}$, and our sensitivity-based error bound $\| \delta \bx \|_{L^2}$ computed from
\eqref{eq:Refinement:ODE:LQOCP} with $\BQ(t)$ as the identity matrix for a range of model perturbations $\epsilon$.
The true solution error and the sensitivity are in very good agreement for $\epsilon \le 0.5$, whereas our sensitivity-based error bound 
overestimates the true error by a constant factor.  Because our sensitivity-based error bound remains roughly proportional to the actual error, 
it is still useful in practice.

\begin{figure}[!htb]
\includegraphics[width=0.49\textwidth]{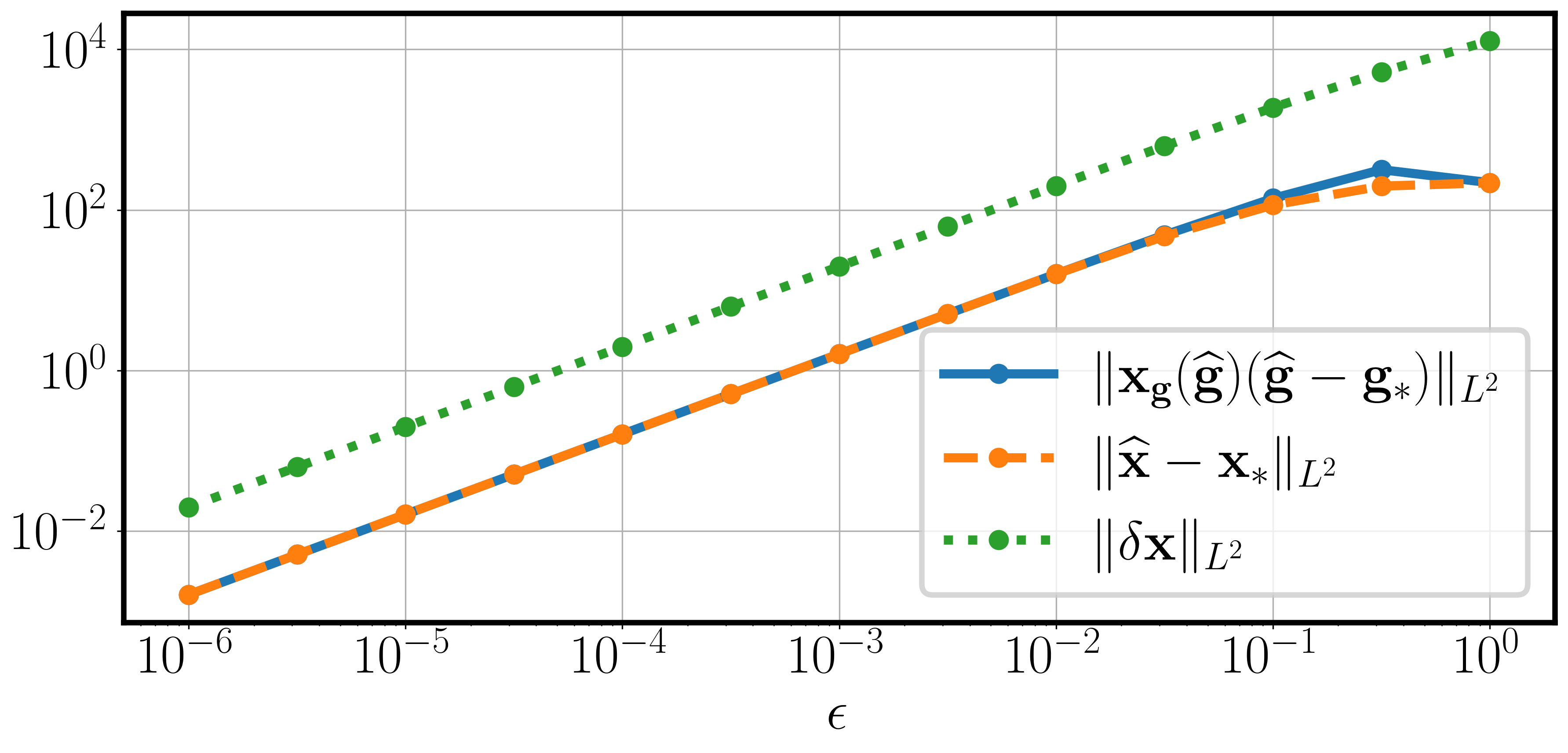} \hfill
\includegraphics[width=0.49\textwidth]{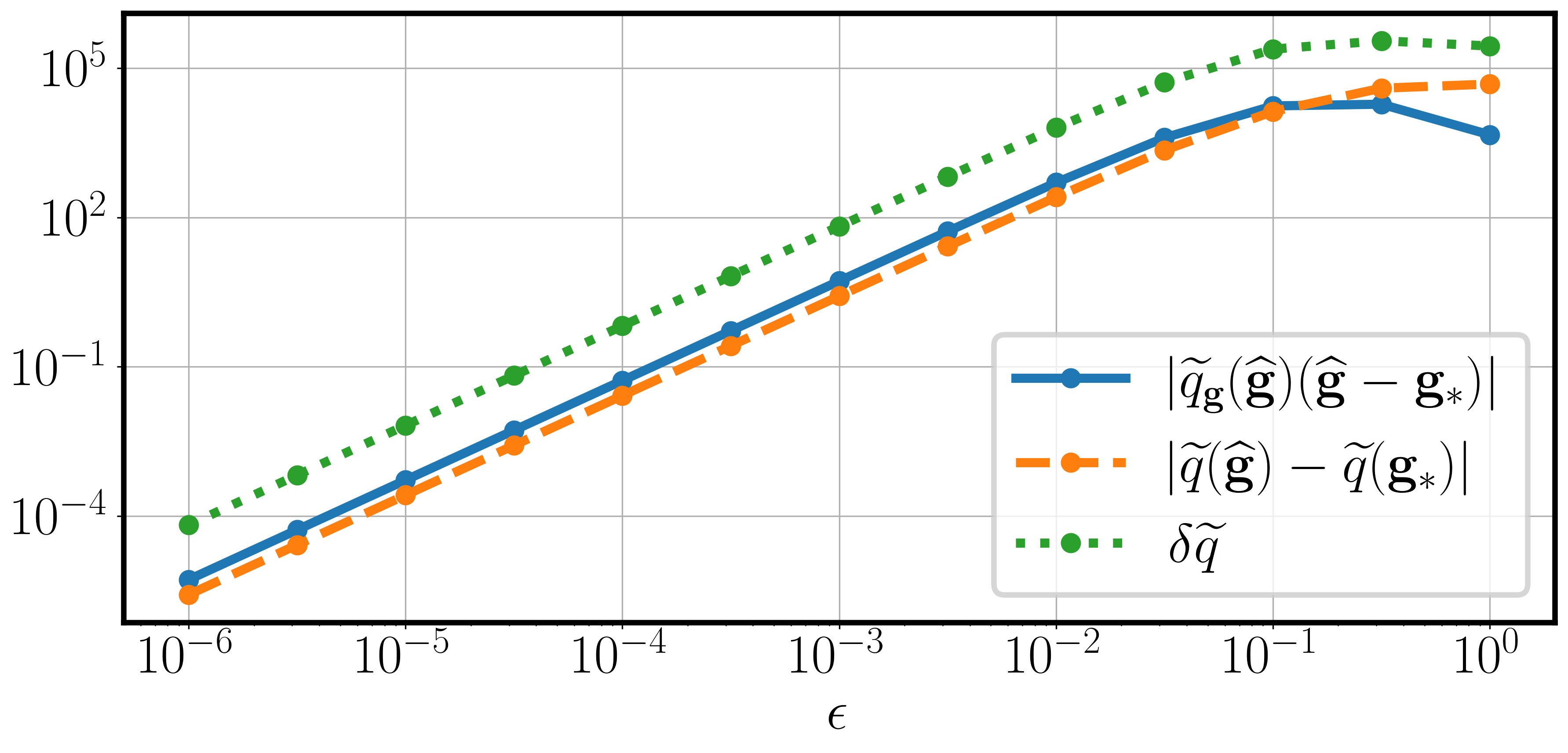}

\vspace*{-2ex}
\renewcommand{\thefigure}{SM.8}
\caption{Left: Strong agreement between the $L^2$ trajectory error estimates (blue), trajectory errors (orange), and sensitivity-based error bounds (green) up to a constant for Hodgkin-Huxley ODE.
 Right: Strong agreement between the QoI error estimates (blue), QoI errors (orange), and error bounds (green) up to a constant for Hodgkin-Huxley ODE. While somewhat conservative, the error bound is still roughly proportional to the error.}
 \label{fig:neuron-epsilon-study}
\end{figure}

Next, we study our sensitivity-based error bound  for the QoI
\begin{equation} \label{eq:HH_QoI}
	\widetilde{q}(\bg) = \int_{t_0}^{t_f} \big( \BV(t; \bg) - \BV_*(t) \big)^2 \, dt,
\end{equation}
where $\BV_*(t) := \BV(t; \bg_*)$ is the voltage component of the solution of the Hodgkin-Huxley model \eqref{eq:HH_model} with the
true component function. 
QoIs related to  \eqref{eq:HH_QoI} arise in objective functions for parameter estimation problems in the Hodgkin-Huxley model \eqref{eq:HH_model} and 
other neuronal models (see, e.g., \cite{ANogaret_CDMeliza_DMargoliash_HDIAbarbanel_2016a}, \cite{WVanGeit_EDeSchutter_PAchard_2008a}).  
For example, a synthetic experiment with true parameter $\bg_*$ and no measurement noise leads to an objective function
involving \eqref{eq:HH_QoI} and possibly a regularization term for $\bg$.
While  \eqref{eq:HH_QoI} is also one component of the error \eqref{eq:x+-err-size} which is 
estimated by \eqref{eq:Refinement:ODE:LQOCP}, the global solution of (the discretized version of) \eqref{eq:Refinement:ODE:LQOCP}
is difficult. On the other hand, our error bound \eqref{eq:Refinement:ODE:QoI-LP} for the QoI \eqref{eq:HH_QoI} can be easily computed analytically, although doing so requires explicit knowledge of the true solution.

The right plot in \cref{fig:neuron-epsilon-study} shows the true QoI error, the sensitivity, and our sensitivity-based error bound
$\delta \widetilde{q} := \int_{t_0}^{t_f}  \big|  \widehat{\BB}(t)^T \widehat{\blambda}(t) + \nabla_g l \big( t, \widehat{\bx}(t), \widehat{\bg} \big( t, \widehat{\bx}(t) \big) \big)  \big|^T  \bepsilon\big( t, \widehat{\bx}(t) \big) \, dt$ of \cref{thm:QoI-error-bound} for several values of the perturbation parameter $\epsilon$.
The true error (dashed orange line) and the sensitivity (solid line) differ by a factor of approximately two for $\epsilon \le 0.1$. 
Our sensitivity-based error bound (dotted green line) overestimates the sensitivity by a constant factor.
Again, this is due to the fact that our sensitivity-based error bound uses the model error bound $ \bepsilon\big(t, \widehat{\bx}(t)\big)$ instead of the true model error 
$\widehat{\bg}\big(t, \widehat{\bx}(t)\big) - \bg_*\big(t, \widehat{\bx}(t)\big)$.  However, because our sensitivity-based error bound remains roughly proportional to the actual error, 
it is still useful in practice.

The reason why the true QoI error and the sensitivity differ by a factor of approximately two for small $\epsilon$ is that
\[
	\widetilde{q}_\bg(\widehat{\bg}) (\widehat{\bg} - \bg_*) 
	= \int_{t_0}^{t_f} \frac{d}{dV} \big(V - \BV_*(t)\big)^2 \big|_{V = \BV(t; \widehat{\bg})} \delta \BV(t; \widehat{\bg}) \, dt 
	= \int_{t_0}^{t_f} 2 \big(\BV(t; \widehat{\bg}) - \BV_*(t)\big) \delta \BV(t; \widehat{\bg}) \, dt,
\]
where $\delta \BV(\cdot \, ; \widehat{\bg}) := \BV_\bg(\cdot \, ; \widehat{\bg}) (\widehat{\bg} - \bg_*) $ is the sensitivity of the voltage function.
If $\widehat{\bg}$ is close to $\bg_*$, then $\delta \BV(t; \widehat{\bg}) \approx \BV(t; \widehat{\bg}) - \BV_*(t)$, implying that
\[
	\widetilde{q}_\bg(\widehat{\bg}) (\widehat{\bg} - \bg_*) \approx 2 \int_{t_0}^{t_f} \big(\BV(t; \widehat{\bg}) - \BV_*(t)\big)^2 \, dt = 2 \, \widetilde{q}(\widehat{\bg}) = 2 \big(\widetilde{q}(\widehat{\bg}) - \widetilde{q}(\bg_*)\big),
\]
where the latter equality follows from the fact that $\widetilde{q}(\bg_*) = 0$. \ However, when these sensitivities are used, e.g.,
for surrogate model refinement, it is often sufficient in practice for the sensitivity to be roughly proportional to the actual error, which is the case here.



\end{document}